\documentclass[a4paper, 12pt]{amsart}
\usepackage{amssymb}
\usepackage{dsfont}
\usepackage{mathrsfs} 
\usepackage[all,cmtip]{xy}
\usepackage{patchcmd}
\input{diagxy}
\usepackage{url}
\usepackage{bm}
\usepackage{mathtools}
\usepackage{stmaryrd}
\usepackage{enumerate}
\usepackage[usenames,dvipsnames,svgnames,table]{xcolor}
\usepackage[top=2.4cm, bottom=2.4cm, left=2cm, right=2cm]{geometry}

\makeatletter
\patchcommand\@starttoc{\begin{quote}}{\end{quote}}
\def\@tocline#1#2#3#4#5#6#7{\relax
  \ifnum #1>\c@tocdepth 
  \else
    \par \addpenalty\@secpenalty\addvspace{#2}%
    \begingroup \hyphenpenalty\@M
    \@ifempty{#4}{%
      \@tempdima\csname r@tocindent\number#1\endcsname\relax
    }{%
      \@tempdima#4\relax
    }%
    \parindent\z@ \leftskip#3\relax \advance\leftskip\@tempdima\relax
    \rightskip\@pnumwidth plus4em \parfillskip-\@pnumwidth
    #5\leavevmode\hskip-\@tempdima
      \ifcase #1
       \or\or \hskip 1em \or \hskip 2em \else \hskip 3em \fi%
      #6\nobreak\relax
    \dotfill\hbox to\@pnumwidth{\@tocpagenum{#7}}\par
    \nobreak
    \endgroup
  \fi}
\makeatother

 \theoremstyle{plain}
 \newtheorem{thm}{Theorem}[section]
 \newtheorem{cor}[thm]{Corollary}
 \newtheorem{lem}[thm]{Lemma}
 
 \newtheorem{prop}[thm]{Proposition}

\theoremstyle{definition}
 \newtheorem{defn}[thm]{Definition}
 
 \newtheorem{hyp}[thm]{Hypothesis}

\theoremstyle{remark}
 \newtheorem{rem}[thm]{Remark}
 \newtheorem{ques}[thm]{Question}

 \newtheorem{nota}[thm]{Notation}
 \newtheorem{conv}[thm]{Convention}
 \newtheorem{exam}[thm]{Example}

 \numberwithin{equation}{section}

\theoremstyle{plain}

\DeclareMathOperator{\VF}{VF}

\DeclareMathOperator{\RV}{RV}
\DeclareMathOperator{\DC}{DC}
\DeclareMathOperator{\MM}{\gM}

\DeclareMathOperator{\OO}{\gO}

 \DeclareMathOperator{\ran}{ran}
 \DeclareMathOperator{\dom}{dom}

 \DeclareMathOperator{\id}{id}

 \DeclareMathOperator{\lh}{lh}

 \DeclareMathOperator{\aut}{Aut}

 \DeclareMathOperator{\supp}{supp}

 \DeclareMathOperator{\dcl}{dcl}
 \DeclareMathOperator{\pr}{pr}

 \DeclareMathOperator{\mgl}{GL}

\DeclareMathOperator{\jcb}{Jcb}

\DeclareMathOperator{\K}{\Bbbk}

\DeclareMathOperator{\res}{res}  

\def\Xint#1{\mathchoice
{\XXint\displaystyle\textstyle{#1}}%
{\XXint\textstyle\scriptstyle{#1}}%
{\XXint\scriptstyle\scriptscriptstyle{#1}}%
{\XXint\scriptscriptstyle\scriptscriptstyle{#1}}%
\!\int}
\def\XXint#1#2#3{{\setbox0=\hbox{$#1{#2#3}{\int}$}
\vcenter{\hbox{$#2#3$}}\kern-.5\wd0}}

\newcommand{\Z}{\mathds{Z}}

\newcommand{\KKK}{\mathds{K}}
\newcommand{\Q}{\mathds{Q}}
\newcommand{\N}{\mathds{N}}

\newcommand{\R}{\mathds{R}}

\newcommand{\omin}{$o$\nobreakdash}

\newcommand{\cmin}{$C$\nobreakdash}

\newcommand{\T}{$T$\nobreakdash}

\newcommand{\gB}{\mathfrak{B}}

\newcommand{\gF}{\mathfrak{F}}

\newcommand{\gM}{\mathfrak{M}}

\newcommand{\gO}{\mathfrak{O}}

\newcommand{\ga}{\mathfrak{a}}
\newcommand{\gb}{\mathfrak{b}}
\newcommand{\gc}{\mathfrak{c}}
\newcommand{\gd}{\mathfrak{d}}

\newcommand{\gi}{\mathfrak{i}}

\newcommand{\gm}{\mathfrak{m}}
\newcommand{\gn}{\mathfrak{n}}
\newcommand{\go}{\mathfrak{o}}
\newcommand{\gp}{\mathfrak{p}}
\newcommand{\gq}{\mathfrak{q}}

\newcommand{\0}{\emptyset}

\DeclareMathAlphabet{\mathpzc}{OT1}{pzc}{m}{it}

\newcommand{\dpar}[1]{
   (\mkern-4mu (#1 )\mkern-4mu )}

 \DeclarePairedDelimiter\abs{\lvert}{\rvert}

 \newcommand{\set}[1]{\left\{#1\right\}}

 \newcommand{\usub}[2]{#1_{\textup{#2}}}

 \newcommand{\lan}[3]{\mathcal{L}_{#1 \textup{#2} #3}}

\newcommand{\mdl}[1]{\mathcal{#1}}  
\newcommand{\bb}[1]{\mathbb{#1}}

\newcommand{\ex}[1]{\exists #1 \;} 

\newcommand{\rest}{\upharpoonright}

\newcommand{\fun}{\longrightarrow}
\newcommand{\efun}{\longmapsto}
\newcommand{\sub}{\subseteq}

\newcommand{\mi}{\smallsetminus}

\newcommand{\colim}[1]{\underset{#1}{\text{colim}}\,}  

\newcommand{\goedel}[1]{\ulcorner #1 \urcorner}

\newcommand{\la}{\langle}
\newcommand{\ra}{\rangle}

\DeclareMathOperator{\rv}{rv}

\DeclareMathOperator{\vv}{val}

\DeclareMathOperator{\RES}{RES}

\DeclareMathOperator{\gsk}{\mathbf{K}_+}
\DeclareMathOperator{\ggk}{\mathbf{K}}

\DeclareMathOperator{\fn}{FN}
\DeclareMathOperator{\fib}{fib}

\DeclareMathOperator{\isp}{I_{sp}}

\DeclareMathOperator{\rad}{rad}

\DeclareMathOperator{\vrv}{vrv}

\DeclareMathOperator{\RVH}{RVH}

\DeclareMathOperator{\can}{\mathbf{c}}

\DeclareMathOperator{\ito}{int}
\DeclareMathOperator{\cl}{cl}

\DeclareMathOperator{\tor}{Tor}
\DeclareMathOperator{\aff}{Aff}
\DeclareMathOperator{\reg}{reg}

\newbox\gnBoxA
\newdimen\gnCornerHgt
\setbox\gnBoxA=\hbox{$\ulcorner$}
\global\gnCornerHgt=\ht\gnBoxA
\newdimen\gnArgHgt

\def\code #1{%
        \setbox\gnBoxA=\hbox{$#1$}%
        \gnArgHgt=\ht\gnBoxA%
        \ifnum \gnArgHgt<\gnCornerHgt
                \gnArgHgt=0pt%
        \else
                \advance \gnArgHgt by -\gnCornerHgt%
        \fi
        \raise\gnArgHgt\hbox{$\ulcorner$} \box\gnBoxA %
                \raise\gnArgHgt\hbox{$\urcorner$}}

\newcommand{\ol}[1]{\overline{#1}}

\DeclareMathOperator{\TCVF}{TCVF}

\newcommand{\dand}{\quad \text{and} \quad}
\newcommand{\LT}{$\lan{T}{}{}$\nobreakdash}

\newcommand{\ddx}{\tfrac{d}{d x}}

\DeclareMathOperator{\sgn}{sgn}

\DeclareMathOperator{\abval}{\abs{\vv}}
\DeclareMathOperator{\abvrv}{\abs{\vrv}}
\DeclareMathOperator{\RVV}{RV_0}
\DeclareMathOperator{\GAA}{\Gamma_0}
\DeclareMathOperator{\absG}{\abs{\Gamma}}
\DeclareMathOperator{\mmdl}{\mdl R_{\rv}}
\DeclareMathOperator{\xmdl}{\mdl R_{\rv}^{\bullet}}

\author[Yimu Yin]{Yimu Yin}
\address{Department of Philosophy \\ Sun Yat-Sen University \\ 135 Xingang Road West \\
Guangzhou, China \\ 510275 }
\email{yimu.yin@hotmail.com}
\title[Euler characteristic in $\TCVF$]{Generalized Euler characteristic in power-bounded \T-convex valued fields}

\begin{document}

\begin{abstract}
  We lay the groundwork in this first installment of a series of papers aimed at developing a theory of Hrushovski-Kazhdan style motivic integration for certain type of non-archimedean \omin-minimal fields, namely power-bounded $T$-convex valued fields, and closely related structures. The main result of the present paper is a canonical homomorphism between the Grothendieck semirings of certain categories of definable sets that are associated with the $\VF$-sort and the $\RV$-sort of the language $\lan{T}{RV}{}$. Many aspects of this homomorphism can be described explicitly. Since these categories do not carry volume forms, the formal groupification of the said homomorphism is understood as a universal additive invariant or a generalized Euler characteristic. It admits, not just one, but two specializations to $\Z$. The overall structure of the construction is modeled on that of the original Hrushovski-Kazhdan construction.
\end{abstract}

\thanks{The research leading to the true claims in this paper has been partially supported by the ERC Advanced Grant NMNAG, the grant ANR-15-CE40-0008 (D\'efig\'eo), the SYSU grant 11300-18821101, and the NSSFC Grant 14ZDB015.}

\maketitle

\tableofcontents

\vskip 15mm

\section{Introduction}\label{intro}

Towards the end of the introduction of \cite{hrushovski:kazhdan:integration:vf} three hopes for the future of the theory of motivic integration are mentioned. We propose to investigate one of them in a series of papers: additive invariants and integration in \omin-minimal valued fields. A prototype of such valued fields is $\R \dpar{ t^{\Q}}$, the generalized power series field over $\R$ with exponents in $\Q$. One of the cornerstones of the methodology of \cite{hrushovski:kazhdan:integration:vf} is \cmin-minimality, which is the right analogue of \omin-minimality for algebraically closed valued fields and other closely related structures that epitomizes the behavior of definable subsets of the affine line. It, of course, fails in an \omin-minimal valued field, mainly due to the presence of a total ordering. Thus the construction we seek has to be carried out in a different framework, which affords a similar type of normal forms for definable subsets of the affine line, a special kind of weak \omin-minimality; this framework is van den Dries and Lewenberg's theory of $T$-convex valued fields \cite{DriesLew95, Dries:tcon:97}.

The reader is referred to the opening discussions in \cite{DriesLew95, Dries:tcon:97} for a more detailed introduction to \T-convexity and a summary of fundamental results. In those papers, how the valuation is expressed is somewhat inconsequential. In contrast, we shall work exclusively with a fixed two-sorted language $\lan{T}{RV}{}$ --- see \S~\ref{defn:lan} and Example~\ref{exam:RtQ} for a quick grasp of the central features of this language --- since such a language is a part of the preliminary setup of any Hrushovski-Kazhdan style integration.

Throughout this paper, let $T$ be a complete power-bounded \omin-minimal \LT-theory extending the theory $\usub{\textup{RCF}}{}$ of real closed fields. For the real field $\R$, the condition of being power-bounded is the same as that of being polynomially bounded. However, for nonarchimedean real closed fields, the former condition is more general and is indeed more natural.

The language $\lan{T}{}{}$ extends the language $\{<, 0, 1, +, -, \times\}$ of ordered rings. Let $\mdl R \coloneqq (R, <, \ldots)$ be a model of $T$. By definition, a \T-convex subring $\OO$ of $\mdl R$ is a convex subring of $\mdl R$ such that, for every definable (no parameters allowed) continuous function $f : R \fun R$, we have $f(\OO) \sub \OO$. The convexity of $\OO$ implies that it is a valuation ring of $\mdl R$. For instance, if $\mdl R$ is nonarchimedean and $\R \sub R$ then the convex hull of $\R$ forms a valuation ring of $\mdl R$ and, accordingly, the infinitesimals form its maximal ideal. Such a convex hull is \T-convex if no definable continuous function can grow so fast as to stretch the standard real numbers into infinity.

Let $\OO$ be a \emph{proper} \T-convex subring of $\mdl R$. The theory $T_{\textup{convex}}$ of the pair $(\mdl R, \OO)$, suitably axiomatized in the language $\lan{}{convex}{}$ that expands $\lan{T}{}{}$ with a new unary relation symbol, is complete, and if $T$ admits quantifier elimination and is universally axiomatizable then $T_{\textup{convex}}$ admits quantifier elimination as well.

Since $T$ is power-bounded, the definable subsets of $R$ afford a type of normal form, a special kind of weak \omin-minimality (see \cite{mac:mar:ste:weako}), which we dub Holly normal form (since it was first studied by Holly in \cite{holly:can:1995}); in a nutshell, every definable subset of $R$ is a boolean combination of intervals and (valuative) discs. Clearly this is a natural generalization of \omin-minimality in the presence of valuation. A number of desirable properties of definable sets in $R$ depends on the existence of such a normal form. For instance, every subset of $R$ defined by a principal type assumes one of the following four forms: a point, an open disc, a closed disc, and a half thin annulus, and, furthermore, these four forms are distinct in the sense that no definable bijection between any two of them is possible.

Let $\vv : R^{\times} \fun \Gamma$ be the valuation map induced by $\OO$, $\K$ the corresponding residue field, and $\res : \OO \fun \K$ the residue map. There is a canonical way of turning $\K$ into a model of $T$ as well, see \cite[Remark~2.16]{DriesLew95}. Let $\MM$ be the maximal ideal of $\OO$. Let $\RV = R^{\times} / (1 + \MM)$ and $\rv : R^{\times} \fun \RV$ be the quotient map. Note that, for each $a \in R$, the map $\vv$ is constant on the set $a + a\MM$, and hence there is an induced map $\vrv : \RV \fun \Gamma$.
The situation is illustrated in the following commutative diagram
\begin{equation*}
\bfig
 \square(0,0)/^{ (}->`->>`->>`^{ (}->/<600, 400>[\OO \mi \MM`R^{\times}`\K^{\times}`
\RV;`\res`\rv`]
 \morphism(600,0)/->>/<600,0>[\RV`\Gamma;\vrv]
 \morphism(600,400)/->>/<600,-400>[R^{\times}`\Gamma;\vv]
\efig
\end{equation*}
where the bottom sequence is exact.
This structure may be expressed and axiomatized in a natural two-sorted first-order language $\lan{T}{RV}{}$, in which $R$ is referred to as the $\VF$-sort and $\RV$ is taken as a new sort. Informally, $\lan{T}{RV}{}$ is viewed as an extension of $\lan{}{convex}{}$.

We expand $(\mdl R, \OO)$ to an $\lan{T}{RV}{}$-structure. The main construction in this paper is carried out in such a setting. For concreteness, the reader is welcome to take $R = \R \dpar{ t^{\Q} }$ and $\OO = \R \llbracket t^{\Q} \rrbracket$ in the remainder of this introduction (see Example~\ref{exam:RtQ} below for more on this generalized power series field).

For a description of the ideas and the main results of the Hrushovski-Kazhdan style integration theory, we refer the reader to the original introduction in \cite{hrushovski:kazhdan:integration:vf} and also the introductions in \cite{Yin:int:acvf, Yin:int:expan:acvf}. There is also a quite comprehensive introduction to the same materials in \cite{hru:loe:lef} and, more importantly, a specialized version that relates the Hrushovski-Kazhdan style integration to the geometry and topology of Milnor fibers over the complex field. The method expounded there will be featured in a sequel to this paper as well. In fact, since much of the work below is closely modeled on that in \cite{hrushovski:kazhdan:integration:vf, Yin:special:trans, Yin:int:acvf, hru:loe:lef}, the reader may simply substitute the term ``theory of power-bounded $T$-convex valued fields'' for ``theory of algebraically closed valued fields'' or more generally ``$V$-minimal theories'' in those introductions and thereby acquire a quite good grip on what the results of this paper look like. For the reader's convenience, however, we shall repeat some of the key points, perhaps with minor changes here and there.

Let $\VF_*$ and $\RV[*]$ be two categories of definable sets that are respectively associated with the $\VF$-sort and the $\RV$-sort as follows. In $\VF_*$, the objects are the definable subsets of cartesian products of the form $\VF^n \times \RV^m$ and the morphisms are the definable bijections. On the other hand, for technical reasons (particularly for keeping track of ambient dimensions), $\RV[*]$ is formulated in a somewhat complicated way and is hence equipped with a gradation by ambient dimensions (see Definition~\ref{defn:c:RV:cat}).

The Grothendieck semigroup of a category $\mdl C$, denoted by $\gsk \mdl C$, is the free semigroup generated by the isomorphism classes of $\mdl C$, subject to the usual scissor relation $[A \mi B] + [B] = [A]$, where $[A]$, $[B]$ denote the isomorphism classes of the objects $A$, $B$ and ``$\mi$'' is certain binary operation, usually just set subtraction. Sometimes $\mdl C$ is also equipped with a binary operation --- for example, cartesian product --- that induces multiplication in $\gsk \mdl C$, in which case $\gsk \mdl C$ becomes a (commutative) semiring. The formal groupification of  $\gsk \mdl C$, which is then a ring, is denoted by $\ggk \mdl C$.

The main construction of the Hrushovski-Kazhdan integration theory is a canonical --- that is, functorial in a suitable way --- homomorphism from the Grothendieck semiring $\gsk \VF_*$ of $\VF_*$ to the Grothendieck semiring $\gsk \RV[*]$ of $\RV[*]$ modulo a semiring congruence relation $\isp$ on the latter. In fact, it turns out to be an isomorphism. This construction has three main steps.
\begin{enumerate}[{Step} 1.]
 \item First we define a lifting map $\bb L$ from the set of objects of $\RV[*]$ into the set of objects of $\VF_*$ (Definition~\ref{def:L}). Next we single out a subclass of isomorphisms in $\VF_*$, which are called special bijections (Definition~\ref{defn:special:bijection}), and show that for any object $A$ in $\VF_*$ there is a special bijection $T$ on $A$ and an object $\bm U$ in $\RV[*]$ such that $T(A)$ is isomorphic to $\bb L \bm U$ (Corollary~\ref{all:subsets:rvproduct}). This implies that $\bb L$
is ``essentially surjective'' on objects, meaning that it is surjective on isomorphism classes of $\VF_*$. For this result alone we do not have to limit our means to special
bijections. However, in Step~3 below, special bijections become an essential ingredient in computing the semiring congruence
relation $\isp$.

 \item We show that, for any two isomorphic objects $\bm U_1$, $\bm U_2$ of $\RV[*]$, their lifts $\bb L \bm U_1, \bb L \bm U_2$ in $\VF_*$ are isomorphic as well (Corollary~\ref{RV:lift}). This implies that $\bb L$ induces a semiring homomorphism
  $
  \gsk \RV[*] \fun \gsk \VF_*,
   $
   which is also denoted by $\bb L$. This homomorphism is surjective by Step~1  and hence, modulo  the semiring congruence relation $\isp$ --- that is, the kernel of $\bb L$ --- the inversion $\int_+$ of the homomorphism $\bb L$ is an isomorphism of semirings.

 \item A number of important properties of the classical integration can already be verified for $\int_+$ and hence, morally, this third step is not necessary. For applications, however, it is much more satisfying to have a precise description of the semiring congruence relation $\isp$. The basic notion used in the description is that of a blowup of an object in $\RV[*]$, which is essentially a restatement of the trivial fact that there is an additive translation from $1 + \MM$ onto $\MM$ (Definition~\ref{defn:blowup:coa}). We then show that, for any two objects $\bm U_1$, $\bm U_2$ in $\RV[*]$, there are isomorphic blowups $\bm U_1^{\flat}$, $\bm U_2^{\flat}$ of them if and only if $\bb L \bm U_1$, $\bb L \bm U_2$ are isomorphic (Proposition~\ref{kernel:L}). The ``if'' direction essentially contains a form of Fubini's theorem and is the most technically involved part of the construction.
\end{enumerate}
We call the semiring homomorphism $\int_+$ thus obtained a Grothendieck homomorphism. If the objects  carry volume forms and the Jacobian transformation preserves the integral, that is, the change of variables formula holds, then it may be called a motivic integration; we will not consider this case here and postpone it to a future installment. When the semirings are formally groupified, this Grothendieck homomorphism is accordingly recast as a ring homomorphism, which is denoted by $\int$ and is understood as a (universal) additive invariant.

The structure of the Grothendieck ring $\ggk \RV[*]$ may be significantly elucidated. To wit, it can be expressed as a tensor product of two other Grothendieck rings $\ggk \RES[*]$ and $\ggk \Gamma[*]$, that is, there is an isomorphism of graded rings:
\[
\bb D: \ggk \RES[*] \otimes_{\ggk \Gamma^{c}[*]} \ggk \Gamma[*] \fun \ggk \RV[*],
\]
where $\RES[*]$ is essentially the category of definable sets in $\R$ (as a model of the theory $T$) and $\Gamma[*]$ is essentially the category of definable sets over $\Q$ (as an \omin-minimal group), both are graded by ambient dimension, and $\Gamma^{c}[*]$ is the full subcategory of $\Gamma[*]$ of finite objects,  whose Grothendieck ring admits a natural embedding into  $\ggk \RES[*]$ as well. This isomorphism results in various retractions from $\ggk \RV[*]$ into $\ggk \RES[*]$ or $\ggk \Gamma[*]$ and, when combined with the Grothendieck homomorphism $\int$ and the two Euler characteristics in \omin-minimal groups (one is a truncated version of the other), yield a (generalized) Euler characteristic
\[
\textstyle \Xint{\textup{G}}  :  \ggk \VF_* \to^{\sim} ( \Z \oplus \bigoplus_{i \geq 1} (\Z[Y]/(Y^2+Y))X^i) / (1 + 2YX + X),
\]
which is actually an isomorphism, and two specializations to $\Z$:
\[
\textstyle \Xint{\textup{R}}^g, \Xint{\textup{R}}^b: \ggk \VF_* \fun \Z,
\]
determined by the assignments $Y \efun -1$ and $Y \efun 0$ or, equivalently, $X \efun 1$ and $X \efun -1$ (see Proposition~\ref{prop:eu:retr:k} and Theorem~\ref{thm:ring}). We will demonstrate the significance of these two specializations, as opposed to only one, in a future paper  that is dedicated to the study of generalized (real) Milnor fibers in the sense of \cite{hru:loe:lef}.

For certain purposes, the difference between model theory and algebraic geometry is somewhat easier to bridge if one works over the complex field, as is demonstrated in \cite{hru:loe:lef}; however, over the real field, although they do overlap significantly, the two worlds seem to diverge in their methods and ideas. Our results should be understood in the context of ``\omin-minimal geometry'' \cite{dries:1998, DrMi96} as opposed to real algebraic geometry. In general, the various Grothendieck rings considered in real algebraic geometry bring about lesser collapse of ``algebraic data'' --- since there are much less morphisms in the background --- and can yield much finer invariants, and hence are more faithful to the geometry in this regard, although the flip side of the story is that the invariants are often computationally intractable (especially when resolution of singularities is involved) and specializations are often needed in practice. For instance, the Grothendieck ring of real algebraic varieties may be specialized to $\Z[X]$, which is called the virtual Poincar\'e polynomial (see \cite{mccrory:paru:virtual:poin}). Our method here does not seem to be suited for recovering invariants at this level, at least not directly.

The role of $T$-convexity in this paper cannot be overemphasized. However, it does not quite work if the exponential function is included in the theory $T$. It remains a worthy challenge to find a suitable framework in which the construction of this paper may be extended to that case.

Much of the content of this paper is extracted from the preprint \cite{Yin:int:tcvf}, which contains a more comprehensive study of \T-convex valued fields. This auxiliary part of the theory we are developing may be regarded as  a sequel to or a variation on the themes of the work in \cite{DriesLew95, Dries:tcon:97}. It has become clear that some of the technicalities thereof may be of independent interest. For instance, the valuative or infinitesimal version of Lipschitz continuity plays a crucial role in proving the existence of Lipschitz stratifications in an arbitrary power-bounded \omin-minimal field (this proof has been published in \cite{halyin} and the result cited there is Corollary~\ref{part:rv:cons}).

Also, in a future paper, we will use the main result here to show that, in both the real and the complex cases,  the Euler characteristic of the topological Milnor fiber coincides with that of the motivic Milnor fiber, avoiding the algebro-geometric machinery employed in \cite[Remark~8.5.5]{hru:loe:lef}.

\section{Basic results in $T$-convex valued fields}

In this section, we first describe the two-sorted language $\lan{T}{RV}{}$ for \omin-minimal valued fields and the $\lan{T}{RV}{}$-theory $\TCVF$. This theory is axiomatized. Then we show that $\TCVF$ admits quantifier elimination. Some of the results in \cite{DriesLew95, Dries:tcon:97} that are crucial for our construction are also translated into the present setting.

\subsection{Some notation}\label{subs:nota}
Recall from the introduction above that $T$ is a complete power-bounded \omin-minimal \LT-theory extending the theory $\usub{\textup{RCF}}{}$ of real closed fields.

\begin{conv}
For the moment, by \emph{definable} we mean definable with arbitrary parameters from the structure in question. But later --- starting in \S~\ref{def:VF} --- we will abandon this practice and work with a fixed set of parameters. The reason for this change will be made abundantly clear when it happens.
\end{conv}

\begin{defn}[Power-bounded]\label{defn.powBd}
Suppose that $\mdl R$ is an \omin-minimal real closed field. A \emph{power function} in $\mdl R$ is a definable endomorphism of the multiplicative group $\mdl R^+$.
We say that $\mdl R$ is \emph{power-bounded} if every definable function $f \colon \mdl R \fun \mdl R$ is eventually dominated by a power function, that is, there exists a power function $g$ such that $\abs{f(x)} \le g(x)$ for all sufficiently large $x$. A complete \omin-minimal theory extending  $\usub{\textup{RCF}}{}$ is \emph{power-bounded} if all its models are.
\end{defn}

All power functions in  $\mdl R$ may be understood as functions of the form $x \efun x^\lambda$, where $\lambda = \ddx f(1)$. The collection of all such $\lambda$ form a subfield and is called the \emph{field of exponents} of $\mdl R$. We will quote the results on power-bounded structures directly from \cite{DriesLew95, Dries:tcon:97} and hence do not need to know more about them other than the things that have already been said. At any rate, a concise and lucid account of the essentials may be found in \cite[\S~ 3]{Dries:tcon:97}.

\begin{rem}[Functional language]\label{rem:cont}
We shall need a generality that is due to Lou van den Dries (private communication). It states that the theory $T$ can be reformulated in another language \emph{all} of whose primitives, except the binary relation $\leq$, are function symbols that are interpreted as \emph{continuous} functions in all the models of $T$. Actually, for this to hold, we only need to assume that $T$ is a complete \omin-minimal theory that extends $\usub{\textup{RCF}}{}$.

More precisely, working in any model of $T$, it can be shown that all definable sets are boolean combinations of sets of the form $f(x) = 0$ or $g(x) > 0$, where $f$ and $g$ are definable total continuous functions. In particular, this holds in the prime model $\mdl P$ of $T$. Taking all definable total continuous functions in $\mdl P$ and the ordering $<$ as the primitives in a new language $\lan{T'}{}{}$, we see that $T$ can be reformulated as an \emph{equivalent} $\lan{T'}{}{}$-theory $T'$ in the sense that the syntactic categories of $T$ and $T'$ are naturally equivalent. In traditional but less general and more verbose model-theoretic jargon, this just says that if a model of $T$ is converted to a model of $T'$ in the obvious way then the two models are bi\"{i}nterpretable via the identity map, and vice versa.

The theory $T'$ also admits quantifier elimination, but it cannot be universally axiomatizable in $\lan{T'}{}{}$. To see this, suppose for contradiction that it can be. Then, by the argument in the proof of \cite[Corollary~2.15]{DMM94}, every definable function $f$ in a model of $T'$, in particular, multiplicative inverse, is given piecewise by terms. But all terms define total continuous functions. This means that, by \omin-minimality, multiplicative inverse near $0$ is given by two total continuous functions, which is absurd.

Now, we may and do extend $T'$ by definitions so that it is universally axiomatizable in the resulting language. Thus every substructure of a model of $T'$ is actually a model of $T'$ and, as such, is an elementary substructure. In fact, since $T'$ has definable Skolem functions, we shall abuse notation slightly and redefine $T$ to be $T'^{\textup{df}}$, where $T'^{\textup{df}}$ is in effect a Skolemization of $T'$ (see \cite[\S\S~2.3--2.4]{DriesLew95} for further explanation). Note that the language of $T$ contains additional function symbols only and some of them must be interpreted in all models of $T$ as discontinuous functions for the reason given above.

To summarize, the main point is that $T$ admits quantifier elimination, is universally axiomatizable, is a definitional extension of $T'$, and all the primitives of $\lan{T'}{}{}$, except $\leq$, define continuous functions in all the models of $T'$.

The syntactical maneuver of passing through $T'$ just described will only be used in Theorem~\ref{thm:complete} below, and it is not really necessary if one works with a concrete \omin-minimal extension of $\usub{\textup{RCF}}{}$ such as $\usub{T}{an}$ defined in \cite{DMM94} (also see Example~\ref{exam:RtQ}).
\end{rem}

We shall work with a sufficiently saturated model $\mdl R \coloneqq (R, <, \ldots)$ of $T$ unless suggested otherwise. Its field of exponents is denoted by $\KKK$.

\begin{nota}[Coordinate projections]\label{indexing}
For each $n \in \N$, let $[n]$ abbreviate the set $\{1, \ldots, n\}$. For any $E \sub [n]$, we write $\pr_E(A)$ for the projection of $A$ into the coordinates contained in $E$. In practice, it is often more convenient to use simple standard descriptions as subscripts. For example, if $E$ is a singleton $\{i\}$ then we shall always write $E$ as $i$ and $\tilde E \coloneqq [n] \mi E$ as $\tilde i$; similarly, if $E = [i]$, $\{k:  i \leq k \leq j\}$, $\{k:  i < k < j\}$, $\{\text{all the coordinates in the sort $S$}\}$, etc., then we may write $\pr_{\leq i}$, $\pr_{[i, j]}$, $\pr_{(i, j)}$, $\pr_{S}$, etc.; in particular, $A_{\VF}$ and $A_{\RV}$ stand for the projections of $A$ into the $\VF$-sort and $\RV$-sort coordinates, respectively.

Unless otherwise specified, by writing $a \in A$ we shall mean that $a$ is a finite tuple of elements (or ``points'') of $A$, whose length, denoted by $\lh(a)$, is not always indicated. If $a = (a_1, \ldots, a_n)$ then, for all $1 \leq i < j \leq n$, following the notational scheme above, $a_i$, $a_{\tilde i}$, $a_{\leq i}$, $a_{[i, j]}$, $a_{[i, j)}$, etc., are shorthand for the corresponding subtuples of $a$. We shall write $\{t\} \times A$, $\{t\} \cup A$, $A \mi \{t\}$, etc., simply as $t \times A$, $t \cup A$, $A \mi t$, etc., when no confusion can arise.

For $a \in \pr_{\tilde E} (A)$, the fiber $\{b : ( b, a) \in A \} \sub \pr_E(A)$ over $a$ is denoted by $A_a$. Note that, in the discussion below, the distinction between the two sets $A_a$ and $A_a \times a$ is usually immaterial and hence they may and often shall be tacitly identified. In particular, given a function $f : A \fun B$ and $b \in B$, the pullback $f^{-1}(b)$ is sometimes written as $A_b$ as well. This is a special case since functions are identified with their graphs. This notational scheme is especially useful when the function $f$ has been clearly understood in the context and hence there is no need to spell it out all the time.
\end{nota}

\begin{nota}[Subsets and substructures]\label{nota:sub}
By a definable set we mean a definable subset in $\mdl R$, and by a subset in $\mdl R$  we  mean a subset in $R$, by which we mean a subset of $R^n$ for some $n$, unless indicated otherwise. Similarly for other structures or sets in place of $\mdl R$ that have been clearly understood in the context.

Often the ambient total ordering in $\mdl R$ induces a total ordering on a definable set $S$ of interest with a distinguished element $e$. Then it makes sense to speak of the positive and the negative parts of $S$ relative to $e$, which are denoted by $S^+$ and $S^-$, respectively. Also write $S^+_e$ for $S^+ \cup e$, etc. There may also be a natural absolute value map $S \fun S^+_e$, which is always denoted by $| \cdot |$; typically $S$ is a sort and $S^\times \coloneqq S \mi e$ is equipped with a (multiplicatively written) group structure, in which case the absolute value map is usually given as a component of a (splitting) short exact sequence
\[
\pm 1 \fun S^\times \fun S^+ \quad \text{or} \quad S^+ \fun S^\times \fun \pm 1.
\]
Note that $e$ cannot be the identity element of $S^\times$. We will also write $A < e$ to mean that $A \sub S$ and $a < e$ for all $a \in A$, etc. If $\phi(x)$ is a formula then $\phi(x) < e$ denotes the subset of $S$ defined by the formula $\phi(x) \wedge x < e$.

Substructures of $\mdl R$ are written as $\mdl S \sub \mdl R$. As has been pointed out above, all substructures $\mdl S$ of $\mdl R$ are actually elementary substructures. If $A \sub \mdl R^n$ is a set definable with parameters coming from $\mdl S$ then $A(\mdl S)$ is the subset in $\mdl S$ defined by the same formula, that is, $A(\mdl S) = A \cap \mdl S^n$. Given a substructure $\mdl S \sub \mdl R$ and a set $A \sub \mdl R$, the substructure generated by $A$ over $\mdl S$ is denoted by $\la \mdl S , A \ra$ or $\mdl S \la A \ra$. Clearly $\la \mdl S , A \ra$ is the definable closure of $A$ over $\mdl S$. Later, we will expand $\mdl R$ and introduce more sorts and structures. In that situation we will write $\mdl S \la A \ra_T$ or $\la \mdl S , A \ra_T$ to emphasize that this is the \LT-substructure generated by $A$ over the \LT-reduct of $\mdl S$.
\end{nota}

\begin{nota}[Topology]
The default topology on  $\mdl R$ is of course the order topology and the default topology on $\mdl R^n$ is the corresponding product topology. Given a subset $S$ in $\mdl R$, we write $\cl(S)$ for its topological closure, $\ito(S)$ for its interior,
and $\partial S \coloneqq \cl(S) \setminus S$ for its frontier (not to be confused with the boundary $\cl(S) \setminus \ito(S)$ of $S$, which is also sometimes denoted by $\partial S$). The same topological discourse applies to a definable set if the ambient total ordering of $\mdl R$ induces a total ordering on it.
\end{nota}

\subsection{The theory $\TCVF$}\label{defn:lan}

The language $\lan{T}{RV}{}$ for \omin-minimal valued fields --- the theory $T$ may vary, of course --- has the following sorts and symbols:
\begin{itemize}
 \item A sort $\VF$, which uses the language $\lan{T}{}{}$.
 \item A sort $\RV$, whose basic language is that of groups, written multiplicatively as $\{1, \times, {^{-1}} \}$, together with a constant symbol $0_{\RV}$ (for notational ease, henceforth this will be written simply as $0$).
 \item A unary predicate $\K^{\times}$ in the $\RV$-sort. The union $\K^{\times} \cup \{0\}$ is denoted by $\K$, which is more conveniently thought of as a sort and, as such, employs the language $\lan{T}{}{}$ as well, where the constant symbols $0$, $1$ are shared with the $\RV$-sort.
 \item A binary relation symbol $\leq$ in the $\RV$-sort.
 \item A function symbol $\rv : \VF \fun \RV_0$.
\end{itemize}
We shall write $\RV$ to mean the $\RV$-sort without the element $0$, and $\RV_0$ otherwise, etc., although quite often the difference is immaterial.

\begin{defn}\label{defn:tcf}
The axioms of the theory $\usub{\textup{TCVF}}{}$ of \emph{$T$-convex valued fields} in the language $\lan{T}{RV}{}$ are presented here informally. Many of them are clearly redundant as axioms, and we try to phrase some of these in such a way as to indicate so. The list also contains additional notation that will be used throughout the paper.
\begin{enumerate}[({Ax.} 1)]
 \item The \LT-reduct of the $\VF$-sort is a model of $T$.

 Recall from Notation~\ref{nota:sub} that $\VF^+ \sub \VF$ is the subset of positive elements and $\VF^- \sub \VF$ the subset of negative elements.

 \item \label{ax:rv} The quadruple $(\RV, 1, \times, {^{-1}})$ forms an abelian group. Inversion  is augmented by $0^{-1} = 0$. Multiplication is augmented by $t \times 0 = 0 \times t = 0$ for all $t \in \RV$. The map $\rv : \VF^{\times} \fun \RV$ is a surjective group homomorphism augmented by $\rv(0) = 0$.

  \item The binary relation $\leq$ is a total ordering on $\RV_{0}$ such that, for all $t, t' \in \RV_{0}$, $t < t'$ if and only if $\rv^{-1}(t) < \rv^{-1}(t')$.

      The distinguished element $0 \in \RV_0$ is more aptly referred to as the \emph{middle element} of $\RV_{0}$. Clearly $\RV^+ = \rv(\VF^+)$  and $\RV^- = \rv(\VF^-)$ (see Notation~\ref{nota:sub}). It follows from (Ax.~\ref{ax:rv}) that $\RV^+$ is an ordered convex subgroup of $\RV$ and the quotient group $\RV / \RV^+$ is isomorphic to the group $\pm 1 \coloneqq \rv(\pm 1)$. This gives rise to an absolute value map on $\RV_{0}$, which is compatible with the absolute value map on $\VF$ in the sense that $\rv(\abs{a}) = \abs{\rv(a)}$ for all $a \in \VF$.

   \item \label{ax:K} The set $\K^{\times}$ forms a \emph{nontrivial} subgroup of $\RV$ and the set $\K^{+} = \K^{\times} \cap \RV^+$ forms a convex subgroup of $\RV^+$.

       The quotient groups $\RV / \K^+$, $\RV^{+} / \K^+$ are denoted by $\Gamma$, $\Gamma^+$ and the corresponding quotient maps by $\vrv$, $\vrv^+$. Also set $\vrv(0) = 0 \in \Gamma_0$. Since $\K^+$ is convex, $\Gamma^+$ is an ordered group, where the induced ordering is also denoted by $\leq$, and the absolute value map on $\RV_0$ descends to $\Gamma_0$ in the obvious sense.

  \item Let $\leq^{-1}$ be the ordering on $\Gamma^+_0$ inverse to $\leq$ and
       $\absG_{\infty} \coloneqq (\Gamma^+_0, +, \leq^{-1})$  the resulting \emph{additively} written ordered abelian group with the top element $\infty$. The composition
      \[
      \abval : \VF \to^{\rv} \RV_0 \to^{\abs{ \cdot}} \RV^+_0 \to^{\vrv^+}  \Gamma^+_0
       \]
       is a (nontrivial) valuation with respect to the ordering $\leq^{-1}$, with valuation ring $\OO = \rv^{-1}(\RV^{\circ}_0)$ and maximal ideal $\MM = \rv^{-1}(\RV^{\circ\circ}_0)$, where, denoting $\vrv^+ \circ \abs{ \cdot}$ by $\abvrv$,
        \begin{align*}
     \RV^{\circ}_0 &= \{t \in \RV: 1 \leq^{-1} \abvrv(t) \}, \\
     \RV^{\circ \circ}_0 &= \{t \in \RV: 1 <^{-1} \abvrv(t)\}.
     \end{align*}

 \item \label{ax:t:model} The $\K$-sort (recall that $\K$ is informally referred to as a sort) is a model of $T$ and, as a field, is the residue field of the valued field $(\VF, \OO)$.

     The natural quotient map $\OO \fun \K$ is denoted by $\res$. For notational convenience, we extend the domain of $\res$ to $\VF$ by setting $\res(a) = 0$ for all $a \in \VF \mi \OO$. The following function is also denoted by $\res$:
     \[
     \RV \to^{\rv^{-1}} \VF \to^{\res} \K.
     \]

 \item \label{ax:tcon} ($T$-convexity). Let $f : \VF \fun \VF$ be a continuous function defined by an \LT-formula. Then $f(\OO) \sub \OO$.

\item Suppose that $\phi$ is an \LT-formula that defines a continuous function $f : \VF^m \fun \VF$. Then $\phi$ also defines a continuous function $\ol f : \K^m \fun \K$. Moreover, for all $a \in \OO^m$, we have $\res(f(a)) = \ol f(\res(a))$.  \label{ax:match}
\end{enumerate}
\end{defn}

By (Ax.~\ref{ax:t:model}) and Remark~\ref{rem:cont}, (Ax.~\ref{ax:match}) can be simplified as: for all function symbols $f$ of $\lan{T'}{}{}$ and all $a \in \OO^m$, $\res(f(a)) = \ol f(\res(a))$. Then it is routine to check that, except the surjectivity of the map $\rv$ and the nontriviality of the value group $\abs{\Gamma}$ (this is an existential axiom and is actually expressed in (Ax.~\ref{ax:K})), $\TCVF$ is also universally axiomatized.

Let $\mdl S$ be a substructure of a model $\mdl M$ of $\TCVF$. We say that $\mdl S$ is \emph{$\VF$-generated} if $\RV_0(\mdl S) = \rv(\VF(\mdl S))$. Thus $\mdl S$ is indeed a model of $\TCVF$ if it is $\VF$-generated and $\Gamma(\mdl S)$ is nontrivial. At any rate, $\VF(\mdl S)$, $\res(\VF(\mdl S))$, and $\K(\mdl S)$ are all models of $T$.

For $A \sub \VF(\mdl M) \cup \RV(\mdl M)$, the substructure generated by $A$ over $\mdl S$ is denoted by $\la \mdl S , A \ra$ or $\mdl S \la A \ra$. Clearly $\VF(\la \mdl S , A \ra) = \la \mdl S , A \ra_T$ (see Notation~\ref{nota:sub}).

\begin{rem}\label{signed:Gam}
Although the behavior of the valuation map $\abval$ in the traditional sense is coded in $\TCVF$, we shall work with the \emph{signed} valuation map, which is more natural in the present setting:
\[
\vv : \VF \to^{\rv} \RVV \to^{\vrv} \GAA,
\]
where the ordering $\leq$ on the \emph{signed value group} $\Gamma_0$ no longer needs to be inverted. It is also tempting to use the ordering $\leq$ in the \emph{value group} $\abs{\Gamma}_{\infty}$ instead of its inverse, but this makes citing results in the literature a bit awkward. We shall actually abuse the notation and denote the ordering $\leq^{-1}$ in $\abs{\Gamma}_{\infty}$ also by $\leq$; this should not cause confusion since the ordering on $\Gamma_0$ will rarely be used (we will indicate so explicitly when it is used).

The axioms above guarantee that the ordered abelian group ${\GAA} /{\pm 1}$ (here $\vv(\pm 1)$ is just written as $\pm 1$) with the bottom element $0$ is isomorphic to $\abs{\Gamma}_{\infty}$ if either one of the orderings is inverted. So $\abval$ may be thought of as the composition $\vv/{\pm 1} : \VF \fun {\GAA} /{\pm 1}$.
\end{rem}

\begin{conv}\label{how:gam}
Semantically we shall treat the value group $\GAA$ as an imaginary sort. However, syntactically any reference to $\GAA$ may be eliminated in the usual way and we can still work with $\lan{T}{RV}{}$-formulas for the same purpose.
\end{conv}

\begin{exam}\label{exam:RtQ}
Here our main reference is \cite{DMM94}. A restricted analytic function $\R^m \fun \R$ is given on the cube $[-1, 1]^n$ by a power series in $n$ variables over $\R$ that converges in a neighborhood of $[-1, 1]^n$, and $0$ elsewhere. Let $\lan{}{an}{}$ be the language that extends the language of ordered rings with a new function symbol for each restricted analytic function, $\usub{\R}{an}$ the real field with its natural $\lan{}{an}{}$-structure, and  $\usub{T}{an}$ the $\lan{}{an}{}$-theory of $\usub{\R}{an}$. Obviously $\usub{T}{an}$ is polynomially bounded. More importantly, it is universally axiomatizable and admits quantifier elimination in a slightly enlarged language, and hence there is no longer any need to extend $\usub{T}{an}$ by definitions as we have arranged in \S~\ref{subs:nota}. (This language is of course more natural than a brute force definitional extension that achieves the same thing, but we do not really care what it is).

A generalized power series with coefficients in the field $\R$ and exponents in the additive group $\Q$ is a formal sum $x = \sum_{q \in \Q} a_q t^q$ such that its support $\supp(x) = \{q \in \Q : a_q \neq 0\}$ is well-ordered. Let $\R \dpar{ t^{\Q} }$, $K$ for short, be the set of all such series. Addition and multiplication in $K$ are defined in the expected way, and this makes $K$ a field, generally referred to as a Hahn field. We consider $\R$ as a subfield of $K$ via the map $a \efun at^0$. The map $ K^\times \fun \Q$ given by $x \efun \min\supp(x)$ is indeed a valuation. Its valuation ring $\R \llbracket t^{\Q} \rrbracket$, $\OO$ for short, consists of those series $x$ with $\min\supp(x) \geq 0$ and its maximal ideal $\MM$ of those series $x$ with $\min\supp(x) > 0$. Its residue field admits a section onto $\R$ and hence is isomorphic to $\R$. It is well-known that $(K, \OO)$ is a henselian valued field and $K$ is real closed. Restricted analytic functions may be naturally interpreted in $K$. According to \cite[Corollary~2.11]{DMM94}, with its naturally induced ordering, $K$ is indeed an elementary extension of $\usub{\R}{an}$ and hence a model of $\usub{T}{an}$.

We turn $K$ into a model of $\TCVF$, with signed valuation, as follows. First of all, set $\RV = K^{\times} / (1 + \MM)$. Let $\rv : K^\times \fun \RV$ be the quotient map. The leading term of a series in $K^\times$ is its first term with nonzero coefficient.  It is easy to see that two series $x$, $y$ have the same leading term if and only if $\rv(x) = \rv(y)$ and hence $\RV$ is isomorphic to the subgroup of $K^\times$ consisting of all the leading terms. There is a natural isomorphism $a_qt^q \efun (q, a_q)$ from this latter group of leading terms to the group $\Q \oplus \R^\times$, through which we may identify $\RV$ with $\Q \oplus \R^\times$. Since $1 + \MM$ is a convex subset of $K^\times$, the total ordering on $K^\times$ induces a total ordering $\leq$ on $\RV$. This ordering $\leq$ is the same as the lexicographic ordering on $\Q \oplus \R^+$ or $\Q \oplus \R^-$ via the identification just made.

Let $\R^{+}$ be the multiplicative group of the positive reals and $\RV^{+} = \Q \oplus\R^+ $. Observe that $\R^{+}$ is a convex subgroup of $\RV$. The quotient group $\Gamma \coloneqq (\Q \oplus \R^\times) / \R^{+}$ is naturally isomorphic to the subgroup $\pm e^{\Q} \coloneqq e^{\Q} \cup - e^{\Q}$ of $\R^\times$ so that $\Q$ is identified with $e^\Q$ via the map $q \efun e^q$. Adding a new symbol $\infty$ to $\RV$, now it is routine to interpret $K$ as an $\lan{T}{RV}{}$-structure, with $T = \usub{T}{an}$ and the signed valuation given by
\[
x \efun \rv(x) = (q, a_q) \efun \sgn(a_q)e^{-q},
\]
where $\sgn(a_q)$ is the sign of $a_q$. It is also a model of $\TCVF$: all the axioms are more or less immediately derivable from the valued field structure, except (Ax.~\ref{ax:tcon}), which holds since $\usub{T}{an}$ is polynomially bounded, and (Ax.~\ref{ax:match}), which follows from \cite[Proposition~2.20]{DriesLew95}.
\end{exam}

\subsection{Quantifier elimination}

Recall from \S~\ref{intro} that  $T_{\textup{convex}}$ is the $\lan{}{convex}{}$-theory of pairs $(\mdl R, \OO)$ with $\mdl R \models T$ and $\OO$ a \emph{proper} $T$-convex subring. We may and shall view $T_{\textup{convex}}$ as the $\lan{}{convex}{}$-reduct of $\TCVF$.

\begin{thm}\label{tcon:qe}
The theory $T_{\textup{convex}}$ admits quantifier elimination and is complete.
\end{thm}
\begin{proof}
See \cite[Theorem~3.10, Corollary~3.13]{DriesLew95}.
\end{proof}

That $\OO$ is a proper subring cannot be expressed by a universal axiom. Of course, we can always add a new constant symbol $\imath$ to $\lan{}{convex}{}$ and an axiom  ``$\imath$ is in the maximal ideal'' to $T_{\textup{convex}}$ so that $T_{\textup{convex}}$ may indeed be formulated as a universal theory. In that case, every substructure of a model of $T_{\textup{convex}}$ is a model of $T_{\textup{convex}}$ and, moreover, $T_{\textup{convex}}$ has definable Skolem functions given by $\lan{T}{}{}(\imath)$-terms (this is an easy consequence of our assumption on $T$, quantifier elimination in $T_{\textup{convex}}$, and universality of $T_{\textup{convex}}$, as in \cite[Corollary~2.15]{DMM94}). We shall not implement this maneuver formally, even though the resulting properties may come in handy occasionally.

\begin{rem}\label{res:exp}
According to \cite[Remark~2.16]{DriesLew95}, there is a natural way to expand the residue field $\K$ of the $T_{\textup{convex}}$-model $(\mdl R, \OO)$  to a \T-model as follows. Let $\mdl R' \sub \OO$ be a maximal subfield with respect to the property of being an elementary \LT-substructure of $\mdl R$. It follows that $\mdl R'$ is isomorphic to $\K$ as fields via the residue map $\res$. Then we can expand $\K$ to a \T-model so that the restriction $\res \rest \mdl R'$ becomes an isomorphism of \LT-structures. This expansion procedure does not depend on the choice of $\mdl R'$.
\end{rem}

\begin{prop}\label{uni:exp}
Every $T_{\textup{convex}}$-model expands to a unique $\TCVF$-model up to isomorphism.
\end{prop}
\begin{proof}
Let $(\mdl R, \OO)$ be a $T_{\textup{convex}}$-model. It is enough to show that there is a canonical $\TCVF$-model expansion $(\mdl R, \RVV(\mdl R))$ of $(\mdl R, \OO)$, where $\mdl R$ is the $\VF$-sort, such that any other such expansion $(\mdl R, \RVV)$ is isomorphic to it. This canonical  expansion is constructed as follows.

Let $\RV(\mdl R)$ be the quotient group $\mdl R^\times / (1 + \MM)$ and $\rv : \mdl R^\times \fun \RV(\mdl R)$ the quotient map. As in Example~\ref{exam:RtQ}, it is routine to convert the pair $(\mdl R, \RVV(\mdl R))$ into an $\lan{T}{RV}{}$-structure and check that it satisfies all the axioms in Definition~\ref{defn:tcf}, where (Ax.~\ref{ax:t:model}) is implied by the construction just described  above. We shall refer to the obvious bijection between $(\mdl R, \RVV(\mdl R))$ and $(\mdl R, \RVV)$ as the identity map.  This map commutes with all the primitives of $\lan{T}{RV}{}$ except, possibly, those in the $\K$-sort. This is where the syntactical maneuver in Remark~\ref{rem:cont} comes in. Recall that all the functional primitives of $\lan{T'}{}{}$ define continuous functions in all the models of $T'$ and $T$ is a definitional extension of $T'$. It follows from (Ax.~\ref{ax:match}) that the identity map indeed induces an \LT-isomorphism between the two $\K$-sorts. Thus the two expansions are isomorphic.
\end{proof}

\begin{thm}\label{thm:complete}
The theory $\TCVF$ is complete.
\end{thm}
\begin{proof}
By Proposition~\ref{uni:exp}, every embedding between two $T_{\textup{convex}}$-models, which is necessarily elementary, expands uniquely to an $\lan{T}{RV}{}$-embedding between two $\TCVF$-models. This latter embedding is indeed elementary since $\TCVF$ admits quantifier elimination, which will be shown below. It follows that the theory $\TCVF$ is complete. But here we do not really need to go through that route. We can simply observe that, by the proof of Proposition~\ref{uni:exp}, $T_{\textup{convex}}$ and $\TCVF$ are equivalent in the sense mentioned in Remark~\ref{rem:cont}, and hence they are both complete if one of them is.
\end{proof}

\begin{conv}
From now on, we shall work in the model $\mmdl$ of $\TCVF$, which is the unique $\lan{T}{RV}{}$-expansion of the sufficiently saturated $T_{\textup{convex}}$-model $(\mdl R, \OO)$. We shall write $\VF(\mmdl)$ simply as $\VF$ or $\mdl R$, depending on the context, $\RVV(\mmdl)$ as $\RV_0$, etc. A subset in $\mmdl$ may simply be referred to as a set.

When we work in the \LT-reduct $\mdl R$ of $\mmdl$ instead of $\mmdl$, or just wish to emphasize that a set is definable in $\mdl R$ instead of $\mmdl$, the symbol ``$\lan{T}{}{}$'' or ``$T$'' will be inserted into the corresponding places in the terminology.
\end{conv}

Let $\mdl S \sub \mdl R$ be a small substructure  and $a, b \in \mdl R \mi \mdl S$ such that they make the same cut in (the ordering of) $\mdl S$. By \omin-minimality, there is an automorphism $\sigma$ of $\mdl R$ over $\mdl S$ such that $\sigma(a) = b$.

Recall that the field of exponents of $\mdl R$ is denoted by $\KKK$.

\begin{thm}\label{theos:qe}
The theory $\TCVF$ admits quantifier elimination.
\end{thm}
\begin{proof}
We shall run the usual Shoenfield test for quantifier elimination. To that end, let $\mdl M$ be a model of $\TCVF$, $\mdl S$ a substructure of $\mdl M$, and $\sigma : \mdl S \fun \mmdl$ an embedding. All we need to do is to extend $\sigma$ to an embedding $\mdl M \fun \mmdl$.

The construction is more or less a variation of that in the proof of \cite[Theorem~3.10]{Yin:QE:ACVF:min}. The strategy is to reduce the situation to Theorem~\ref{tcon:qe}. In the process of doing so, instead of the dimension inequality of the general theory of valued fields, the Wilkie inequality \cite[Corollary~5.6]{Dries:tcon:97} is used (see \cite[\S~3.2]{DriesLew95} for the notion of ranks of $T$-models). Note that, to use this inequality, we need to assume that $T$ is power-bounded.

Let $\mdl S_* = \la \VF(\mdl S) \ra$ and $t \in \RV(\mdl S) \mi \RV(\mdl S_*)$. Note that if such a $t$ does not exist then we have $\mdl S = \mdl S_*$ and its $\lan{}{convex}{}$-reduct is an $\lan{}{convex}{}$-substructure of the $\lan{}{convex}{}$-reduct of $\mdl M$, and hence an embedding as desired can be easily obtained by applying Theorem~\ref{tcon:qe} and Proposition~\ref{uni:exp}. Let $a \in \VF(\mdl M)$ with $\rv(a) = t$ and $b \in \VF$ with $\rv(b) = \sigma(t)$. Observe that, according to $\sigma$, $a$ and $b$ must make the same cut in $\VF(\mdl S)$ and $\VF(\sigma(\mdl S))$, respectively, and hence there is an \LT-isomorphism
\[
\bar \sigma : \la \mdl S_*, a \ra_T \fun \la \sigma(\mdl S_*), b \ra_T
\]
with $\bar \sigma(a) = b$ and $\bar \sigma \rest \VF(\mdl S) = \sigma \rest \VF(\mdl S)$. We shall show that $\bar \sigma$ expands to an isomorphism between $\la \mdl S_*, a \ra$ and $\la \sigma(\mdl S_*), b \ra$ that is compatible with $\sigma$.

Case (1): There is an $a_1 \in \la \mdl S_*, a \ra_T$ such that
\[
\abs{\OO(\mdl S_*) } < a_1 < \abs{ \VF(\mdl S_*) \mi \OO(\mdl S_*) }.
\]
Set $\abs{\Gamma}(\mdl S_*) = G$. Since $\OO(\la\mdl S_*, a \ra)$ is $T$-convex, by \cite[Lemma~5.4]{Dries:tcon:97} and \cite[Remark~3.8]{DriesLew95},
\begin{itemize}
  \item either $a_1 \in \OO(\la\mdl S_*, a \ra)$ and $\absG(\la \mdl S_*, a \ra) = G$ or
  \item $a_1 \notin \OO(\la\mdl S_*, a \ra)$ and $\absG(\la \mdl S_*, a \ra) \cong G \oplus \KKK$.
\end{itemize}
By the Wilkie inequality, if
\[
\absG(\la \mdl S_*, a \ra) \cong G \oplus \KKK
\]
then $\K(\la \mdl S_*, a \ra) = \K(\mdl S_*)$ and hence $\abvrv(t) \notin G$, which implies $\abvrv(\sigma(t)) \notin \sigma(G)$; conversely, if \[
\absG(\la \sigma(\mdl S_*), b \ra) \cong \sigma(G) \oplus \KKK
\]
then $\abvrv(t) \notin G$. Therefore
\[
\absG(\la \mdl S_*, a \ra) \cong G \oplus \KKK \quad \text{if and only if} \quad \absG(\la \sigma(\mdl S_*), b \ra) \cong \sigma(G) \oplus \KKK,
\]
which, by \cite[Remark~3.8]{DriesLew95}, is equivalent to saying that $a_1 \in \OO(\la\mdl S_*, a \ra)$ if and only if $\bar \sigma(a_1) \in \OO(\la \mdl \sigma(\mdl S_*), b \ra)$.

Subcase (1a): $a_1 \in \OO(\la \mdl S_*, a \ra)$. Subcase~(1a) of the proof of \cite[Theorem~3.10]{DriesLew95} shows that $\bar \sigma$ expands to an $\lan{}{convex}{}$-isomorphism and hence to an $\lan{T}{RV}{}$-isomorphism, which is also denoted by $\bar \sigma$. Since $\absG(\la \mdl S_*, a \ra) = G$, we may assume $t \in \K(\mdl M)$. By the Wilkie inequality, $\K(\la \mdl S_*, a \ra)$ is precisely the $T$-model generated by $t$ over $\K(\mdl S_*)$. So $\RV(\la \mdl S_*, a \ra) = \la \RV(\mdl S_*), t \ra$ and
\[
\bar \sigma \rest \RV(\la \mdl S_*, a \ra) = \sigma \rest \RV(\la \mdl S_*, a \ra).
\]

Subcase (1b): $a_1 \notin \OO(\la \mdl S_*, a \ra)$. As above, Subcase~(1b) of the proof of \cite[Theorem~3.10]{DriesLew95} shows that $\bar \sigma$ expands to an $\lan{T}{RV}{}$-isomorphism and this time $\K(\la \mdl S_*, a \ra) = \K(\mdl S_*)$. Again it is clear that
\[
\bar \sigma \rest \RV(\la \mdl S_*, a \ra) = \sigma \rest \RV(\la \mdl S_*, a \ra).
\]

Case (2): Case (1) fails. Then there is also no $b_1 \in \la \sigma(\mdl S_*), b \ra_T$ such that
\[
\abs{ \OO(\sigma(\mdl S_*)) } < b_1 < \abs{ \VF(\sigma(\mdl S_*)) \mi \OO(\sigma(\mdl S_*)) }.
\]
Using Case~(2) of the proof of \cite[Theorem~3.10]{DriesLew95}, compatibility between $\bar \sigma$ and $\sigma$ may be deduced as in Case (1) above.

Iterating this procedure, we may assume $\mdl S = \mdl S_*$. The theorem follows.
\end{proof}

\begin{cor}
For all set $A \sub \VF$, $\la A \ra$ is an elementary substructure of $\mmdl$ if and only if $\Gamma(\la A \ra)$ is nontrivial, that is, $\Gamma(\la A \ra) \neq \pm 1$.
\end{cor}

\begin{cor}\label{trans:VF}
Every parametrically $\lan{T}{RV}{}$-definable subset of $\VF^n$ is parametrically $\lan{}{convex}{}$-definable.
\end{cor}

This corollary already follows from Proposition~\ref{uni:exp}. Anyway, it enables us to transfer results in the theory of $T$-convex valued fields \cite{DriesLew95, Dries:tcon:97} into our setting, which we shall do without further explanation.

We include here a couple of generalities on immediate isomorphisms. Their proofs are built on that of Theorem~\ref{theos:qe} and hence we shall skip some details.

\begin{defn}
Let $\mdl M$, $\mdl N$ be substructures and $\sigma : \mdl M \fun \mdl N$ an $\lan{T}{RV}{}$-isomorphism. We say that $\sigma$ is an \emph{immediate isomorphism} if $\sigma(t) = t$ for all $t \in \RV(\mdl M)$.
\end{defn}

Note that if $\sigma$ is an immediate isomorphism then, \emph{ex post facto}, $\RV(\mdl M) = \RV(\mdl N)$.

\begin{lem}\label{imm:ext}
Every immediate isomorphism $\sigma : \mdl M \fun \mdl N$ can be extended to an immediate automorphism of $\mmdl$.
\end{lem}
\begin{proof}
Let $\mdl M_* = \la \VF(\mdl M) \ra$ and $\mdl N_* = \la \VF(\mdl N) \ra$. Let $t \notin \RV(\mdl M_*)$ and $a \in \rv^{-1}(t)$. Since $\sigma$ is immediate, $a$ makes the same cut in $\VF(\mdl M)$ and $\VF(\mdl N)$ according to $\sigma$. By the proof of  Theorem~\ref{theos:qe}, $\sigma$ may be extended to an immediate isomorphism $\la \mdl M, a \ra \fun \la \mdl N, a \ra$. Iterating this procedure, we reach a stage where the assertion simply follows from Theorem~\ref{tcon:qe}.
\end{proof}

We have something much stronger. For that, the following crucial property is needed.

\begin{prop}[Valuation property]\label{val:prop}
Let $\mdl M$ be a $\VF$-generated substructure and $a \in \VF$. Suppose that $\Gamma(\la \mdl M, a \ra) \neq \Gamma(\mdl M)$. Then there is a $d \in \VF(\mdl M)$ such that $\vv(a - d) \notin \vv(\mdl M)$.
\end{prop}
\begin{proof}
For the polynomially bounded case, see~\cite[Proposition~9.2]{DriesSpei:2000} and the remark thereafter. Apparently this is established in full generality (power-bounded) in \cite{tyne}, which is in a repository that is password-protected.
\end{proof}

\begin{lem}\label{imm:iso}
Let $\sigma : \mdl M \fun \mdl N$ be an immediate isomorphism. Let $a \in \VF \mi \VF(\mdl M)$ and $b \in \VF \mi \VF(\mdl N)$ such that $\rv(a - c) = \rv(b -\sigma(c))$ for all $c \in \VF(\mdl M)$. Then $\sigma$ may be extended to an immediate isomorphism $\bar \sigma : \la \mdl M, a \ra \fun \la \mdl N, b \ra$ with $\bar \sigma(a) = b$.
\end{lem}

Observe that, since every element of $\VF(\la \mdl M, a \ra) = \la \mdl M, a \ra_T$ is of the form $f(a, c)$, where $c \in \VF(\mdl M)$ and $f$ is a function symbol of $\lan{T}{}{}$, and similarly for $\la \mdl N, b \ra$, the lemma is equivalent to saying that $\rv(a - c) = \rv(b -\sigma(c))$ for all $c \in \VF(\mdl M)$ implies $\rv(f(a,c)) = \rv(f(b,\sigma(c)))$ for all $c \in \VF(\mdl M)$ and all function symbols of $\lan{T}{}{}$.

\begin{proof}
Without loss of generality, we may assume that $\mdl M$, $\mdl N$ are $\VF$-generated. According to $\sigma$, $a$ and $b$ must make the same cut respectively in $\VF(\mdl M)$ and $\VF(\mdl N)$, and hence there is an \LT-isomorphism $\bar \sigma : \la \mdl M, a \ra_T \fun \la \mdl N, b \ra_T$ with $\bar \sigma(a) = b$ that extends $\sigma \rest \VF(\mdl M)$. We shall first show that $\bar \sigma$ expands to an $\lan{T}{RV}{}$-isomorphism. There are two cases to consider, corresponding to the two cases in the proof of Theorem~\ref{theos:qe}.

Case (1): There is an $a' \in \la \mdl M, a \ra_T$ such that
\[
\abs{ \OO(\mdl M)} < a' < \abs{ \VF(\mdl M) \mi \OO(\mdl M) }.
\]
Let $f$ be a function symbol of $\lan{T}{}{}$ and $c \in \VF(\mdl M)$ such that $f(a, c) = a'$. Let $b' = \bar \sigma(f(a, c))$. Then we also have
\[
\abs{\OO(\mdl N)} < b' < \abs{ \VF(\mdl N) \mi \OO(\mdl N) }.
\]
If $a' \notin \OO(\la \mdl M, a \ra)$ then $\Gamma(\la \mdl M, a \ra) \neq \Gamma(\mdl M)$. By the valuation property, there is a $d \in \VF(\mdl M)$ such that $\vv(a - d) \notin \Gamma(\mdl M)$. Then $\vv(b - \sigma(d)) \notin \Gamma(\mdl N)$ and hence, by the Wilkie inequality, $\OO(\la \mdl N, b\ra)$ is the convex hull of $\OO(\mdl N)$ in $\la \mdl N, b \ra_T$. This implies $b' \notin \OO(\la \mdl N, b \ra)$. By symmetry and \cite[Remark~3.8]{DriesLew95}, we see that $a' \in \OO(\la \mdl M, a \ra)$ if and only if $b' \in \OO(\la \mdl N, b \ra)$, and hence
\[
\bar \sigma(\OO(\la \mdl M, a \ra)) = \OO(\la \mdl N, b \ra).
\]

Case (2): Case (1) fails. We may proceed exactly as in Case (2) of the proof of Theorem~\ref{theos:qe}.

This concludes our proof that $\bar \sigma$ expands to an $\lan{T}{RV}{}$-isomorphism.

Next, we show that $\bar \sigma$ is indeed immediate. If $\RV(\la \mdl M, a \ra) = \RV(\mdl M)$ then also $\RV(\la \mdl N, b \ra) = \RV(\mdl N)$, and there is nothing more to be done. So suppose $\RV(\la \mdl M, a \ra) \neq \RV(\mdl M)$. We claim that there is a $d \in \VF(\mdl M)$ such that $\rv(a - d) \notin \RV(\mdl M)$. We consider two (mutually exclusive) cases.

Case (1): $\Gamma(\la \mdl M, a \ra) \neq \Gamma(\mdl M)$. Then the valuation property gives such a $d$ directly.

Case (2): $\K(\la \mdl M, a \ra) \neq \K(\mdl M)$. Let $a'$ be as above. Let $\OO'$ be the \T-convex subring of $\VF(\mdl M)$ that does not contain $a'$, that is, $\OO'$ is the convex hull of $\OO(\mdl M)$ in $\la \mdl M, a \ra_T$.  Let $\vv'$, $\Gamma'(\la \mdl M, a \ra)$ be the corresponding signed valuation map and signed value group. Then the valuation property yields a $d \in \VF(\mdl M)$ such that $\vv'(a - d) \notin \Gamma'(\mdl M)$. Since
\[
\abs{\Gamma'}(\la \mdl M, a \ra) \cong \abs{\Gamma}(\mdl M) \oplus \KKK,
\]
there is a $\gamma \in \abs{\Gamma}(\mdl M)$ such that (exactly) one of the following two relations hold:
\begin{gather*}
\abs{ \OO_\gamma(\mdl M)} < \abs{a- d} < \abs{ \VF(\mdl M) \mi \OO_\gamma(\mdl M) },\\
\abs{ \MM_\gamma(\mdl M)} < \abs{a- d} < \abs{ \VF(\mdl M) \mi \MM_\gamma(\mdl M) },
\end{gather*}
where
\[
\OO_\gamma = \{c \in \VF: \abval(c) \geq \gamma\} \dand \MM_\gamma = \{c \in \VF: \abval(c) > \gamma\}.
\]
It is not hard to see that, in either case,  $\rv(a - d) \notin \RV(\mdl M)$.

Since $\rv(a - d) = \rv(b - \sigma(d)) \eqqcolon t$, by the Wilkie inequality, $\RV(\la \mdl M, a \ra) = \la \RV(\mdl M), t \ra$ and hence   $\bar \sigma$ must be immediate.
\end{proof}

\subsection{Fundamental structure of $T$-convex valuation}

We review some fundamental facts concerning the valuation in $\mmdl$. Additional notation and terminology are also introduced.

Recall \cite[Theorem~A]{Dries:tcon:97}: The structure of definable sets in the $\K$-sort is precisely that given by the theory $T$.

Recall \cite[Theorem~B]{Dries:tcon:97}: The structure of definable sets in the (imaginary) $\abs \Gamma$-sort is precisely that given by the \omin-minimal theory of nontrivially ordered vector spaces over $\KKK$. The structure of definable sets in the (imaginary) $\Gamma$-sort is the same one modulo the sign. In particular, every definable function in the $\Gamma$-sort is definably piecewise $\KKK$-linear modulo the sign.

\begin{lem}\label{gk:ortho}
If $f : \Gamma \fun \K$ is a definable function then $f(\K)$ is finite. Similarly, if $g : \K \fun \Gamma$ is a definable function then $g(\Gamma)$ is finite.
\end{lem}
\begin{proof}
See \cite[Proposition~5.8]{Dries:tcon:97}.
\end{proof}

Note that \cite[Theorem~B, Proposition~5.8]{Dries:tcon:97} require that $T$ be power-bounded.

\begin{nota}\label{gamma:what}
Recall convention~\ref{how:gam}. There are two ways of treating an element $\gamma \in \Gamma$: as a point --- when we study $\Gamma$ as an independent structure, see Convention~\ref{how:gam}  --- or a subset of $\mmdl$ --- when we need to remain in the realm of definable sets in $\mmdl$. The former perspective simplifies the notation but is of course dispensable.

We shall write $\gamma$ as $\gamma^\sharp$ when we want to emphasize that it is the set $\vrv^{-1}(\gamma)$ in $\mmdl$ that is being  considered. More generally, if $I$ is a set in $\Gamma$ then we write $I^\sharp = \bigcup\{\gamma^\sharp: \gamma \in I\}$. Similarly, if $U$ is a set in $\RV$ then $U^\sharp$ stands for $\bigcup\{\rv^{-1}(t): t \in U\}$.
\end{nota}

Since $\TCVF$ is a weakly \omin-minimal theory (see \cite[Corollary~3.14]{DriesLew95} and Corollary~\ref{trans:VF}), we can use the dimension theory of \cite[\S~4]{mac:mar:ste:weako} in $\mmdl$.

\begin{defn}
The \emph{$\VF$-dimension} of a definable set $A$,  denoted by $\dim_{\VF}(A)$, is the largest natural number $k$ such that, possibly after re-indexing of the $\VF$-coordinates, $\pr_{\leq k}(A_t)$ has nonempty interior for some $t \in A_{\RV}$.
\end{defn}

For all substructures $\mdl M$ and all $a \in \VF$, $\VF(\dcl_{\mdl M}( a)) = \la \mdl M , a \ra_T$, where $\dcl_{\mdl M}(a)$ is the definable closure of $a$ over $\mdl M$. This implies that the exchange principle with respect to definable closure --- or algebraic closure, which is the same thing since there is an ordering --- holds in the $\VF$-sort, because it holds for \T-models. Therefore, by \cite[\S~4.12]{mac:mar:ste:weako}, we may equivalently define $\dim_{\VF}(A)$ to be the maximum of the algebraic dimensions of the fibers $A_t$, $t \in A_{\RV}$.

Algebraic dimension is defined for (any sort of) any theory whose models have the exchange property with respect to algebraic closure, or more generally any suitable notion of closure. In the present setting, the algebraic dimension of a set $B \sub \VF^n$ that is definable over a substructure $\mdl M$ is just the maximum of the ranks of the \T-models $\la \mdl M , b \ra_T$, $b \in B$, relative to the \T-model $\VF(\mdl M)$ (again, see \cite[\S~3.2]{DriesLew95} for the notion of ranks of $T$-models). It can be shown that this does not depend on the choice of $\mdl M$.

Yet another way to define this notion of $\VF$-dimension is to imitate \cite[Definiton~4.1]{Yin:special:trans}, since we have:

\begin{lem}\label{altVFdim}
If $\dim_{\VF}(A) = k$ then $k$ is the smallest number such that there is a definable injection $f: A \fun \VF^k \times \RV^l$.
\end{lem}
\begin{proof}
This is immediate by a straightforward argument combining the exchange principle, Lemma~\ref{RV:no:point} below, and compactness.

Alternatively, we may just quote \cite[Theorem~4.11]{mac:mar:ste:weako}.
\end{proof}

\begin{rem}[$\RV$-dimension and $\Gamma$-dimension]\label{rem:RV:weako}
It is routine to verify that the axioms concerning only the $\RV$-sort are all universal except for the one asserting that $\K^{\times}$ is a proper subgroup, which is existential. These axioms amount to a weakly \omin-minimal theory also and the exchange principle holds for this theory. Therefore, we can use the dimension theory of \cite[\S~4]{mac:mar:ste:weako} directly in the $\RV$-sort as well. We call it the $\RV$-dimension and the corresponding operator is denoted by $\dim_{\RV}$. Note that $\dim_{\RV}$ does not depend on parameters (see \cite[\S~4.12]{mac:mar:ste:weako}) and agrees with the \omin-minimal dimension in the $\K$-sort (see \cite[\S~4.1]{dries:1998}) whenever both are applicable.

Similarly we shall use \omin-minimal dimension in the $\Gamma$-sort and call it the $\Gamma$-dimension. The corresponding operator is denoted by $\dim_{\Gamma}$.
\end{rem}

\begin{lem}\label{dim:cut:gam}
Let $U \sub \RV^n$ be a definable set with $\dim_{\RV}(U) = k$. Then $\dim_{\RV}(U_{\gamma}) = k$ for some $\gamma \in \vrv(U)$.
\end{lem}

Here $U_{\gamma}$ denotes the pullback of $\gamma$ along the obvious function $\vrv \rest U$, in line with the convention set in the last paragraph of Notation~\ref{indexing}.

\begin{proof}
By \cite[Theorem~4.11]{mac:mar:ste:weako} we may assume $n=k$. Then, for some $\gamma \in \vrv(U)$, $U_{\gamma}$ contains an open subset of $\RV^n$. The lemma follows.
\end{proof}

\begin{lem}\label{gam:red:K}
Let $D \sub \Gamma^n$ be a definable set  with $\dim_{\Gamma}(D) = k$. Then $D^\sharp$ is definably bijective to a disjoint union of finitely many sets of the form $(\K^+)^{n-k} \times D'^\sharp$, where $D' \sub \Gamma^k$.
\end{lem}
\begin{proof}
Over a definable finite partition of $D$, we may assume that $D \sub (\Gamma^+)^n$ and the restriction $\pr_{\leq k} \rest D$ is injective. It follows from \cite[Theorem~B]{Dries:tcon:97} that the induced  function $f : D_{\leq k} \fun D_{>k}$ is piecewise $\KKK$-linear. Thus, for every $\gamma \in D_{\leq k}$ and every $t \in \gamma^\sharp$ there is a $t$-definable point in $f(\gamma)^\sharp$. The assertion follows.
\end{proof}

Taking disjoint union of finitely many definable sets of course will introduce extra bookkeeping coordinates, but we shall suppress this in notation.

\begin{rem}[\omin-minimal sets in $\RV$]\label{omin:res}
The theory of \omin-minimality, in particular its terminologies and notions, may be applied to a set $U \sub \RV^n$ such that $\vrv(U)$ is a singleton or, more generally, is finite. For example, we shall say that $U$ is a \emph{cell} if the multiplicative translation $U / u \sub (\K^+)^n$ of $U$ by some $u \in U$ is an \omin-minimal cell (see \cite[\S~3]{dries:1998}); this definition does not depend on the choice of $u$. Similarly, the \emph{\omin-minimal Euler characteristic} $\chi(U)$ of such a set $U$ is the \omin-minimal Euler characteristic of $U / u$ (see \cite[\S~4.2]{dries:1998}). This definition may be extended to disjoint unions of finitely many (not necessarily disjoint) sets $U_i \sub \RV^n \times \Gamma^m$ such that each $\vrv(U_i)$ is finite.
\end{rem}

\begin{thm}\label{groth:omin}
Let $U$, $V$ be definable sets in $\RV$ with $\vrv(U)$, $\vrv(V)$ finite. Then there is a definable bijection between $U$ and $V$ if and only if
\[
\dim_{\RV}(U) = \dim_{\RV}(V) \dand \chi(U) = \chi(V).
\]
\end{thm}
\begin{proof}
See \cite[\S~8.2.11]{dries:1998}.
\end{proof}

\begin{defn}[Valuative discs]\label{defn:disc}
A set $\gb \sub \VF$ is an \emph{open disc} if there is a $\gamma \in |\Gamma|$ and a $b \in \gb$ such that $a \in \gb$ if and
only if $\abval(a - b) > \gamma$; it is a \emph{closed disc} if $a \in \gb$ if and only if $\abval(a - b) \geq \gamma$. The point $b$ is a \emph{center} of $\gb$. The value $\gamma$ is the \emph{valuative radius} or simply the \emph{radius} of $\gb$, which is denoted by $\rad (\gb)$. A set of the form $t^\sharp$, where $t \in \RV$, is called an \emph{$\RV$-disc} (recall Notation~\ref{gamma:what}).

A closed disc with a maximal open subdisc removed is called a \emph{thin annulus}.

A set $\gp \sub \VF^n \times \RV_0^m$ of the form $(\prod_{i \leq n} \gb_i) \times t$ is an (\emph{open, closed, $\RV$-}) \emph{polydisc}, where each $\gb_i$ is an (open, closed, $\RV$-) disc. The \emph{polyradius} $\rad(\gp)$ of $\gp$ is the tuple $(\rad(\gb_1), \ldots, \rad(\gb_n))$, whereas the \emph{radius} of $\gp$ is $\min \rad(\gp)$. If all the discs $\gb_i$ are of the same valuative radius then $\gp$ is referred to as a \emph{ball}.

The open and the closed polydiscs centered at a point $a  \in \VF^n$ with polyradius $\gamma \in |\Gamma|^n$ are denoted by $\go(a, \gamma)$ and $\gc(a, \gamma)$, respectively.

The \emph{$\RV$-hull} of a set $A$, denoted by $\RVH(A)$, is the union of all the $\RV$-polydiscs whose intersections with $A$ are nonempty. If $A$ equals $\RVH(A)$ then $A$ is called an \emph{$\RV$-pullback}.
\end{defn}

The map $\abval$ is constant on a disc if and only if it does not contain $0$ if and only it is contained in an $\RV$-disc. If two discs have nonempty intersection then one of them contains the other. Many such elementary facts about discs will be used throughout the rest of the paper without further explanation.

\begin{nota}[The definable sort $\DC$ of discs]\label{disc:exp}
At times it will be more convenient to work in the traditional expansion $\mdl R_{\rv}^{\textup{eq}}$ of $\mmdl$ by all definable sorts. However, for our purpose, a much simpler expansion $\mdl R_{\rv}^{\bullet}$ suffices. This expansion has only one additional sort $\DC$ that contains, as elements, all the open and closed discs (since each point in $\VF$ may be regarded as a closed disc of valuative radius $\infty$, for convenience, we may and occasionally do think of $\VF$ as a subset of $\DC$). Heuristically, we may think of a disc that is properly contained in an $\RV$-disc as a ``thickened'' point of certain stature in $\VF$. For each $\gamma \in \absG$, there are two additional cross-sort maps $\VF \fun \DC$ in $\mdl R_{\rv}^{\bullet}$, one sends $a$ to the open disc, the other to the closed disc, of radius $\gamma$ that contain $a$.

The expansion $\mdl R_{\rv}^{\bullet}$ can help reduce the technical complexity of our discussion. However, as is the case with the imaginary $\Gamma$-sort, it is conceptually inessential since, for the purpose of this paper, all allusions to discs as (imaginary) elements may be eliminated in favor of objects already definable in $\mmdl$.

Whether parameters in $\DC$ are used or not shall be indicated explicitly, if it is necessary. Note that it is redundant to include in $\DC$ discs centered at $0$, since they may be identified with their valuative radii.

For a disc $\ga \sub \VF$, the corresponding imaginary element in $\DC$ is denoted by $\code{\ga}$ when notational distinction makes the discussion more streamlined; $\code{\ga}$ may be heuristically thought of as the ``name'' of $\ga$. Conversely, a set $D \sub \DC$ is often identified with the set $\{\ga : \code \ga \in D\}$, in which case $\bigcup D$ denotes a subset of $\VF$.
\end{nota}

\begin{nota}\label{nota:tor}
For each $\gamma \in |\Gamma|$, let $\MM_\gamma$ and $\OO_{\gamma}$ be the open and closed discs around $0$ with radius $\gamma$, respectively. Assume $\gamma \geq 0$. Let $\RV_{\gamma} = \VF^{\times} / (1 + \MM_\gamma)$, which is a subset of $\DC$. It is an abelian group and also inherits an ordering from $\VF^\times$. The canonical map $\VF^{\times} \fun \RV_{\gamma}$ is denoted by $\rv_{\gamma}$ and is augmented by $\rv_{\gamma}(0) = 0$.

If $\code \gb \in \DC$, $b \in \gb$, and $\rad(\gb) \leq \abval(b) + \gamma$ then $\gb$ is a union of discs of the form $\rv_{\gamma}^{-1}(\code \ga)$. In this case, we shall abuse the notation slightly and write $\code \ga \in \gb$, $\gb \sub \RV_{\gamma}$, etc.

For each $\code \ga \in \RV_{\gamma}$, let $\tor (\code \ga) \sub \RV_{\gamma}$ be the $\code \ga$-definable subset such that $\rv^{-1}_{\gamma}(\tor (\code \ga))$ forms the smallest closed disc containing $\ga$. Set
\begin{align*}
\tor^{\times}(\code \ga) & = \tor (\code \ga) \mi \code \ga,\\
\tor^+(\code \ga) &= \{t \in \tor(\code \ga):  t > \code \ga\},\\
 \tor^-(\code \ga) &= \{t \in \tor(\code \ga):  t < \code \ga\}.
\end{align*}
If $\code \ga = (\code {\ga_1}, \ldots, \code {\ga_n})$ with $\code {\ga_i} \in \RV_{\gamma_i}$ then $\prod_i \tor(\code{\ga_i})$ is simply written as $\tor(\code \ga)$; similarly for $\tor^{\times}(\code \ga)$, $\tor^+(\code \ga)$, etc.

If $\gamma = 0$ then we may, for all purposes, identify $\tor^{\times} (\code \ga)$, $\tor (\code \ga)$, etc., with $\tor^{\times}(\alpha) \coloneqq \abvrv^{-1}(\alpha) \sub \RV$, $\tor(\alpha) \coloneqq \tor^{\times}(\alpha) \cup \{0\}$, etc., where $\alpha = \rad(\ga)$.
\end{nota}

\begin{rem}[$\K$-torsors]\label{rem:K:aff}
Let $\code \ga \in \RV_{\gamma}$ and $\alpha = \rad(\ga)$. Since, via additive translation by $\code \ga$, there is a canonical $\code \ga$-definable order-preserving bijection
\[
\aff_{\goedel{\ga}} :\tor(\code \ga) \fun \tor(\alpha),
\]
we see that $\code \ga$-definable subsets of $\tor(\code \ga)^n$ naturally correspond to those of $\tor(\alpha)^n$. If there is an $\code \ga$-definable $t \in \tor^{\times}(\alpha)$ then, via multiplicative translation by $t$, this correspondence may be extended to $\code \ga$-definable subsets of $\tor(0)^n = \K^n$. More generally, for any $t \in \tor^{\times}(\alpha)$, the induced bijection $\tor(\code \ga) \fun \K$ is denoted by $\aff_{\goedel \ga, t}$. Consequently, $\tor(\code \ga)$ may be viewed as a $\K$-torsor and, as such, is equipped with much of the structure of $\K$.
\end{rem}

\begin{defn}[Derivation between $\K$-torsors]\label{rem:tor:der}
Let $\code \ga$, $\alpha$ be as above. Let $\goedel \gb \in \RV_{\delta}$ and $\beta = \rad(\gb)$. Let $f : \tor(\code \ga) \fun \tor(\code \gb)$ be a function. We define the \emph{derivative} $\ddx f$ of $f$ at any point $\code \gd \in \tor(\code \ga)$ as follows. Choose any $t \in \tor^{\times}(\alpha)$ and any $s \in \tor^{\times}(\beta)$. Consider the function
\[
f_{\goedel \ga, \goedel \gb, t,s} : \K \to^{\aff^{-1}_{\goedel \ga, t}} \tor(\code \ga) \to^f \tor(\goedel \gb) \to^{\aff_{\goedel \gb, s}} \K.
\]
Put $r = \aff_{\goedel \ga, t}(\goedel \gd)$ and suppose that $\frac{d}{dx} f_{\goedel \ga, \goedel \gb, t,s}(r) \in \K$ exists. Then we set
\[
\tfrac{d}{d x} f(\goedel \gd) = s t^{-1} \tfrac{d}{d x} f_{\goedel \ga, \goedel \gb, t,s}(r) \in \tor(\beta - \alpha).
\]
It is routine to check that this construction does not depend on the choice of $\code \ga$, $\goedel \gb$, $t$, $s$ and hence the derivative $\ddx f(\goedel \gd)$ is well-defined.
\end{defn}

\begin{defn}[$\vv$-intervals]
Let $\ga$, $\gb$ be discs, not necessarily disjoint. The subset $\ga < x < \gb$ of $\VF$, if it is not empty, is called an \emph{open $\vv$-interval} and is denoted by $(\ga, \gb)$, whereas the subset
\[
\{a \in \VF : \ex{x \in \ga, y \in \gb} ( x \leq a \leq y) \}
\]
if it is not empty, is called a \emph{closed $\vv$-interval} and is denoted by $[\ga, \gb]$. The other $\vv$-intervals $[\ga, \gb)$, $(-\infty, \gb]$, etc., are defined in the obvious way, where $(-\infty, \gb]$ is a closed (or half-closed) $\vv$-interval that is unbounded from below.

Let $A$ be such a $\vv$-interval. The discs $\ga$, $\gb$ are called the \emph{end-discs} of $A$. If $\ga$, $\gb$ are both points in $\VF$ then of course we just say that $A$ is an interval and if $\ga$, $\gb$ are both $\RV$-discs then we say that $A$ is an $\RV$-interval. If $A$ is of the form $(\ga, \gb]$ or $[\gb, \ga)$, where $\ga$ is an open disc and $\gb$ is the smallest closed disc containing $\ga$, then $A$ is called a \emph{half thin annulus} and the \emph{radius} of $A$ is $\rad(\gb)$.

Two $\vv$-intervals are \emph{disconnected} if their union is not a $\vv$-interval.
\end{defn}

Obviously the open $\vv$-interval $(\ga, \gb)$ is empty if $\ga$, $\gb$ are not disjoint. Equally obvious is that a $\vv$-interval is definable over some substructure $\mdl S$ if and only if its end-discs are definable over $\mdl S$.

\begin{rem}[Holly normal form]\label{rem:HNF}
By the valuation property Proposition~\ref{val:prop} and \cite[Proposition~7.6]{Dries:tcon:97}, we have an important tool called \emph{Holly normal form} \cite[Theorem~4.8]{holly:can:1995} (henceforth abbreviated as HNF); that is, every definable subset of $\VF$ is a unique union of finitely many definable pairwise disconnected $\vv$-intervals. This is obviously a generalization of the \omin-minimal condition.
\end{rem}

\section{Definable sets in $\VF$}\label{def:VF}
From here on, we shall work with a fixed small substructure $\mdl S$ of $\mdl R_{\rv}$, also occasionally of $\mdl R_{\rv}^{\bullet}$ (primarily in this section). The conceptual reason for this move is that the Grothendieck rings in our main construction below change their meaning if the  set of  parameters changes. In particular, allowing all parameters trivializes the whole construction somewhat. For instance, every definable set will contain a definable point. Consequently, all Galois actions on the classes of finite definable sets are killed off, and this is highly undesirable for motivic integration in algebraically closed valued fields. Admittedly, this problem is not as severe in our setting. Anyway, we follow the practice in \cite{hrushovski:kazhdan:integration:vf}.

Note that $\mdl S$ is regarded as a part of the language now and hence, contrary to the usual convention in the model-theoretic literature, ``$\0$-definable'' or ``definable'' only means ``$\mdl S$-definable'' instead of ``parametrically definable'' if no other qualifications are given. To simplify the notation, we shall not mention $\mdl S$ and its extensions in context if no confusion can arise. For example, the definable closure operator $\dcl_{\mdl S}$, etc., will simply be written as $\dcl$, etc.

For the moment we do not require that $\mdl S$ be $\VF$-generated or $\Gamma(\mdl S)$ be nontrivial. When we work in $\mdl R_{\rv}^{\bullet}$ --- either by introducing parameters of the form $\code \ga$ or the phrase ``in $\mdl R_{\rv}^{\bullet}$'' --- the substructure $\mdl S$ may contain names for discs that may or may not be definable from $\VF(\mdl S) \cup \RV(\mdl S)$.

\subsection{Definable functions and atomic open discs}

The structural analysis of definable sets in $\VF$ below is, for the most part, of a rather technical nature. One highlight is Corollary~\ref{part:rv:cons}. It is a crucial ingredient of the proof in \cite{halyin} that all definable closed sets in an arbitrary power-bounded \omin-minimal field admit Lipschitz stratification.

\begin{conv}\label{topterm}
Since apart from $\leq$ the language $\lan{T}{}{}$ only has function symbols, we may and shall assume that, in any $\lan{T}{RV}{}$-formula, every \LT-term occurs in the scope of an instance of the function symbol $\rv$. For example, if $f(x)$, $g(x)$ are \LT-terms then the formula $f(x) < g(x)$ is equivalent to $\rv(f(x) - g(x)) < 0$. The \LT-term $f(x)$ in $\rv(f(x))$ shall be referred to as a \emph{top \LT-term}.
\end{conv}

We begin by studying definable functions between various regions of the structure.

\begin{lem}\label{Ocon}
Let $f : \OO \fun \VF$ be a definable function. Then for some $\gamma \in \GAA$ and $a \in \OO$ we have $\vv(f(b)) = \gamma$ for all $b > a$ in $\OO$.
\end{lem}
\begin{proof}
See \cite[Proposition~4.2]{Dries:tcon:97}.
\end{proof}

Note that this is false if $T$ is not power-bounded.

A definable function $f$ is \emph{quasi-\LT-definable} if it is a restriction of an \LT-definable function (with parameters in $\VF(\mdl S)$, of course).

\begin{lem}\label{fun:suba:fun}
Every definable function $f : \VF^n \fun \VF$ is piecewise quasi-\LT-definable; that is, there are a definable finite partition $A_i$ of $\VF^n$ and \LT-definable functions $f_i: \VF^n \fun \VF$ such that $f \rest A_i = f_i \rest A_i$ for all $i$.
\end{lem}
\begin{proof}
By compactness, this is immediately reduced to the case $n = 1$. In that case, let $\phi(x, y)$ be a quantifier-free formula that defines $f$. Let $\tau_i(x, y)$ enumerate the top \LT-terms in $\phi(x, y)$. For each $a \in \VF$ and each $t_i(a, y)$, let $B_{a, i} \sub \VF$ be the characteristic finite subset of the function $t_i(a, y)$ given by \omin-minimal monotonicity (see \cite[\S~3.1]{dries:1998}). It is not difficult to see that if $f(a) \notin \bigcup_i B_{a, i}$ then there would be a $b \neq f(a)$ such that
\[
\rv(\tau_i(a, b)) = \rv(\tau_i(a, f(a)))
\]
for all $i$ and hence $\phi(a, b)$ holds, which is impossible since $f$ is a function. The lemma follows.
\end{proof}

This lemma is just a variation of \cite[Lemma~2.6]{Dries:tcon:97}.

\begin{cor}[Monotonicity]\label{mono}
Let $A \sub \VF$ and $f : A \fun \VF$ be a definable function. Then there is a definable finite partition of $A$ into $\vv$-intervals $A_i$ such that every $f \rest A_i$ is quasi-\LT-definable, continuous, and monotone (constant or strictly increasing or strictly decreasing). Consequently, each $f(A_i)$ is a $\vv$-interval.
\end{cor}
\begin{proof}
This is immediate by Lemma~\ref{fun:suba:fun}, \omin-minimal monotonicity, and HNF.
\end{proof}

This corollary is a version of \cite[Corollary~2.8]{Dries:tcon:97}, slightly finer due to the presence of HNF.

\begin{cor}\label{uni:fun:decom}
For the function $f$ in Corollary~\ref{mono}, there is a definable function $\pi : A \fun \RV^2$ such that, for each $t \in \RV^2$, $f \rest A_t$ is either constant or injective.
\end{cor}
\begin{proof}
This follows easily from monotonicity. Also, the proof of \cite[Lemma~4.11]{Yin:QE:ACVF:min} still works.
\end{proof}

\begin{lem}\label{RV:no:point}
Given a tuple $t = (t_1, \ldots, t_n) \in \RV$, if $a \in \VF$ is $t$-definable then $a$ is definable. Similarly, for $\gamma = (\gamma_1, \ldots, \gamma_n) \in \Gamma$, if $t \in \RV$ is $\gamma$-definable then $t$ is definable.
\end{lem}
\begin{proof}
The first assertion follows directly from Lemma~\ref{fun:suba:fun}. It can also be easily seen through an induction on $n$ with the trivial base case $n=0$. For any $b \in t_n^\sharp$, by the inductive hypothesis, we have $a \in \VF(\la b \ra)$. If $a$ were not definable then, by the exchange principle, we would have $b \in \VF(\la a \ra)$ and hence $t_n^\sharp \sub \VF(\la a \ra)$, which is impossible. The second assertion is similar, using the exchange principle in the $\RV$-sort (see Remark~\ref{rem:RV:weako}).
\end{proof}

\begin{cor}\label{function:rv:to:vf:finite:image}
Let $U \sub \RV^m$ be a definable set and $f : U \fun \VF^n$ a definable function. Then $f(U)$ is finite.
\end{cor}
\begin{proof}
We may assume $n=1$. Then this is immediate by Lemma~\ref{RV:no:point} and compactness.
\end{proof}

There is a more general version of Lemma~\ref{RV:no:point} that involves parameters in the $\DC$-sort:

\begin{lem}\label{ima:par:red}
Let $\code \ga = (\code{\ga_1}, \ldots, \code{\ga_n}) \in \DC^n$. If $a \in \VF$ is $\code \ga$-definable then $a$ is definable.
\end{lem}
\begin{proof}
We proceed by induction on $n$. Let $b \in \ga_n$ and $t \in \RV$ such that $\abvrv(t) = \rad(\ga_n)$. Then $a$ is $(\code {\ga_1}, \ldots, \code {\ga_{n-1}}, t, b)$-definable. By the inductive hypothesis and Lemma~\ref{RV:no:point}, we have $a \in \VF(\la b \ra)$. If $a$ were not definable then we would have $b \in \VF(\la a \ra)$ and hence $\ga_n \sub \VF(\la a \ra)$, which is impossible unless $\ga_n$ is a definable point in $\VF$.
\end{proof}

\begin{lem}\label{open:K:con}
In $\mdl R^{\bullet}_{\rv}$, let $\ga \sub \VF$ be a definable open disc and $f : \ga \fun \K$ a definable nonconstant function. Then there is a definable proper subdisc $\gb \sub \ga$ such that $f \rest (\ga \mi \gb)$ is constant.
\end{lem}
\begin{proof}
If $\gb_1$ and $\gb_2$ are two proper subdiscs of $\ga$ such that $f \rest (\ga \mi \gb_1)$ and $f \rest (\ga \mi \gb_2)$ are both constant then $\gb_1$ and $\gb_2$ must be concentric, that is, $\gb_1 \cap \gb_2 \neq \0$, for otherwise $f$ would be constant. Therefore, it is enough to show that $f \rest (\ga \mi \gb)$ is constant for some proper subdisc $\gb \sub \ga$. To that end, without loss of generality, we may assume that $\ga$ is centered at $0$. For each $\gamma \in \vv(\ga) \sub \Gamma$, by \cite[Theorem~A]{Dries:tcon:97} and \omin-minimality, $f(\vv^{-1}(\gamma))$ contains a $\gamma$-definable element $t_{\gamma} \in \K$. By weak \omin-minimality, $f(\vv^{-1}(\gamma)) = t_{\gamma}$ for all but finitely many $\gamma \in \vv(\ga)$. Let $g : \vv(\ga) \fun \K$ be the definable function given by $\gamma \fun t_{\gamma}$. By Lemma~\ref{gk:ortho}, the image of $g$ is finite. The assertion follows.
\end{proof}

Alternatively, we may simply quote \cite[Theorem~1.2]{jana:omin:res}.

\begin{defn}
Let $D$ be a set of parameters. We say that a (not necessarily definable) nonempty set $A$ \emph{generates a (complete) $D$-type} if, for every $D$-definable set $B$, either $A \sub B$ or $A \cap B = \0$. In that case, $A$ is \emph{$D$-type-definable} if no set properly contains $A$ and also generates a $D$-type. If $A$ is $D$-definable and generates a $D$-type, or equivalently, if $A$ is both $D$-definable and $D$-type-definable then we say that $A$ is \emph{$D$-atomic} or \emph{atomic over $D$}.
\end{defn}

We simply say ``atomic'' when $D =\0$.

In the literature, a type could be a partial type and hence a type-definable set may have nontrivial intersection with a definable set. In this paper, since partial types do not play a role, we shall not carry the superfluous qualifier ``complete'' in our terminology.

\begin{rem}[Taxonomy of atomic sets]\label{rem:type:atin}
It is not hard to see that, by HNF, if $\gi \sub \VF$ is atomic then $\gi$ must be a $\vv$-interval. In fact, there are only four possibilities for $\gi$: a point, an open disc, a closed disc, and a half thin annulus. There are no ``meaningful'' relations between them, see Lemma~\ref{atom:type}.
\end{rem}

\begin{lem}\label{atom:gam}
In $\mdl R^{\bullet}_{\rv}$, let $\ga$ be an atomic set. Then $\ga$ remains $\gamma$-atomic for all $\gamma \in \Gamma$. Moreover, if $\ga \sub \VF^n$ is an open polydisc then it remains $\code \ga$-atomic.
\end{lem}
\begin{proof}
The first assertion is a direct consequence of definable choice in the $\Gamma$-sort. For the second assertion, let $\gamma = \rad(\ga)$. If $\ga$ were not $\code \ga$-atomic then, by compactness, there would be a $\gamma$-definable subset $A \sub \VF^n$ such that $A \cap \ga$ is nonempty and, for all open polydisc $\gb$ with $\gamma = \rad(\gb)$, if $A \cap \gb$ is nonempty then it is a proper subset of $\gb$ --- this contradicts the first assertion that $\ga$ is $\gamma$-atomic.
\end{proof}

Recall from \cite[Definition~4.5]{mac:mar:ste:weako} the notion of a cell in a weakly \omin-minimal structure. In our setting, it is easy to see that, by HNF, we may require that the images of the bounding functions $f_1$, $f_2$ of a cell $(f_1, f_2)_A$ in the $\VF$-sort be contained in $\DC$; cell decomposition \cite[Theorem~4.6]{mac:mar:ste:weako} holds accordingly. Cells are in general not invariant under coordinate permutations; however, by cell decomposition, an atomic subset of $\VF^n$ must be a cell and must remain so under coordinate permutations.

\begin{lem}\label{open:rv:cons}
In $\mdl R^{\bullet}_{\rv}$, let $\ga \sub \VF^n$ be an atomic open polydisc and $f : \ga \fun \VF$ a definable function. If $f$ is not constant then $f(\ga)$ is an (atomic) open disc; in particular, $\rv \rest f(\ga)$ is always constant.
\end{lem}
\begin{proof}
By atomicity, $f(\ga)$ must be an atomic $\vv$-interval. We proceed by induction on $n$. For the base case $n=1$, suppose for contradiction that $f(\ga)$ is a closed disc (other than a point) or a half thin annulus. By monotonicity, we may assume that $f(\ga)$ is, say, strictly increasing. Then $f^{-1}$ violates Lemma~\ref{Ocon}, contradiction.

For the case $n > 1$, suppose for contradiction again that $f(\ga)$ is a closed disc or a half thin annulus. By the inductive hypothesis, for every $a \in \pr_{1}(\ga)$ there is a maximal open subdisc $\gb_a \sub f(\ga)$ that contains $f(\ga_a)$, similarly for every $a \in \pr_{>1}(\ga)$. It follows that $f(\ga)$ is actually contained in a maximal open subdisc of $f(\ga)$, which is absurd.
\end{proof}

\begin{cor}\label{poly:open:cons}
Let $f : \VF^n \fun \VF$ be a definable function and $\ga \sub \VF^n$ an open polydisc. If $(\rv \circ f) \rest \ga$ is not constant then there is an $\code \ga$-definable nonempty proper subset of $\ga$.
\end{cor}

Here is a strengthening of Lemma~\ref{atom:gam}:

\begin{lem}\label{atom:self}
Let $B \sub \VF^n$ be \LT-type-definable and $\ga = \ga_1 \times \ldots \times \ga_n \sub B$ an open polydisc. Then, for all $a = (a_1, \ldots, a_n)$ and $b = (b_1, \ldots, b_n)$ in $\ga$, there is an immediate automorphism $\sigma$ of $\mmdl$ with $\sigma(a) = b$. Consequently, $\ga$ is $(\code \ga, t)$-atomic for all $t \in \RV$.
\end{lem}
\begin{proof}
To see that the first assertion implies the second, suppose for contradiction that there is an $(\code \ga, t)$-definable nonempty proper subset $A \sub \ga$. Let $a \in A$, $b \in \ga \mi A$, and $\sigma$ be an immediate automorphism of $\mmdl$ with $\sigma(a) = b$. Then $\sigma$ is also an immediate automorphism of $\mmdl$ over $\la \code \ga, t \ra$, contradicting the assumption that $A$ is $(\code \ga, t)$-definable.

For the first assertion, by Lemma~\ref{imm:ext}, it is enough to show that there is an immediate isomorphism $\sigma : \la a \ra \fun \la b \ra$ sending $a$ to $b$. Write
\[
\ga' = \ga_1 \times \ldots \times \ga_{n-1}, \quad a' =  (a_1, \ldots, a_{n-1}),  \quad b' =  (b_1, \ldots, b_{n-1}).
\]
Then, by induction on $n$ and Lemma~\ref{imm:iso}, it is enough to show that, for any immediate isomorphism $\sigma' : \la a' \ra \fun \la b' \ra$ sending $a'$ to $b'$ and any \LT-definable function $f : \VF^{n-1} \fun \VF$,
\[
\rv(a_n - f(a')) = \rv(b_n - \sigma'(f(a'))).
\]
This is clear for the base case $n=1$, since $\ga$ must be disjoint from $\VF(\mdl S)$. For the case $n > 1$, we choose an immediate automorphism of $\mmdl$ extending $\sigma'$, which is still denoted by $\sigma'$; this is possible by Lemma~\ref{imm:ext}. By the inductive hypothesis and Lemma~\ref{open:rv:cons}, $f(\ga') = f(\sigma'(\ga')) = \sigma'(f(\ga'))$ is either a point or an open disc. Since $B$ is \LT-type-definable, it follows that $f(\ga')$ must be disjoint from $\ga_n$ and hence the desired condition is satisfied.
\end{proof}

\begin{cor}\label{part:rv:cons}
Let $A \sub \VF^n$ and $f : A \fun \VF$ be an \LT-definable function. Then there is an \LT-definable finite partition $A_i$ of $A$ such that, for all $i$, if $\ga \sub A_i$ is an open polydisc then $\rv \rest f(\ga)$ is constant and $f(\ga)$ is either a point or an open disc.
\end{cor}
\begin{proof}
For $a \in A$, let $D_a \sub A$ be the \LT-type-definable subset containing $a$. By Lemma~\ref{atom:self}, every open polydisc $\ga \sub D_a$ is $\code \ga$-atomic and hence, by Lemma~\ref{open:rv:cons}, the assertion holds for $\ga$. Then, by compactness, the assertion must hold in a definable subset $A_a \sub A$ that contains $a$; by compactness again, it holds in finitely many definable subsets $A_1, \ldots, A_m$ of $A$ with $\bigcup_i A_i = A$. Then the partition of $A$ generated by $A_1, \ldots, A_m$ is as desired.
\end{proof}

\begin{rem}\label{rem:LT:com}
Clearly the conclusion of Corollary~\ref{part:rv:cons} still holds if we replace ``\LT-definable'' with ``definable'' everywhere therein. Moreover, its proof  works almost verbatim in all situations where we want to partition an \LT-definable set $A \sub \VF^n$ into finitely many \LT-definable pieces $A_i$ such that certain definable property, not necessarily \LT-definable, holds on every open polydisc (or other imaginary elements) contained in $A_i$.
\end{rem}

Here is a variation of Lemma~\ref{atom:self}.

\begin{lem}\label{atom:exp}
Let $\ga \sub \VF^n$ be an $\code \ga$-atomic open polydisc. Let $e \in \VF^{\times}$ with $\abval(e) \gg 0$ (here $\gg$ stands for ``sufficiently larger than''). Then $\ga$  is $(\code \ga, e)$-atomic.
\end{lem}
\begin{proof}
The argument is somewhat similar to that in the proof of Lemma~\ref{atom:self}. We proceed by induction on $n$. Write $\ga = \ga_1 \times \ldots \times \ga_n$ and $\ga' = \ga_1 \times \ldots \times \ga_{n-1}$.
Let $(a' , a_n)$ and $(b', b_n)$ be two points in $\ga' \times \ga_n$. By the inductive hypothesis and Lemma~\ref{atom:self}, there is an immediate isomorphism $\sigma' : \la a', e \ra \fun \la b', e \ra$ with $\sigma'(e) = e$ and $\sigma'(a') = b'$. Thus, it is enough to show that, for all \LT-definable function $f : \VF^{n} \fun \VF$,
\[
\rv(a_n - f(e, a')) = \rv(b_n - \sigma'(f(e, a'))).
\]

Suppose for contradiction that we can always find an $e \in \VF^{\times}$ that is arbitrarily close to $0$ such that $f(e, \ga') \cap \ga_n \neq \0$ (this must hold for some such $f$, for otherwise we are already done by compactness); more precisely, by weak \omin-minimality, without loss of generality, there is an open interval $(0, \epsilon) \sub \VF^+$ such that $f(e, \ga') \cap \ga_n \neq \0$ for all $e \in (0, \epsilon)$.

For each $a' \in \ga'$, let $f_{a'}$ be the $a'$-\LT-definable function on $\VF^+$ given by $b \efun f(b, a')$. By \omin-minimal monotonicity, there is an $\code{\ga'}$-definable function $l : \ga' \fun \VF^+$ such that $f^*_{a'} \coloneqq f_{a'} \rest A_{a'}$ is continuously monotone (of the same kind) for all $a' \in \ga'$, where $A_{a'} \coloneqq (0, l(a'))$. By Lemma~\ref{open:rv:cons}, $l(\ga')$ is either a point or an open disc. Thus the $\vv$-interval $(0, l(\ga'))$ is nonempty, which implies that $f^*_{a'}(e) \in \ga_n$ for some $a' \in \ga'$ and some $e \in A_{a'}$. In that case, we must have that,  for all  $a' \in \ga'$, $\ga_n \sub f^*_{a'}(A_{a'})$ and hence $ f^*_{a'}$ is bijective. By \omin-minimality in the $\Gamma$-sort, $\abval((f^*_{a'})^{-1}(\ga_n))$ has to be a singleton, say, $\beta_{a'}$; in fact, the function given by $a' \efun \beta_{a'}$ has to be constant and hence we may write $\beta_{a'}$ as $\beta$. It follows that, for all $e \in \VF^+$ with $\abval(e) > \beta$, $f(e, \ga') \cap \ga_n = \0$, contradiction.
\end{proof}

Next we come to the issue of finding definable points in definable sets. As we have mentioned above, this is a trivial issue if the space of parameters is not fixed.

\begin{lem}\label{S:def:cl}
The substructure $\mdl S$ is definably closed.
\end{lem}
\begin{proof}
By Lemma~\ref{RV:no:point}, we have $\VF(\dcl( \mdl S)) = \VF(\mdl S)$. Suppose that $t \in \RV$ is definable. By the first sentence of Remark~\ref{rem:RV:weako}, if $\vrv(\RV(\mdl S))$ is nontrivial then $\RV(\mdl S)$ is a model of the reduct of $\TCVF$ to the $\RV$-sort and hence, by quantifier elimination, is an elementary substructure of $\RV$, which implies $t \in \RV(\mdl S)$. On the other hand, if $\vrv(\RV(\mdl S))$ is trivial then $\RV(\mdl S) = \K(\mdl S)$ and it is not hard, though a bit tedious, to check, using quantifier elimination again, that $t \in \K(\mdl S)$.
\end{proof}

If $\mdl S$ is $\VF$-generated and $\Gamma(\mdl S)$ is nontrivial then $\mdl S$ is an elementary substructure and hence every definable set contains a definable point. This, of course, fails if $\mdl S$ carries extra $\RV$-data, by the above lemma. However, we do have:

\begin{lem}\label{clo:disc:bary}
Every definable closed disc $\gb$ contains a definable point.
\end{lem}
\begin{proof}
Suppose for contradiction that $\gb$ does not contain a definable point. Since $\mmdl$ is sufficiently saturated, there is an open disc $\ga$ that is disjoint from $\VF(\mdl S)$ and properly contains $\gb$. Let $a \in \ga \mi \gb$ and $b \in \gb$. Clearly $\rv(c - b) = \rv(c - a)$ for all $c \in \VF(\mdl S)$. As in the proof of Lemma~\ref{atom:self}, there is an immediate automorphism $\sigma$ of $\mmdl$ such that $\sigma(a) = b$. This means that $\gb$ is not definable, which is a contradiction.
\end{proof}

Notice that the argument above does not work if $\gb$ is an open disc.

\begin{cor}\label{open:disc:def:point}
Let $\ga \sub \VF$ be a disc and $A$ a definable subset of $\VF$. If $\ga \cap A$ is a nonempty proper subset of $\ga$ then $\ga$ contains a definable point.
\end{cor}
\begin{proof}
It is not hard to see that, by HNF, if $\ga \cap A$ is a nonempty proper subset of $\ga$ then $\ga$ contains a definable closed disc and hence the claim is immediate by Lemma~\ref{clo:disc:bary}.
\end{proof}

\begin{lem}\label{one:atomic}
Let $A \sub \VF$ be a definable set that contains infinitely many open discs of radius $\beta$. Then one of these discs $\ga$ is $(\code \ga, \beta)$-atomic.
\end{lem}
\begin{proof}
By Lemmas~\ref{atom:gam} and \ref{atom:self}, it is enough to show that some open disc $\ga \sub A$ of radius $\beta$ is contained in a type-definable set. Suppose for contradiction that this is not the case. By Corollary~\ref{open:disc:def:point} and HNF, for every definable set $B \sub A$, we have either $\ga \cap B = \0$ or $\ga \sub B$ for all but finitely many such open discs $\ga \sub A$. Passing to $\mdl R^{\bullet}_{\rv}$ and applying compactness (with the parameter $\beta$), the claim follows.
\end{proof}

\subsection{Contracting from $\VF$ to $\RV$}

We can relate definable sets in $\VF$ to those in $\RV$, specifically, $\RV$-pullbacks, through a procedure called contraction. But a more comprehensive study of the latter will be postponed to the next section.

\begin{defn}[Disc-to-disc]\label{defn:dtdp}
Let $A$, $B$ be two subsets of $\VF$ and $f : A \fun B$ a bijection. We say that $f$ is \emph{concentric} if, for all open discs $\ga \sub A$, $f(\ga)$ is also an open disc; if both $f$ and $f^{-1}$ are concentric then $f$ has the \emph{disc-to-disc property} (henceforth abbreviated as ``dtdp'').

More generally, let $f : A \fun B$ be a bijection between two sets $A$ and $B$, each with exactly one $\VF$-coordinate. For each $(t, s) \in f_{\RV}$, let $f_{t, s} = f \cap (\VF^2 \times (t, s))$,
which is called a \emph{$\VF$-fiber} of $f$. We say that $f$ has \emph{dtdp} if every $\VF$-fiber of $f$ has dtdp.
\end{defn}

We are somewhat justified in not specifying ``open disc'' in the terminology since if $f$ has dtdp then, for all open discs $\ga \sub A$ and all closed discs $\gc \sub \ga$, $f(\gc)$ is also a closed disc. In fact, this latter property is stronger: if $f(\gc)$ is  a closed disc for all closed discs $\gc \sub A$ then $f$ has dtdp. But we shall only be concerned with open discs, so we ask for it directly.

\begin{lem}\label{open:pro}
Let $f : A \fun B$ be a definable bijection between two sets $A$ and $B$, each with exactly one $\VF$-coordinate. Then there is a definable finite partition $A_i$ of $A$ such that each $f \rest A_i$ has dtdp.
\end{lem}
\begin{proof}
By compactness, we may simply assume that $A$ and $B$ are subsets of $\VF$. Then we may proceed exactly as in the proof of Corollary~\ref{part:rv:cons}, using Lemmas~\ref{open:rv:cons} and~\ref{atom:self} (also see Remark~\ref{rem:LT:com}).
\end{proof}

\begin{defn}
Let $A$ be a subset of $\VF^n$. The \emph{$\RV$-boundary} of $A$, denoted by $\partial_{\RV}A$, is the definable subset of $\rv(A)$ such that $t \in \partial_{\RV} A$ if and only if $t^\sharp \cap A$ is a proper nonempty subset of $t^\sharp$. The definable set $\rv(A) \mi \partial_{\RV}A$, denoted by $\ito_{\RV}(A)$, is called the \emph{$\RV$-interior} of $A$.
\end{defn}

Obviously, $A \sub \VF^n$ is an $\RV$-pullback if and only if $\partial_{\RV} A$ is empty. Note that $\partial_{\RV}A$ is in general different from the topological boundary $\partial(\rv(A))$ of $\rv(A)$ in $\RV^n$ and neither one of them includes the other.

\begin{lem}\label{RV:bou:dim}
Let $A$ be a definable subset of $\VF^n$. Then $\dim_{\RV}(\partial_{\RV} A) < n$.
\end{lem}
\begin{proof}
We do induction on $n$. The base case $n=1$ follows immediately from HNF.

We proceed to the inductive step. Since $\partial_{\RV} A_a$ is finite for every $i \in [n]$ and every $a \in \pr_{\tilde i}(A)$, by Corollary~\ref{open:disc:def:point} and compactness, there are a definable finite partition $A_{ij}$ of $\pr_{\tilde i}(A)$ and, for each $A_{ij}$, finitely many definable functions $f_{ijk} : A_{ij} \fun \VF$ such that
\[
\textstyle\bigcup_k \rv(f_{ijk}(a)) = \partial_{\RV} A_a \quad \text{for all } a \in A_{ij}.
\]
By Corollary~\ref{part:rv:cons}, we may assume that if $t^\sharp \sub A_{ij}$ then the restriction $\rv \rest f_{ijk}(t^\sharp)$ is constant. Hence each $f_{ijk}$ induces a definable function $C_{ijk} : \ito_{\RV}(A_{ij}) \fun \RVV$.
Let
\[
\textstyle C = \bigcup_{i, j, k} C_{ijk} \dand B = \bigcup_{i,j} \bigcup_{t \in \partial_{\RV} A_{ij}} \rv(A)_t.
\]
Obviously $\dim_{\RV}(C) < n$. By the inductive hypothesis, for all $A_{ij}$ we have $\dim_{\RV}(\partial_{\RV} A_{ij}) < n-1$. Thus $\dim_{\RV}(B) < n$. Since $\partial_{\RV} A \sub B \cup C$, the claim follows.
\end{proof}

For $(a, t) \in \VF^n \times \RV_0^m$, we write $\rv(a,t)$ to mean $(\rv(a), t)$, similarly for other maps.

\begin{defn}[Contractions]\label{defn:corr:cont}
A function $f : A \fun B$ is \emph{$\rv$-contractible} if there is a (necessarily unique) function $f_{\downarrow} : \rv(A) \fun \rv(B)$, called the \emph{$\rv$-contraction} of $f$, such that
\[
(\rv \rest B) \circ f = f_{\downarrow} \circ (\rv \rest A).
\]
Similarly, it is \emph{$\res$-contractible} (resp.\ \emph{$\vv$-contractible}) if the same holds in terms of $\res$ (resp.\ $\vv$ or $\vrv$, depending on the coordinates) instead of $\rv$.
\end{defn}

The subscripts in these contractions will be written as $\downarrow_{\rv}$, $\downarrow_{\res}$, etc., if they occur in the same context and therefore need to be distinguished from one another notationally.

\begin{lem}\label{fn:alm:cont}
For every definable function $f : \VF^n \fun \VF$ there is a definable set $U \sub \RV^n$ with $\dim_{\RV}(U) < n$ such that $f \rest (\VF^n \mi  U^\sharp)$ is $\rv$-contractible.
\end{lem}
\begin{proof}
By Corollary~\ref{poly:open:cons}, for any $t \in \RV^n$, if $\rv(f(t^\sharp))$ is not a singleton then $t^\sharp$ has a $t$-definable proper subset. By compactness, there is a definable subset $A \sub \VF^n$ such that $t \in \partial_{\RV} A$ if and only if $\rv(f(t^\sharp))$ is not a singleton. So the assertion follows from Lemma~\ref{RV:bou:dim}.
\end{proof}

For any definable set $A$, a property holds \emph{almost everywhere} in $A$ or \emph{for almost every point} in $A$ if it holds away from a definable subset of $A$ of a smaller $\VF$-dimension. This terminology will also be used with respect to other notions of dimension.

\begin{rem}[Regular points]
Let $f : \VF^n \fun \VF^m$ be a definable function. By Lemma~\ref{fun:suba:fun} and \omin-minimal differentiability, $f$ is $C^p$ almost everywhere for all $p$ (see \cite[\S~7.3]{dries:1998}). For each $p$, let $\reg^p(f) \sub \VF^n$ be the definable subset of regular $C^p$-points of $f$. If $p=0$ then we write $\reg(f)$, which is simply the subset of the regular points of $f$.

Assume $n=m$. If $a \in \reg(f)$ and $f$ is $C^1$ in a neighborhood of $a$ then $\reg^1(f)$ contains a neighborhood of $a$ on which the sign of the Jacobian of $f$, which is denoted by $\jcb_{\VF} f$, is constant. If $f$ is locally injective on a definable open subset $A \sub \VF^n$ then $f$ is regular almost everywhere in $A$ and hence, for all $p$, $\dim_{\VF}(A \mi \reg^p(f)) < n$.

By \cite[Theorem~A]{Dries:tcon:97},  the situation is quite similar if $f$ is a (parametrically) definable function of the form $\tor(\alpha)^n \fun \tor(\beta)^m$, $\alpha, \beta \in \absG$, and $\dim_{\VF}$ is replaced by $\dim_{\RV}$, in particular, if $f$ is such a function from $\K^n$ into $\K^m$, or more generally, from $\tor(u)$ into $\tor(v)$, where $u \in \RV^n_{\alpha}$ and $v \in \RV^m_{\beta}$ (see Notation~\ref{rem:K:aff} and Definition~\ref{rem:tor:der}).
\end{rem}

\begin{rem}[$\rv$-contraction of univariate functions]\label{contr:uni}
Suppose that $f$ is a definable function from $\OO^\times$ into $\OO$. By monotonicity, there are a definable finite set $B \sub \OO^\times$ and a definable finite partition of $A \coloneqq \OO^\times \mi B$ into infinite $\vv$-intervals $A_i$ such that both $f$ and $\ddx f$ are quasi-\LT-definable, continuous, and monotone on each $A_i$. If $\rv(A_i)$ is not a singleton then let $U_i \sub \K$ be the largest open interval contained in $\rv(A_i)$. Let
\[
A^*_i = U_i^\sharp, \quad U = \textstyle{\bigcup_i U_i}, \quad A^* = U^\sharp, \quad f^* = f \rest A^*.
\]
By Lemma~\ref{fn:alm:cont}, we may refine the partition such that both $f^*$ and $\frac{d}{d x} f^*$ are $\rv$-contractible. By Lemma~\ref{gk:ortho}, $\vv \rest f^*(A^*_i)$ and $\vv \rest \tfrac{d}{d x} f^*(A^*_i)$ must be constant, say $\alpha_i$ and $\beta_i$, respectively. So it makes sense to speak of $\ddx f^*_{\downarrow_{\rv}}$ on each $U_i$, which a priori is not the same as $(\ddx f^*)_{\downarrow_{\rv}}$. Deleting finitely many points from $U$ if necessary, we assume that $f^*_{\downarrow_{\rv}}$, $(\ddx f^*)_{\downarrow_{\rv}}$, and $\ddx f^*_{\downarrow_{\rv}}$ are all continuous monotone functions on each $U_i$.

We claim that $\abs{\beta_i} = \abs{\alpha_i}$ unless $f^*_{\downarrow_{\rv}} \rest U_i$ is constant. Suppose for contradiction that $f^*_{\downarrow_{\rv}} \rest U_i$ is not constant and $\abs{\beta_i} \neq \abs{\alpha_i}$. First examine the case $\abs{\beta_i} < \abs{\alpha_i}$. A moment of reflection shows that, then,  $f^* \rest A^*_i$ would increase or decrease too fast to confine $f^*(A_i^*)$ in $\vv^{-1}(\alpha_i)$. Dually, if $\abs{\beta_i} > \abs{\alpha_i}$ then $f^* \rest A^*_i$ would increase or decrease too slowly to make $f^*_{\downarrow_{\rv}}(U_i)$ contain more than one point. In either case, we have reached a contradiction. Actually, a similar estimate shows that if $\abs{\beta_i} = \abs{\alpha_i} < \infty$ then $f^*_{\downarrow_{\rv}} \rest U_i$ cannot be constant.

Finally, we show that  $\abs{\beta_i} = \abs{\alpha_i}$ implies $(\ddx f^*)_{\downarrow_{\rv}} = \ddx f^*_{\downarrow_{\rv}}$ on $U_i$ (note that if $\abs{\beta_i} > \abs{\alpha_i}$ then $\ddx f^*_{\downarrow_{\rv}} = 0$). Suppose for contradiction that, say,
\[
(\ddx f^*)_{\downarrow_{\rv}}(\rv(a)) > \ddx f^*_{\downarrow_{\rv}}(\rv(a)) > 0
\]
for some $a \in A^*_i$. Then there is an open interval $I \sub U_i$ containing $\rv(a)$ such that $(\ddx f^*)_{\downarrow_{\rv}}(I) > \ddx f^*_{\downarrow_{\rv}}(I)$. It follows that $f^*_{\downarrow_{\rv}}(I)$ is properly contained in $\rv(f^*(I^\sharp)) = f^*_{\downarrow_{\rv}}(I)$, which is absurd. The other cases are similar.
\end{rem}

The higher-order multivariate version is more complicated to state than to prove:

\begin{lem}\label{univar:der:contr}
Let $A  \sub (\OO^\times)^n$ be a definable $\RV$-pullback with $\dim_{\RV}(\rv(A)) = n$ and $f : A \fun \OO$ a definable function. Let $p \in \N^n$ be a multi-index of order $\abs{p} = d$ and $k \in \N$ with $k \gg d$. Suppose that $f$ is $C^k$ and, for all $q \leq p$, $\frac{\partial^q}{\partial x^q} f$ is $\rv$-contractible and its contraction $(\frac{\partial^q}{\partial x^q} f)_{\downarrow_{\rv}}$ is also $C^k$. Then there is a definable set $V \sub \rv(A)$ with $\dim_{\RV}(V) < n$ and $U \coloneqq \rv(A) \mi V$ open such that, for all $a \in U^\sharp$ and all $q' < q \leq p$ with $\abs{q'} + 1 = q$, exactly one of the following two conditions holds:
\begin{itemize}
 \item either $\frac{\partial^{q}}{\partial x^{q}} f(a) = 0$ or $\abval (\frac{\partial^{q'}}{\partial x^{q'}} f(a)) < \abval (\frac{\partial^{q}}{\partial x^{q}} f(a))$,
 \item $(\frac{\partial^{q - q'}}{\partial x^{q - q'}} \frac{\partial^{q'}}{\partial x^{q'}} f)_{\downarrow_{\rv}}(\rv (a)) = \frac{\partial^{q - q'}}{\partial x^{q - q'}}(\frac{\partial^{q'}}{\partial x^{q'}} f)_{\downarrow_{\rv}}(\rv( a)) \neq 0$.
\end{itemize}
If the first condition never occurs then, for all $q \leq p$, we actually have $(\frac{\partial^q}{\partial x^q} f )_{\downarrow_{\rv}} = \frac{\partial^{q}}{\partial x^{q}} f_{\downarrow_{\rv}}$ on $U$. At any rate, for all $q \leq p$, we have $(\frac{\partial^q}{\partial x^q} f )_{\downarrow_{\res}} = \frac{\partial^{q}}{\partial x^{q}} f_{\downarrow_{\res}}$ on $U$.
\end{lem}
\begin{proof}
First observe that, by induction on $d$, it is enough to consider the case $d =1$ and $p = (0, \ldots, 0, 1)$. For each $a \in \pr_{<n}(A)$, by the discussion in Remark~\ref{contr:uni}, there is an $a$-definable finite set $V_{a}$ of $\rv(A)_{\rv(a)}$ such that the assertion holds for the restriction $f \rest (A_a \mi V_{a}^\sharp)$. Let $A^* = \bigcup_{a \in \pr_{<n}(A)} V_{a}^\sharp \sub A$. By Lemma~\ref{RV:bou:dim}, $\dim_{\RV}(\partial_{\RV} A^*) < n$ and hence $\dim_{\RV}(\rv(A^*)) < n$. Therefore, by Lemma~\ref{fn:alm:cont}, there is a definable open set $U \sub \ito(\rv(A) \mi \rv(A^*))$ that is as desired.
\end{proof}

Suppose that $f = (f_1, \ldots, f_m) : A \fun \OO$ is a sequence of definable $\res$-contractible functions, where the set $A$ is as in Lemma~\ref{univar:der:contr}. Let $P(x_1, \ldots, x_m)$ be a partial differential operator with definable $\res$-contractible coefficients $a_i : A \fun \OO$ and $P_{\downarrow_{\res}}(x_1, \ldots, x_m)$ the corresponding operator with $\res$-contracted coefficients $a_{i\downarrow_{\res}} : \res(A) \fun \K$. Note that both $P(f) : A \fun \OO$ and $P_{\downarrow_{\res}}(f_{\downarrow_{\res}}) : \res(A) \fun \K$ are defined almost everywhere. By Lemma~\ref{univar:der:contr}, such an operator $P$ almost commutes with $\res$:

\begin{cor}\label{rv:op:comm}
For almost all $t \in \rv(A)$ and all $a \in t^\sharp$,
\[
\res(P(f)(a)) = P_{\downarrow_{\res}}(f_{\downarrow_{\res}})(\res(a)).
\]
\end{cor}

\begin{cor}
Let $U$, $V$ be definably connected subsets of $(\K^+)^n$ and $f : U^\sharp \fun V^\sharp$ a definable $\res$-contractible function. Suppose that $f_{\downarrow_{\res}} : U \fun V$ is continuous and locally injective. Then there is a definable subset $U^* \sub U$ of $\RV$-dimension $< n$ such that the sign of $\jcb_{\VF} f$ is constant on $(U \mi U^*)^\sharp$.
\end{cor}
\begin{proof}
This follows immediately from Corollary~\ref{rv:op:comm} and \cite[Theorem~3.2]{pet:star:otop}.
\end{proof}

\begin{lem}\label{atom:type}
In $\xmdl$, let $\ga \sub \VF$ be an atomic subset and $f : \ga \fun \VF$ a definable injection. Then $\ga$ and $f(\ga)$ must be of the same one of the four possible forms (see Remark~\ref{rem:type:atin}).
\end{lem}
\begin{proof}
This is trivial if $\ga$ is a point. The case of $\ga$ being an open disc is covered by Lemma~\ref{open:rv:cons}. So we only need to show that if $\ga$ is a closed disc then $f(\ga)$ cannot be a half thin annulus. We shall give two proofs. The first one works only when $T$ is polynomially bounded, but is more intuitive and much simpler.

Suppose that $T$ is polynomially bounded. Suppose for contradiction that $\code \ga$ is of the form $\tor(\goedel \gm)$ for some $\goedel \gm \in \RV_{\gamma}$ and $\goedel{f(\ga)}$ is of the form $\tor^+(\goedel \gn)$ for some $\goedel \gn \in \RV_{\delta}$. By Lemma~\ref{open:pro} and monotonicity, $f$ induces an increasing (or decreasing, which can be handled similarly) bijection $f_{\downarrow} : \tor(\goedel \gm) \fun \tor^+(\goedel \gn)$.
In fact,  for all $p \in \N$,
\[
\tfrac{d^p}{d x^p} f_{\downarrow} : \tor(\goedel \gm) \fun \tor^{+}(\delta - p \gamma)
\]
cannot be constant and hence must be continuous, surjective, and increasing. Using additional parameters, we can translate $f_{\downarrow}$ into a function $\K \fun \K^+$ and this function cannot be polynomially bounded by elementary differential calculus, which is a contradiction.

We move on to the second proof. The argument is essentially the same as that in the proof of \cite[Lemma~3.45]{hrushovski:kazhdan:integration:vf}.

Consider the group
\[
G \coloneqq \aut(\tor(\goedel \gm) / \K) \leq \aut(\xmdl / \K).
\]
Suppose for contradiction that $G$ is finite. Since every $G$-orbit is finite, every point in $\tor(\goedel \gm)$ is $\K$-definable. It follows that there exists a nonconstant definable function $\tor(\goedel \gm) \fun \K$. But this is not possible since $\ga$ is atomic.

Let $\Lambda$ be the group of affine transformations of $\K$, that is, $\Lambda = \K^{\times} \ltimes \K$, where the first factor is the multiplicative group of $\K$ and the second the additive group of $\K$. Every automorphism in $G$ is a $\K$-affine transformation of $\tor(\goedel \gm)$ and hence $G$ is a subgroup of $\Lambda$. For each $\K$-definable relation $\phi$ on $\tor(\goedel \gm)$, let $G_{\phi} \sub \Lambda$ be the definable subgroup of $\K$-affine transformations that preserve $\phi$. So $G = \bigcap_{\phi} G_{\phi}$. Since there is no infinite descending chain of definable subgroups of $\Lambda$, we see that $G$ is actually an infinite definable group. Then we may choose two nontrivial automorphisms $g, g' \in G$ whose fixed points are distinct. It follows that the commutator of $g$, $g'$ is a translation and hence, by \omin-minimality, $G$ contains all the translations, that is, $\K \leq G$.

By a similar argument, every automorphism in $H \coloneqq \aut(\tor^+(\goedel \gn) / \K)$ is a $\K$-linear transformation of $\tor^+(\goedel \gn)$ and hence $H = \K^+ \leq \K^{\times}$.

Now any definable bijection between $\tor(\goedel \gm)$ and $\tor^+(\goedel \gn)$ would induce a definable group isomorphism $\K \fun \K^+$, that is, an exponential function, which of course contradicts the assumption that $T$ is power-bounded.
\end{proof}

\begin{defn}[$\vv$-affine and $\rv$-affine]\label{rvaffine}
Let $\ga$ be an open disc and $f : \ga \fun \VF$ an injection.
We say that $f$ is \emph{$\vv$-affine} if there is a (necessarily unique) $\gamma \in \Gamma$, called the \emph{shift} of $f$, such that, for all $a, a' \in \ga$,
\[
\abval(f(a) - f(a')) = \gamma + \abval(a - a').
\]
We say that $f$ is \emph{$\rv$-affine} if there is a (necessarily unique) $t \in \RV$, called the \emph{slope} of $f$, such that, for all $a, a' \in \ga$,
\[
\rv(f(a) - f(a')) = t \rv(a - a').
\]
\end{defn}

Obviously $\rv$-affine implies $\vv$-affine. With the extra structure afforded by the total ordering, we can reproduce (an analogue of) \cite[Lemma~3.18]{Yin:int:acvf} with a somewhat simpler proof:

\begin{lem}\label{rv:lin}
In $\xmdl$, let $f : \ga \fun \gb$ be a definable bijection between two atomic open discs. Then $f$ is $\rv$-affine and hence $\vv$-affine with respect to $\rad(\gb) - \rad(\ga)$.
\end{lem}
\begin{proof}
Since $f$ has dtdp by Lemma~\ref{open:pro}, for all $\rad(\ga) < \delta$ and all
\[
\gd \coloneqq \tor(\goedel \gc) \sub \rv_{\delta- \abval(\ga)}(\ga),
\]
it induces a $\goedel \gd$-definable $C^1$ function $f_{\goedel \gd} : \gd \fun \tor(\goedel{f(\gc)})$. The codomain of its derivative $\ddx f_{\goedel \gd}$ can be narrowed down to either $\tor^+(\epsilon - \delta)$ or $\tor^{-}(\epsilon - \delta)$, where $\epsilon = \rad(f(\gc))$. By Lemma~\ref{open:rv:cons}, there is a $t \in \RV$ such that $\ddx f(\ga) \sub t^\sharp$. By Lemma~\ref{atom:gam}, $\ga$ remains atomic over $\delta$. Then, by (an accordingly modified version of) Remark~\ref{contr:uni}, we must have that, for all $\gd$ as above, all $\goedel \gc \in \gd$, and all $a \in \gc$,
\[
\ddx f_{\goedel \gd}(\goedel \gc) = \rv(\ddx f(a)) = t
\]
and hence
\[
\aff_{\goedel{f(\gc)}} \circ f_{\goedel \gd} \circ \aff^{-1}_{\goedel \gc} : \tor(\delta) \fun \tor(\epsilon)
\]
is a linear function given by $u \efun tu$ (see Definition~\ref{rem:tor:der} for the notation). It follows that, for
\begin{itemize}
  \item $a$ and $a'$ in $\ga$,
  \item $\gd$ the smallest closed disc containing $a$ and $a'$,
  \item $\gc$ and $\gc'$ the maximal open subdiscs of $\gd$ containing $a$ and $a'$, respectively,
\end{itemize}
we have
\[
\rv(f(a) - f(a')) = \rv(f(\gc) - f(\gc')) = t \rv(\gc - \gc') = t \rv(a - a').
\]
That is, $f$ is $\rv$-affine. Moreover, it is clear from dtdp that $\abvrv(t) = \rad(\gb) - \rad(\ga)$.
\end{proof}

\section{Grothendieck semirings}\label{sect:groth}

In this section, we define various categories of definable sets and explore the relations between their Grothendieck semirings. The first main result is that the Grothendieck semiring $\gsk \RV[*]$ of the $\RV$-category $\RV[*]$ can be naturally expressed as a tensor product of the Grothendieck semirings of two of its full subcategories $\RES[*]$ and $\Gamma[*]$. The second main result is that there is a natural surjective semiring homomorphism from $\gsk \RV[*]$ onto the Grothendieck semiring $\gsk \VF_*$ of the $\VF$-category $\VF_*$.

\begin{hyp}\label{hyp:gam}
By (the proof of) Lemma~\ref{S:def:cl}, every definable set in $\RV$ contains a definable point if and only if $\Gamma(\mdl S) \neq \pm 1$. Thus, from now on, we shall assume that $\Gamma(\mdl S)$ is nontrivial.
\end{hyp}

\subsection{The categories of definable sets}
As in Definition~\ref{defn:dtdp}, an $\RV$-fiber of a definable set $A$ is a set of the form $A_a$, where $a \in A_{\VF}$. The $\RV$-fiber dimension of $A$ is the maximum of the $\RV$-dimensions of its $\RV$-fibers and is denoted by $\dim^{\fib}_{\RV}(A)$.

\begin{lem}\label{RV:fiber:dim:same}
Suppose that $f : A \fun A'$ is a definable bijection. Then $\dim^{\fib}_{\RV}(A) = \dim^{\fib}_{\RV} (A')$.
\end{lem}
\begin{proof}
Let $\dim^{\fib}_{\RV}(A) = k$ and $\dim^{\fib}_{\RV}(A') =
k'$. For each $a \in \pr_{\VF}(A)$, let $h_{a} : A_a \fun A'_{\VF}$ be the $a$-definable function induced by $f$ and $\pr_{\VF}$. By Corollary~\ref{function:rv:to:vf:finite:image}, the image of $h_{a}$ is finite. It follows that $k \leq k'$. Symmetrically
we also have $k \geq k'$ and hence $k = k'$.
\end{proof}

\begin{defn}[$\VF$-categories]\label{defn:VF:cat}
The objects of the category $\VF[k]$ are the definable sets of $\VF$-dimension $\leq k$ and $\RV$-fiber dimension $0$ (that is, all the $\RV$-fibers are finite). Any definable bijection between two such objects is a morphism of $\VF[k]$. Set $\VF_* = \bigcup_k \VF[k]$.
\end{defn}

\begin{defn}[$\RV$-categories]\label{defn:c:RV:cat}
The objects of the category $\RV[k]$ are the pairs $(U, f)$ with $U$ a definable set in $\RVV$ and $f : U \fun \RV^k$ a definable finite-to-one function. Given two such objects $(U, f)$, $(V, g)$, any definable bijection $F : U \fun V$ is a \emph{morphism} of $\RV[k]$.
\end{defn}

Set $\RV[{\leq} k] = \bigoplus_{i \leq k} \RV[i]$ and $\RV[*] = \bigoplus_{k} \RV[k]$; similarly for the other categories below.

\begin{nota}\label{0coor}
We emphasize that if $(U, f)$ is an object of $\RV[k]$ then $f(U)$ is a subset of $\RV^k$ instead of $\RV_0^k$, while $0$ can occur in any coordinate of $U$. An object of  $\RV[*]$ of the form $(U, \id)$ is often written as $U$.

More generally, if $f : U \fun \RV_0^k$ is a definable finite-to-one function then $(U, f)$ denotes the obvious object of $\RV[{\leq} k]$. Often $f$ will be a coordinate projection (every object in $\RV[*]$ is isomorphic to an object of this form). In that case, $(U, \pr_{\leq k})$ is simply denoted by $U_{\leq k}$ and its class in $\gsk \RV[k]$ by $[U]_{\leq k}$, etc.
\end{nota}

\begin{rem}\label{fintoone}
Alternatively, we could allow only injections instead of finite-to-one functions in defining the objects of $\RV[k]$. Insofar as the Grothendieck semigroup $\gsk \RV[k]$ is concerned, this is not more restrictive in our setting since for any $\bm U \coloneqq (U, f) \in \RV[k]$ there is a definable finite partition $\bm U_i \coloneqq (U_i, f_i)$ of $\bm U$, in other words, $[\bm U] = \sum_i [\bm U_i]$ in $\gsk \RV[k]$, such that each $f_i$ is injective. It is technically more convenient to work with finite-to-one functions, though (for instance, we can take finite disjoint unions).
\end{rem}

In the above definitions and other similar ones below, all morphisms are actually isomorphisms and hence the categories are all groupoids. For the cases $k =0$, the reader should interpret things such as $\RV^0$ and how they interact with other things in a natural way. For instance, $\RV^0$ may be treated as the empty tuple. So the categories $\VF[0]$, $\RV[0]$ are equivalent.

About the position of ``$*$'' in the notation: ``$\VF_*$'' suggests that the category is filtrated and ``$\RV[*]$'' suggests that the category is graded.

\begin{defn}[$\RES$-categories]\label{defn:RES:cat}
The category $\RES[k]$ is the full subcategory of $\RV[k]$ such that $(U, f) \in \RES[k]$ if and only if $\vrv(U)$ is finite.
\end{defn}

\begin{rem}[Explicit description of ${\gsk \RES[k]}$]\label{expl:res}
Let $\RES$ be the category whose objects are the definable sets $U$ in $\RVV$ with $\vrv(U)$ finite and whose morphisms are the definable bijections. The obvious forgetful functor $\RES[*] \fun \RES$ induces a surjective semiring homomorphism $\gsk \RES[*] \fun \gsk \RES$, which is clearly not injective.

The semiring $\gsk \RES$ is actually generated by isomorphism classes $[U]$ with $U$ a set in $\K^+$. By Theorem~\ref{groth:omin}, we have the following explicit description of $\gsk \RES$. Its underlying set is $(0 \times \N) \cup (\N^+ \times \Z)$. For all $(a, b), (c, d) \in \gsk \RES$,
\[
(a, b) + (c, d) = (\max\{a, c\}, b+d), \quad (a, b) \times (c, d) = (a + c, b \times d).
\]
By the computation in \cite{kage:fujita:2006}, the dimensional part is lost in the groupification $\ggk \RES$ of $\gsk \RES$, that is, $\ggk \RES = \Z$, which is of course much simpler than $\gsk \RES$. However, following the philosophy of \cite{hrushovski:kazhdan:integration:vf}, we shall work with Grothendieck semirings whenever possible.

By Lemma~\ref{gk:ortho}, if $(U, f) \in \RES[*]$ then $\vrv(f(U))$ is finite as well. Therefore the semiring $\gsk \RES[*]$ is generated by isomorphism classes $[(U, f)]$ with $f$ a bijection between two sets in $\K^+$. As above, each $\gsk \RES[k]$ may be described explicitly as well. The semigroup $\gsk \RES[0]$ is canonically isomorphic to the semiring $(0, 0) \times \N$. For $k > 0$, the underlying set of $\gsk \RES[k]$ is $\bigcup_{0 \leq i \leq k}((k, i) \times \Z)$, and its semigroup operation is given by
\[
(k, i, a) + (k, i', a') = (k, \max\{i, i'\}, a + a').
\]
Moreover, multiplication in $\gsk \RES[*]$ is given by
\[
(k, i, a) \times (l, j, b) = (k+l, i + j, a \times b).
\]
\end{rem}

\begin{defn}[$\Gamma$-categories]\label{def:Ga:cat}
The objects of the category $\Gamma[k]$ are the finite disjoint unions of definable subsets of $\Gamma^k$. Any definable bijection between two such objects is a \emph{morphism} of $\Gamma[k]$. The category $\Gamma^{c}[k]$ is the full subcategory of $\Gamma[k]$ such that $I \in \Gamma^{c}[k]$ if and only if $I$ is finite.
\end{defn}

Clearly $\gsk \Gamma^c[k]$ is naturally isomorphic to $\N$ for all $k$ and hence $\gsk \Gamma^c[*] \cong \N[X]$.

\begin{nota}\label{nota:RV:short}
We introduce the following shorthand for distinguished elements in the various Grothendieck semigroups and their groupifications (and closely related constructions):
\begin{gather*}
\bm 1_{\K} = [\{1\}] \in \gsk \RES[0], \quad [1] = [(\{1\}, \id)] \in \gsk \RES[1],\\
[\bm T] = [(\K^+, \id)] \in \gsk \RES[1], \quad [\bm A] = 2 [\bm T] +  [1] \in \gsk \RES[1],\\
\bm 1_{\Gamma} = [\Gamma^0] \in \gsk \Gamma[0], \quad [e] = [\{1\}] \in \gsk \Gamma[1], \quad [\bm H] = [(0,1)] \in \gsk \Gamma[1],\\
[\bm P] = [(\RV^{\circ \circ}, \id)] -  [1] \in \ggk \RV[1].
\end{gather*}
Here $\RV^{\circ \circ} = \RV^{\circ \circ}_0 \mi 0$. Note that the interval $\bm H$ is formed in the signed value group $\Gamma$, whose ordering is inverse to that of the value group $\abs \Gamma_\infty$ (recall Remark~\ref{signed:Gam}). The interval $(1, \infty) \sub  \Gamma$ is denoted by $\bm H^{-1}$.

As in~\cite{hrushovski:kazhdan:integration:vf}, the elements $[\bm P]$ and $\bm 1_{\K} + [\bm P]$ in $\ggk \RV[*]$ play special roles in the main construction (see Propositions~\ref{kernel:L} and the remarks thereafter).
\end{nota}

The following lemma is a generality proven elsewhere. It is only needed to prove Lemma~\ref{gam:pulback:mono}.

\begin{lem}\label{gen:mat:inv}
Let $K$ be an integral domain and $M$ a torsion-free $K$-module, the latter is viewed as the main sort of a first-order structure of some expansion of the usual $K$-module language. Let $\gF$ be a class of definable functions in the sort $M$ such that
\begin{itemize}
  \item all the identity functions are in $\gF$,
  \item all the functions in $\gF$ are definably piecewise $K$-linear, that is, they are definably piecewise of the form $x \efun M x + c$, where $M$ is a matrix with entries in $K$ and $c$ is a definable point,
  \item $\gF$ is closed under composition, inversion, composition with $\mgl(K)$-transformations ($K$-linear functions with invertible matrices), and composition with coordinate projections.
\end{itemize}
If $g : D \fun E$ is a bijection in $\gF$, where $D, E \sub M^n$, then $g$ is definably a piecewise $\mgl_n(K)$-transformation.
\end{lem}
\begin{proof}
See \cite[Lemma~2.29]{Yin:int:expan:acvf}.
\end{proof}

\begin{lem}\label{gam:pulback:mono}
Let $g$ be a $\Gamma[k]$-morphism. Then $g$ is definably a piecewise $\mgl_k(\KKK)$-transformation modulo the sign, that is, a piecewise $\mgl_k(\KKK) \times \Z_2$-transformation. Consequently, $g$ is a $\vrv$-contraction (recall Definition~\ref{defn:corr:cont}).
\end{lem}
\begin{proof}
For the first claim, it is routine to check that Lemma~\ref{gen:mat:inv} is applicable to the class of definable functions in the $\abs \Gamma$-sort. The second claim follows from the fact that the natural actions of $\mgl_k(\KKK)$ on $(\RV^+)^k$ and $(\Gamma^+)^k$ commute with the map $\vrv$.
\end{proof}

\begin{rem}\label{why:glz}
In \cite{hrushovski:kazhdan:integration:vf}, $\Gamma[k]$-morphisms are by definition piecewise $\mgl_k(\Z)$-transformations. This is because, in the setting there, the $\vrv$-contractions are precisely the piecewise $\mgl_k(\Z)$-transformations, which form a proper subclass of definable bijections in the $\Gamma$-sort, which in general are piecewise $\mgl_k(\Q)$-transformations.
\end{rem}

\begin{lem}\label{G:red}
For all $I \in \Gamma[k]$ there are finitely many definable sets $H_i \sub \Gamma^{n_i}$ with $\dim_{\Gamma}(H_i) = n_i \leq k$ such that $[I] = \sum_i [H_i] [e]^{k -n_i}$ in $\gsk \Gamma[k]$.
\end{lem}
\begin{proof}
We do induction on $k$. The base case $k = 0$ is trivial. For the inductive step $k > 0$, the claim is also trivial if $\dim_{\Gamma}(I) = k$; so let us assume that $\dim_{\Gamma}(I) < k$. By \cite[Theorem~B]{Dries:tcon:97}, we may partition $I$ into finitely many definable pieces $I_i$ such that each $I_i$ is the graph of a definable function $I'_i \fun \Gamma$, where $I'_i \in \Gamma[k-1]$. So the claim simply follows from the inductive hypothesis.
\end{proof}

\begin{rem}\label{gam:res}
There is a natural map $\Gamma[*] \fun \RV[*]$ given by $I \efun \bm I \coloneqq (I^\sharp, \id)$ (see Notation~\ref{gamma:what}). By Lemma~\ref{gam:pulback:mono}, this map induces a homomorphism $\gsk \Gamma[*] \fun \gsk \RV[*]$ of graded semirings. By \cite[Theorem~A]{Dries:tcon:97} and Theorem~\ref{groth:omin}, this homomorphism restricts to an injective homomorphism $\gsk \Gamma^{c}[*] \fun \gsk \RES[*]$ of graded semirings. There is also a similar semiring homomorphism $\gsk \Gamma^c[*] \fun \gsk \RES$, but it is not injective.
\end{rem}

\begin{ques}
Is the homomorphism $\gsk \Gamma[*] \fun \gsk \RV[*]$ above injective?
\end{ques}

Now, the map from $\gsk \RES[*] \times \gsk \Gamma[*]$ to $\gsk \RV[*]$ naturally determined by the assignment
\[
([(U, f)], [I]) \efun [(U \times I^\sharp, f \times \id)]
\]
is well-defined and is clearly $\gsk \Gamma^{c}[*]$-bilinear. Hence it induces a $\gsk \Gamma^{c}[*]$-linear map
\[
\bb D: \gsk \RES[*] \otimes_{\gsk \Gamma^{c}[*]} \gsk \Gamma[*] \fun \gsk \RV[*],
\]
which is a homomorphism of graded semirings. We shall abbreviate ``$\otimes_{\gsk \Gamma^{c}[*]}$'' as ``$\otimes$'' below. Note that, by the universal mapping property, groupifying a tensor product in the category of $\gsk \Gamma^{c}[*]$-semimodules is the same, up to isomorphism, as taking the corresponding tensor product in the category of $\ggk \Gamma^{c}[*]$-modules. We will show that $\bb D$ is indeed an isomorphism of graded semirings.

\subsection{The tensor expression}

Heuristically, $\RV$ may be viewed as a union of infinitely many one-dimensional vector spaces over $\K$. Weak \omin-minimality states that every definable subset of $\RV$ is nontrivial only within finitely many such  one-dimensional spaces. The tensor expression of $\gsk \RV[*]$ we seek may be thought of as a generalization of this phenomenon to all definable sets in $\RV$.

\begin{lem}\label{resg:decom}
Let $A \sub \RV^k \times \Gamma^l$ be an $\alpha$-definable set, where $\alpha \in \Gamma$. Set $\pr_{\leq k}(A) = U$ and suppose that $\vrv(U)$ is finite. Then there is an $\alpha$-definable finite partition $U_i$ of $U$ such that, for each $i$ and all $t, t' \in U_i$, we have $A_t = A_{t'}$.
\end{lem}
\begin{proof}
By stable embeddedness, for every $t \in U$, $A_t$ is $(\vrv(t), \alpha)$-definable in the $\Gamma$-sort alone. Since $\vrv(U)$ is finite, the assertion simply follows from compactness.
\end{proof}

\begin{lem}\label{gam:tup:red}
Let $\beta$, $\gamma = (\gamma_1, \ldots, \gamma_m)$ be finite tuples in $\Gamma$. If there is a $\beta$-definable nonempty proper subset of $\gamma^\sharp$ then, for some $\gamma_i$ and every $t \in \gamma^\sharp_{\tilde i}$, $\gamma^\sharp_i$ contains a $t$-definable point. Consequently, if $U$ is such a  subset of $\gamma^\sharp$ then either $U$ contains a definable point or there exists a subtuple  $\gamma_* \sub  \gamma$ such that $\pr_{\gamma_*}(U) = \gamma^\sharp_*$, where $\pr_{\gamma_*}$ denotes the obvious coordinate projection, and there is a $\beta$-definable function from $\gamma^\sharp_*$ into $(\gamma \mi \gamma_*)^\sharp$.
\end{lem}
\begin{proof}
For the first claim we do induction on $m$. The base case $m = 1$ simply follows from \omin-minimality in the $\K$-sort and Lemma~\ref{RV:no:point}. For the inductive step $m > 1$, let $U$ be a $\beta$-definable nonempty proper subset of $\gamma^\sharp$. By the inductive hypothesis, we may assume
\[
\{ t \in \pr_{>1}(U) : U_t \neq \gamma^\sharp_1\} = \gamma^\sharp_{> 1}.
\]
Then $\gamma_1$ is as desired.

The second claim follows easily from the first.
\end{proof}


\begin{lem}\label{RV:decom:RES:G}
Let $U \sub \RV^m$ be a definable set. Then there are finitely many definable sets of the form $V_i \times D_i \sub (\K^+)^{k_i} \times \Gamma^{l_i}$ such that $k_i + l_i = m$ for all $i$ and $[U] = \sum_i [V_i \times D_i^\sharp]$ in $\gsk \RV[*]$.
\end{lem}
\begin{proof}
The case $m=1$ is an immediate consequence of weak \omin-minimality in the $\RV$-sort. For the case $m>1$, by Lemma~\ref{gam:tup:red}, compactness, and a routine induction on $m$, over a definable finite partition of $U$, we may assume that $U$ is a union of sets of the form $t \times \gamma^\sharp$, where $t \in (\K^+)^k$, $\gamma \in \Gamma^l$, and $k+l=m$. Then the assertion follows from Lemma~\ref{resg:decom}.
\end{proof}

Let $Q$ be a set of parameters in $\mdl R^{\bullet}_{\rv}$. We say that a $Q$-definable set $I \sub \Gamma^m$ is \emph{$Q$-reducible} if $I^\sharp$ is $Q$-definably bijective to $\K^+ \times I_{\tilde i}^\sharp$, where $i \in [m]$ and $I_{\tilde i} = \pr_{\tilde i}(I)$. For every $t \in (\K^+)^{n}$ and every $\alpha \in \Gamma^m$, $\alpha$ is $(t,\alpha)$-reducible if and only if, by  Lemma~\ref{gam:tup:red}, there is a $(t,\alpha)$-definable nonempty proper subset of $\alpha^\sharp$ if and only if, by  Lemma~\ref{gam:tup:red} again, there is an $\alpha$-definable set $U \sub (\K^+)^{n}$ containing $t$ such that $\alpha$ is $(u,\alpha)$-reducible for every $u \in U$ if and only if, by \omin-minimality in the $\K$-sort and Lemma~\ref{RV:no:point}, $\alpha$ is $\alpha$-reducible.

We say that a definable set $A$ in $\RV$ is \emph{$\Gamma$-tamped} of \emph{height} $l$ if there are $U \in \RES[k]$ and $I \in \Gamma[l]$ with $\dim_{\Gamma}(I) = l$ such that $A = U \times I^\sharp$. In that case, there is only one way to write $A$ as such a product, and if $B = V \times J^\sharp \sub A$ is also $\Gamma$-tamped then the coordinates occupied by $J^\sharp$ are also occupied by $I^\sharp$, in particular, $\dim_{\Gamma}(J) = l$ if and only if $V \sub U$ and $J \sub I$.

\begin{lem}\label{Gtamp}
Let $A = U \times I^\sharp$, $B = V \times J^\sharp$ be $\Gamma$-tamped sets of the same height $l$, where $U$, $V$ are sets in $\K^+$. Let $f$ be a definable bijection whose domain contains $A$ and whose range contains $ B$. Suppose that $B \mi f( A)$, $A \mi f^{-1}(B)$ do not have $\Gamma$-tamped subsets of height $l$. Then there are finitely many $\Gamma$-tamped sets $A_i = U_i \times I_i^\sharp \sub U \times I^\sharp$ and $B_i = V_i \times J_i^\sharp \sub V \times J^\sharp$ such that
\begin{itemize}
 \item $ A \mi \bigcup_i  A_i$ and $ B \mi \bigcup_i  B_i$ do not have $\Gamma$-tamped subsets of height $l$,
  \item each restriction $f \rest A_i$ is of the form $p_i \times q_i$, where $p_i : U_i \fun V_i$, $q_i : I_i^\sharp \fun J_i^\sharp$ are bijections and the latter $\vrv$-contracts to a $\Gamma[*]$-morphism $q_{i \downarrow} : I_i \fun J_i$.
\end{itemize}
\end{lem}

Let $t \times \alpha^\sharp \sub A$. If $t \times \alpha^\sharp \sub A \mi f^{-1}(B)$ then, by Lemma~\ref{resg:decom}, it is contained in a definable set  $U' \times I'^\sharp \sub A \mi f^{-1}(B)$ with $U' \sub U$ and $I' \sub I$. Since $A \mi f^{-1}(B)$ does not have $\Gamma$-tamped subsets of height $l$, we must have $\dim_{\Gamma}(I') < l$. It follows from (the proof of) Lemma~\ref{gam:red:K} that $I'$ is piecewise reducible, which implies that $\alpha$ is $\alpha$-reducible. At any rate, if $\alpha$ is $(t,\alpha)$-reducible then $\alpha$ is $\alpha$-reducible and hence there is a reducible subset of $I$ that contains $\alpha$.

\begin{proof}
Remove all the reducible subsets of $I$ from $I$ and call the resulting set $\bar I$; similarly for $\bar J$. Then, for all $t \in U$ and all $\alpha \in \bar I$, $f(t \times \alpha^\sharp)$ must be contained in a set of the form $s \times \beta^\sharp$, for otherwise it would have a $(t,\alpha)$-definable nonempty proper subset and hence would be $(t,\alpha)$-reducible. In fact, $f(t \times \alpha^\sharp) = s \times \beta^\sharp$, for otherwise $\beta$ is $(t,\alpha)$-reducible and hence, by \omin-minimality in the $\K$-sort and the assumption $\dim_{\Gamma}(I) = \dim_{\Gamma}(J) = l$, a $(t,\alpha)$-definable subset of $\alpha^\sharp$ can be easily constructed. For the same reason, we must actually have $\beta \in \bar J$. It follows that $f(U \times \bar I^\sharp) =  V \times \bar J^\sharp$. Then, by compactness, there are finitely many reducible subsets $I_i$ of $I$ such that, for all $t \in U$ and all $\alpha \in I_* = I \mi \bigcup_i I_i$, $f(t \times \alpha^\sharp) = s \times \beta^\sharp$ for some $s \in V$ and $\beta \in J$. Applying Lemma~\ref{resg:decom} to (the graph of) the function on $U \times I_*$ induced by $f$, the lemma follows.
\end{proof}

\begin{prop}\label{red:D:iso}
$\bb D$ is an isomorphism of graded semirings.
\end{prop}
\begin{proof}
Surjectivity of $\bb D$ follows immediately from Lemma~\ref{RV:decom:RES:G}. For injectivity, let $\bm U_i \coloneqq (U_i, f_i)$, $\bm V_j \coloneqq (V_j, g_j)$ be objects in $\RES[*]$ and $I_i$, $J_j$ objects in $\Gamma[*]$ such that $\bb D([\bm U_i] \otimes [I_i])$, $\bb D([\bm V_j] \otimes [J_j])$ are objects in $\gsk \RV[l]$ for all $i$, $j$. Set
\[
\textstyle M_i = U_i \times I_i^\sharp, \quad N_i = V_j \times J_j^\sharp, \quad M = \biguplus_i M_i, \quad N = \biguplus_j N_j.
\]
Suppose that there is a definable bijection $f : M \fun N$. We need to show
\[
\textstyle \sum_i [\bm U_i] \otimes [I_i] = \sum_j [\bm V_j] \otimes [J_j].
\]
By Lemma~\ref{gam:red:K}, we may assume that all $M_{i}$, $N_{j}$ are $\Gamma$-tamped. By \omin-minimal cell decomposition, without changing the sums, we may assume that each $U_i$ is a disjoint union of finitely many copies of $(\K^+)^i$ and thereby re-index $M_i$ more informatively as $M_{i, m} = U_i \times I_m^\sharp$, where $I_m$ is an object in $\Gamma[m]$; similarly each $N_j$ is re-indexed as $N_{j, n}$. By Lemma~\ref{dim:cut:gam}, the respective maximums of the numbers $i+m$, $j+n$ are the $\RV$-dimensions of $M$, $N$ and hence must be equal; it is denoted by $p$. Let $q$ be the largest $m$ such that $i + m = p$ for some $M_{i, m}$ and $q'$ the largest $n$ such that $j + n = p$ for some $N_{j, n}$. It is not hard to see that we may arrange $q = q'$.

We now proceed by induction on $q$. The base case $q=0$ is rather trivial. For the inductive step, by Lemma~\ref{Gtamp}, we see that certain products contained in $M_{p-q, q}$, $N_{p-q, q}$ give rise to the same sum and the inductive hypothesis may be applied to the remaining portions.
\end{proof}


We may view $\Gamma$ as a double cover of $\abs \Gamma$ via the identification $\Gamma / {\pm 1} = \abs \Gamma$. Consequently we can associate two Euler characteristics $\chi_{\Gamma,g}$, $\chi_{\Gamma, b}$ with the $\Gamma$-sort, induced by those on $|\Gamma|$ (see \cite{kage:fujita:2006} and also~\cite[\S~ 9]{hrushovski:kazhdan:integration:vf}). They are distinguished by
\[
\chi_{\Gamma, g}(\bm H) = \chi_{\Gamma, g}(\bm H^{-1}) = -1 \quad \text{and} \quad \chi_{\Gamma, b}(\bm H) = \chi_{\Gamma, b}(\bm H^{-1}) = 0.
\]
Similarly, there is an Euler characteristic $\chi_{\K}$ associated with the $\K$-sort (there is only one). We shall denote all of these Euler characteristics simply by $\chi$ if no confusion can arise. Using these $\chi$ and the groupification of $\bb D$ (also denoted by $\bb D$), we can construct various retractions from the Grothendieck ring $\ggk \RV[*]$ to (certain localizations of) the Grothendieck rings $\ggk \RES[*]$ and $\ggk \Gamma[*]$.

\begin{lem}\label{gam:euler}
The Euler characteristics induce naturally three graded ring homomorphisms:
\[
\mdl E_{\K} : \ggk \RES[*] \fun \Z[X] \quad \text{and} \quad \mdl E_{\Gamma, g}, \mdl E_{\Gamma, b} : \ggk \Gamma[*] \fun \Z[X].
\]
\end{lem}
\begin{proof}
For $U \in \RES[k]$ and $I \in \Gamma[k]$, we set $\mdl E_{\K, k}([U]) = \chi(U)$ (see Remark~\ref{omin:res}) and $\mdl E_{\Gamma, k}([I]) = \chi(I)$. These maps are well-defined and they induce graded ring homomorphisms $\mdl E_{\K} \coloneqq \sum_k \mdl E_{\K, k} X^k$ and $\mdl E_{\Gamma} \coloneqq \sum_k \mdl E_{\Gamma, k} X^k$ as desired.
\end{proof}

By the computation in \cite{kage:fujita:2006}, $\ggk \Gamma[*]$ is canonically isomorphic to the graded ring
\[
\textstyle \Z[X, Y^{(2)}] \coloneqq \Z \oplus \bigoplus_{i \geq 1} (\Z[Y]/(Y^2+Y))X^i,
\]
where $YX$ represents the class $[\bm H] = [\bm H^{-1}]$ in $\ggk \Gamma[1]$. Thus $\mdl E_{\Gamma, g}$, $\mdl E_{\Gamma, b}$ are also given by
\[
\Z[X, Y^{(2)}] \two^{Y \efun -1}_{Y \efun 0} \Z[X].
\]

\begin{rem}[Explicit description of ${\ggk \RV[*]}$]\label{rem:poin}
Of course, $\mdl E_{\K}$ is actually an isomorphism. The  homomorphism $\gsk \Gamma^{c}[*] \fun \gsk \RES[*]$ in Remark~\ref{gam:res} and $\mdl E_{\K}$ then induce an isomorphism $\mdl E_{\K^c} : \ggk \Gamma^{c}[*] \fun \Z[X]$. But this isomorphism is different from the groupification $\mdl E_{\Gamma^c}$ of the canonical isomorphism $\gsk \Gamma^{c}[*] \cong \gsk \N[*]$. This latter isomorphism $\mdl E_{\Gamma^c}$ is also induced by $\mdl E_{\Gamma, g}$, $\mdl E_{\Gamma, b}$ (the two homomorphisms agree on $\ggk \Gamma^{c}[*]$). They are distinguished by $\mdl E_{\K^c}([e]) = -X$ and $\mdl E_{\Gamma^c}([e]) = X$. We have a commutative diagram
\[
\bfig
 \hSquares(0,0)/<-`->`->`->`->`<-`->/[{\ggk \RES[*]}`{\ggk \Gamma^{c}[*]}`{\ggk \Gamma[*]}`\Z[X]`\Z[X]`{\Z[X, Y^{(2)}]}; ``\mdl E_{\K}`\mdl E_{\Gamma^c}`\cong`\tau`]
\efig
\]
where $\tau$ is the involution determined by $X \efun -X$. The graded ring
\[
\Z[X] \otimes_{\Z[X]} \Z[X, Y^{(2)}]
\]
may be identified with $\Z[X, Y^{(2)}]$ via the isomorphism given by $x \otimes y \efun \tau(x)y$. Consequently, by Proposition~\ref{red:D:iso}, there is a graded ring isomorphism
\[
\ggk \RV[*] \to^{\sim} \Z[X, Y^{(2)}] \quad \text{with} \quad \bm 1_{\K} + [\bm P] \efun 1 + 2YX + X.
\]
Setting
\[
 \Z^{(2)}[X] = \Z[X, Y^{(2)}] / (1 + 2YX + X),
\]
we see that there is a canonical ring isomorphism
\[
\bb E_{\Gamma}: \ggk \RV[*] / (\bm 1_{\K} + [\bm P]) \to^{\sim} \Z^{(2)}[X].
\]
There are exactly two ring homomorphisms $\Z^{(2)}[X] \fun \Z$ determined by the assignments $Y \efun -1$ and $Y \efun 0$ or, equivalently, $X \efun 1$ and $X \efun -1$. Combining these with $\bb E_{\Gamma}$, we see that there are exactly two ring homomorphisms
\[
\bb E_{\Gamma,g},  \bb E_{\Gamma,b}: \ggk \RV[*] / (\bm 1_{\K} + [\bm P]) \fun \Z.
\]
\end{rem}

\begin{prop}\label{prop:eu:retr:k}
There are two ring homomorphisms
\[
\bb E_{\K, g}: \ggk \RV[*] \fun \ggk \RES[*][[\bm A]^{-1}] \quad \text{and} \quad \bb E_{\K, b}: \ggk \RV[*] \fun \ggk \RES[*][[1]^{-1}]
\]
such that
\begin{itemize}
  \item their ranges are precisely the zeroth graded pieces of their respective codomains,
  \item $\bm 1_{\K} + [\bm P]$ vanishes under both of them,
  \item for all $x \in \ggk \RES[k]$, $\bb E_{\K, g} (x) = x [\bm A]^{-k}$ and $\bb E_{\K, b}(x) = x [1]^{-k}$.
\end{itemize}
\end{prop}
\begin{proof}
We first define, for each $n$, a homomorphism
\[
\bb E_{g, n}: \ggk \RV[n] \fun \ggk \RES[n]
\]
as follows. By Proposition~\ref{red:D:iso}, there is an isomorphism
\[
\textstyle \bb D_n : \bigoplus_{i + j = n} \ggk \RES[i] \otimes \ggk \Gamma[j] \to^{\sim} \ggk \RV[n].
\]
Let the group homomorphism $\mdl E_{g, j} : \ggk \Gamma[j] \fun \Z$ be defined as in Lemma~\ref{gam:euler}, using $\chi_{\Gamma, g}$. Let
\[
E_{g}^{i, j}: \ggk \RES[i] \otimes \ggk \Gamma[j] \fun \ggk \RES[i {+} j]
\]
be the group homomorphism determined by $x \otimes y \efun \mdl E_{g, j}(y) x [\bm T]^{j}$. Let
\[
\textstyle E_{g, n} = \sum_{i + j = n} E_{g}^{i, j} \dand \bb E_{g, n} = E_{g, n} \circ \bb D_n^{-1}.
\]

Note that, due to the presence of the tensor $\otimes_{\ggk \Gamma^{c}[*]}$ and the replacement of $y$ with $\mdl E_{g, j}(y) [\bm T]^{j}$, there is this issue of compatibility between the various components of $E_{g, n}$. In our setting, this is easily resolved since all definable bijections are allowed in $\Gamma[*]$ and hence $\gsk \Gamma^c[*]$ is generated by isomorphism classes of the form $[e]^k$. In the setting of \cite{hrushovski:kazhdan:integration:vf}, however, one has to pass to a quotient ring to achieve compatibility (see Remark~\ref{why:glz} and also \cite[\S~2.5]{hru:loe:lef}).

Now, it is straightforward to check the equality
\[
\bb E_{g, n}(x)\bb E_{g, m}(y) = \bb E_{g, n+m}(xy).
\]

The group homomorphisms $\tau_{m, k} : \ggk \RES[m] \fun \ggk \RES[m{+}k]$ given by $x \efun x [\bm A]^k$ determine a colimit system and the group homomorphisms
\[
\textstyle\bb E_{g, \leq n} \coloneqq \sum_{m \leq n} \tau_{m, n-m} \circ \bb E_{g, m} : \ggk \RV[{\leq} n] \fun \ggk \RES[n]
\]
determine a homomorphism of colimit systems. Hence we have a ring homomorphism:
\[
\colim{n} \bb E_{g, \leq n} : \ggk \RV[*] \fun \colim{\tau_{n, k}} \ggk \RES[n].
\]
For all $n \geq 1$ we have
\[
\bb E_{g, \leq n}(\bm 1_{\K} + [\bm P]) = [\bm A]^n - 2[\bm T][\bm A]^{n-1} - [1] [\bm A]^{n-1} = 0.
\]
This yields the desired homomorphism $\bb E_{\K, g}$ since the colimit in question can be embedded into the zeroth graded piece of $\ggk \RES[*][[\bm A]^{-1}]$.

The construction of $\bb E_{\K, b}$ is completely analogous, with $[\bm A]$ replaced by $[1]$ and $\chi_{\Gamma, g}$ by $\chi_{\Gamma, b}$.
\end{proof}

Since the zeroth graded pieces of both $\ggk \RES[*][[\bm A]^{-1}]$ and $\ggk \RES[*][[1]^{-1}]$ are canonically isomorphic to $\Z$, the homomorphisms $\bb E_{\K, g}$, $\bb E_{\K, b}$ are just the homomorphisms $\bb E_{\Gamma, g}$, $\bb E_{\Gamma, b}$ in Remark~\ref{rem:poin}, more precisely, $\bb E_{\K, g} = \bb E_{\Gamma, g}$ and $\bb E_{\K, b} = \bb E_{\Gamma, b}$.

\section{Generalized Euler characteristic}

From here on, our discussion will be of an increasingly formal nature. Many statements are exact copies of those in \cite{Yin:special:trans, Yin:int:acvf, Yin:int:expan:acvf} and often the same proofs work, provided that the auxiliary results are replaced by the corresponding ones obtained above. For the reader's convenience, we will write down all the details.

\subsection{Special bijections}

Our first task is to connect $\gsk \VF_*$ with $\gsk \RV[*]$, more precisely, to establish a surjective homomorphism $\gsk \RV[*] \fun \gsk \VF_*$. Notice the direction of the arrow. The main instrument in this endeavor  is special bijections.

\begin{conv}\label{conv:can}
We reiterate \cite[Convention~2.32]{Yin:int:expan:acvf} here, with a different terminology, since this trivial-looking convention is actually quite crucial for understanding the discussion below, especially the parts that involve special bijections. For any set $A$, let
\[
\can(A) = \{(a, \rv(a), t) : (a, t) \in A \text{ and } a \in \pr_{\VF}(A)\}.
\]
The natural bijection $\can : A \fun \can(A)$ is called the \emph{regularization} of $A$. We shall tacitly substitute $\can(A)$ for $A$  if it is necessary or is just more convenient. Whether this substitution has been performed or not should be clear in context (or rather, it is always performed).
\end{conv}

\begin{defn}[Special bijections]\label{defn:special:bijection}
Let $A$ be a (regularized) definable set whose first coordinate is a $\VF$-coordinate (of course nothing is special about the first $\VF$-coordinate, we choose it simply for notational ease). Let $C \sub \RVH(A)$ be an $\RV$-pullback (see Definition~\ref{defn:disc}) and
\[
\lambda: \pr_{>1}(C \cap A) \fun \VF
\]
a definable function whose graph is contained in $C$. Recall Notation~\ref{nota:tor}. Let
\[
\textstyle C^{\sharp} = \bigcup_{x \in \pr_{>1} (C)} \MM_{\abvrv(\pr_1(x_{\RV}))} \times x \dand \RVH(A)^{\sharp} = C^{\sharp} \uplus (\RVH(A) \mi C),
\]
where $x_{\RV} = \pr_{\RV}(x)$. The \emph{centripetal transformation $\eta : A \fun \RVH(A)^{\sharp}$ with respect to $\lambda$} is defined by
\[
\begin{cases}
  \eta (a, x) = (a - \lambda(x), x), & \text{on } C \cap A,\\
  \eta = \id, & \text{on } A \mi C.
\end{cases}
\]
Note that $\eta$ is injective. The inverse of $\eta$ is naturally called the \emph{centrifugal transformation with respect to $\lambda$}. The
function $\lambda$ is referred to as the \emph{focus} of $\eta$ and the $\RV$-pullback $C$ as the \emph{locus} of $\lambda$ (or $\eta$).

A \emph{special bijection} $T$ on $A$ is an alternating composition of centripetal transformations and regularizations. By Convention~\ref{conv:can}, we shall only display the centripetal transformations in such a composition. The \emph{length} of such a special bijection $T$, denoted by $\lh(T)$, is the number of centripetal transformations in $T$. The range of $T$ is sometimes denoted by $A^{\flat}$.
\end{defn}

For functions between sets that have only one $\VF$-coordinate, composing with special bijections on the right and inverses of special bijections on the left obviously preserves dtdp.

\begin{lem}\label{inverse:special:dim:1}
Let $T$ be a special bijection on $A \sub \VF \times \RV^m$ such that $A^{\flat}$ is an $\RV$-pullback. Then there is a definable function $\epsilon : \pr_{\RV} (A^{\flat}) \fun \VF$ such that, for every $\RV$-polydisc $\gp = t^\sharp \times s \sub A^{\flat}$,
$(T^{-1}(\gp))_{\VF} = t^\sharp + \epsilon(s)$.
\end{lem}
\begin{proof}
It is clear that $\gp$ is the image of an open polydisc $\ga \times r \sub A$.  Let $T'$ be $T$ with the last centripetal transformation deleted. Then $T'(\ga \times r)$ is also an open polydisc $\ga' \times r'$. The range of the focus map of $\eta_n$ contains a point in the smallest closed disc containing $\ga'$. This point can be transported back by the previous focus maps to a point in the smallest closed disc containing $\ga$. The lemma follows easily from this observation.
\end{proof}

Note that, since $\dom(\epsilon) \sub \RV^l$ for some $l$, by Corollary~\ref{function:rv:to:vf:finite:image}, $\ran(\epsilon)$ is actually finite.

A definable set $A$ is called a \emph{deformed $\RV$-pullback} if there is a special bijection $T$ on $A$ such that $A^{\flat}$ is an $\RV$-pullback.

\begin{lem}\label{simplex:with:hole:rvproduct}
Every definable set $A \sub \VF \times \RV^m$ is a deformed $\RV$-pullback.
\end{lem}
\begin{proof}
By compactness and HNF  this is immediately reduced to the situation where $A \sub \VF$ is contained in an $\RV$-disc and is a $\vv$-interval with end-discs $\ga$, $\gb$. This may be further divided into several cases according to whether $\ga$, $\gb$ are open or closed discs and whether the ends of $A$ are open or closed. In each of these cases, Lemma~\ref{clo:disc:bary} is applied in much the same way as its counterpart is applied in the proof of \cite[Lemma~4.26]{Yin:QE:ACVF:min}. It is a tedious exercise and is left to the reader.
\end{proof}

Here is an analogue of \cite[Theorem~5.4]{Yin:special:trans} (see also \cite[Theorem~4.25]{Yin:int:expan:acvf}):

\begin{thm}\label{special:term:constant:disc}
Let $F(x) = F(x_1, \ldots, x_n)$ be an $\lan{T}{}{}$-term. Let $u \in \RV^n$ and $R : u^\sharp \fun A$ be a special bijection.  Then there is a special bijection $T : A \fun A^\flat$ such that $F \circ R^{-1} \circ T^{-1}$ is $\rv$-contractible. In a commutative diagram,
\[
\bfig
  \square(0,0)/`->`->`->/<1500,400>[A^\flat`\VF`\rv(A^\flat)`\RV_0;
  `\rv`\rv`(F \circ R^{-1} \circ T^{-1})_{\downarrow}]
  \morphism(0,400)<500,0>[A^\flat`A; T^{-1}]
  \morphism(500,400)<500,0>[A`u^\sharp; R^{-1}]
  \morphism(1000,400)<500,0>[u^\sharp`\VF; F]
 \efig
\]
\end{thm}
\begin{proof}
First observe that if the assertion holds for one $\lan{T}{}{}$-term then it holds simultaneously for any finite number of $\lan{T}{}{}$-terms, since $\rv$-contractibility is preserved by further special bijections on $A^\flat$. We do induction on $n$. For the base case $n=1$, by Corollary~\ref{part:rv:cons} and Remark~\ref{rem:LT:com}, there is a definable finite partition $B_i$ of $u^\sharp$ such that, for all $i$, if $\ga \sub B_i$ is an open disc then $\rv \rest F(\ga)$ is constant. By consecutive applications of Lemma~\ref{simplex:with:hole:rvproduct}, we obtain a special bijection $T$ on $A$ such that each $(T \circ R) (B_i)$ is an $\RV$-pullback. Clearly $T$ is as required.

For the inductive step, we may concentrate on a single $\RV$-polydisc $\gp = v^\sharp \times (v, r) \sub A$. Let $\phi(x, y)$ be a quantifier-free formula that defines the function $\rv \circ f$. Recall Convention~\ref{topterm}. Let $G_{i}(x)$ enumerate the top $\lan{T}{}{}$-terms of $\phi$. For  $a \in v_1^\sharp$, write $G_{i,a} = G_{i}(a, x_2, \ldots, x_n)$. By the inductive hypothesis, there is a special bijection $R_{a}$ on $(v_2, \ldots, v_n)^\sharp$ such that every $G_{i,a} \circ R_a^{-1}$ is $\rv$-contractible. Let $U_{k, a}$ enumerate the loci of the components of $R_{a}$ and $\lambda_{k, a}$ the corresponding focus maps. By compactness,
\begin{itemize}
  \item for each $i$, there is a quantifier-free formula $\psi_i$ such that $\psi_i(a)$ defines $(G_{i,a} \circ R_a^{-1})_{\downarrow}$,
  \item there is a quantifier-free formula $\theta$ such that $\theta(a)$ determines the sequence $\rv(U_{k, a})$ and the $\VF$-coordinates targeted by $\lambda_{k, a}$.
\end{itemize}
Let $H_{j}(x_1)$ enumerate the top $\lan{T}{}{}$-terms of the formulas $\psi_i$, $\theta$. Applying the inductive hypothesis again, we obtain a special bijection $T_1$ on $v_1^\sharp$ such that every $H_{j} \circ T_1^{-1}$ is $\rv$-contractible. This means that, for every $\RV$-polydisc $\gq \sub T_1(v_1^\sharp)$ and all $a_1, a_2 \in T_1^{-1}(\gq)$,
\begin{itemize}
  \item the formulas $\psi_i(a_1)$, $\psi_i(a_2)$ define the same $\rv$-contraction,
  \item the special bijections $R_{a_1}$, $R_{a_2}$ may be glued together in the obvious sense to form one special bijection on $\{a_1, a_2\} \times (v_2, \ldots, v_n)^\sharp$.
\end{itemize}
Consequently, $T_1$ and $R_{a}$ naturally induce a special bijection $T$ on $\gp$ such that every $G_{i} \circ T^{-1}$ is $\rv$-contractible. This implies that $F \circ R^{-1} \circ T^{-1}$ is $\rv$-contractible and hence $T$ is as required.
\end{proof}

\begin{cor}\label{special:bi:term:constant}
Let $A \sub \VF^n$ be a definable set and $f : A \fun \RV^m$ a definable function. Then there is a special bijection $T$ on $A$ such that $A^\flat$ is an $\RV$-pullback and the function $f \circ T^{-1}$ is $\rv$-contractible.
\end{cor}
\begin{proof}
By compactness, we may assume that $A$ is contained in an $\RV$-polydisc $\gp$. Let $\phi$ be a quantifier-free formula that defines $f$. Let $F_i(x, y)$ enumerate the top $\lan{T}{}{}$-terms of $\phi$. For $s \in \RV^{m}$, let $F_{i,  s} = F_{i}(x, s)$. By Theorem~\ref{special:term:constant:disc}, there is a special bijection $T$ on $\gp$ such that each function $F_{i, s} \circ T^{-1}$ is contractible. This means that, for each $\RV$-polydisc $\gq \sub T(\gp)$,
\begin{itemize}
  \item either $T^{-1}(\gq) \sub A$ or $T^{-1}(\gq) \cap A = \0$,
  \item if $T^{-1}(\gq) \sub A$ then $(f \circ T^{-1})(\gq)$ is a singleton.
\end{itemize}
So $T \rest A$ is as required.
\end{proof}

\begin{defn}[Lifting maps]\label{def:L}
Let $U$ be a set in $\RV$ and $f : U \fun \RV^k$ a function. Let $U_f$ stand for the set $\bigcup \{f(u)^\sharp \times u: u \in U\}$. The \emph{$k$th lifting map}
\[
\mathbb{L}_k: \RV[k] \fun \VF[k]
\]
 is given by $(U,f) \efun U_f$.
The map $\mathbb{L}_{\leq k}: \RV[{\leq} k] \fun \VF[k]$ is given by $\bigoplus_{i} \bm U_i \efun \biguplus_{i} \bb L_i \bm U_i$.
Set $\mathbb{L} = \bigcup_k \mathbb{L}_{\leq k}$.
\end{defn}

\begin{cor}\label{all:subsets:rvproduct}
Every definable set $A \sub \VF^n \times \RV^m$ is a deformed $\RV$-pullback. In particular, if $A \in \VF_*$ then there are a $\bm U \in \RV[{\leq} n]$ and a special bijection from $A$ onto $\mathbb{L}_{{\leq} n}(\bm U)$.
\end{cor}
\begin{proof}
For the first assertion, by compactness, we may assume $A \sub \VF^n$. Then it is a special case of Corollary~\ref{special:bi:term:constant}. The second assertion follows from Lemma~\ref{RV:fiber:dim:same}.
\end{proof}

\begin{defn}[Lifts]\label{def:lift}
Let $F: (U, f) \fun (V, g)$ be an $\RV[k]$-morphism. Then $F$ induces a definable finite-to-finite correspondence $F^\dag \sub f(U) \times g(V)$. Since $F^\dag$ can be decomposed into finitely many definable bijections, for simplicity, we assume that $F^\dag$ is itself a bijection. Let $F^{\sharp} : f(U)^\sharp \fun g(V)^\sharp$ be a definable bijection that $\rv$-contracts to $F^\dag$. Then $F^\sharp$ is called a \emph{lift} of $F$. By Convention~\ref{conv:can}, we shall think of $F^\sharp$ as a definable bijection $\bb L(U, f) \fun \bb L(V, g)$ that $\rv$-contracts to $F^\dag$.
\end{defn}

\begin{lem}\label{simul:special:dim:1}
Let $f : A \fun B$ be a definable bijection between two sets that have exactly one $\VF$-coordinate each. Then there are special bijections $T_A : A \fun A^{\flat}$, $T_B : B \fun B^{\flat}$ such that $A^{\flat}$, $B^{\flat}$ are $\RV$-pullbacks and $f^{\flat}_{\downarrow}$ is bijective in
the commutative diagram
\[
\bfig
  \square(0,0)/->`->`->`->/<600,400>[A`A^{\flat}`B`B^{\flat};
  T_A`f``T_B]
 \square(600,0)/->`->`->`->/<600,400>[A^{\flat}`\rv(A^{\flat})`B^{\flat} `\rv(B^{\flat});  \rv`f^{\flat}`f^{\flat}_{\downarrow}`\rv]
 \efig
\]
Thus, if $A, B \in \VF_*$ then $f^{\flat}$ is a lift of $f^{\flat}_{\downarrow}$, where the latter is regarded as an $\RV[1]$-morphism between $\rv(A^{\flat})_{1}$ and $\rv(B^{\flat})_1$ (recall Notation~\ref{0coor}).
\end{lem}
\begin{proof}
By Corollaries~\ref{special:bi:term:constant}, \ref{all:subsets:rvproduct}, and Lemma~\ref{open:pro}, we may assume that $A$, $B$ are $\RV$-pullbacks, $f$ is $\rv$-contractible and has dtdp, and there is a special bijection $T_B: B \fun B^{\flat}$ such that $(T_B \circ f)^{-1}$ is $\rv$-contractible. Let $T_B = \eta_{n} \circ \ldots \circ \eta_{1}$, where each $\eta_{i}$ is a centripetal transformation (and regularization maps are not displayed). Then it is enough to construct a special bijection $T_A = \zeta_{n} \circ \ldots \circ \zeta_{1}$ on $A$ such that, for each $i$, both $f_i \coloneqq T_{B, i} \circ f \circ T_{A, i}^{-1}$ and $T_{A, i} \circ (T_B \circ f)^{-1}$ are $\rv$-contractible, where $T_{B, i} = \eta_{i} \circ \ldots \circ \eta_{1}$ and $T_{A, i} = \zeta_{i} \circ \ldots \circ \zeta_{1}$.

To that end, suppose that $\zeta_i$ has been constructed for each $i \leq k < n$. Let $A_{k} = T_{A, k}(A)$ and $B_k = T_{B, k}(B)$. Let $D \sub B_k$ be the locus of $\eta_{k+1}$ and $\lambda$ the corresponding focus map. Since $f_k$ is $\rv$-contractible and has dtdp, each $\RV$-polydisc $\gp \sub B_k$ is a union of disjoint sets of the form $f_k(\gq)$, where $\gq \sub A_k$ is an $\RV$-polydisc. For each $t = (t_1, t_{\tilde 1}) \in \dom(\lambda)$, let $O_{t}$ be the set of those $\RV$-polydiscs $\gq \sub A_k$ such that $f_k(\gq) \sub t^\sharp_1 \times t$. Let
\begin{itemize}
  \item $\gq_{t} \in O_{t}$ be the $\RV$-polydisc with $(\lambda(t), t) \in \go_{ t} \coloneqq f_k(\gq_t)$,
  \item $C = \bigcup_{t \in \dom(\lambda)} \gq_{t} \sub A_k$ and $a_{t} = f_k^{-1}(\lambda( t),
t) \in \gq_{t}$,
  \item $\kappa : \pr_{>1} (C) \fun \VF$  the corresponding focus
map given by $\pr_{>1} (\gq_{t}) \efun \pr_1(a_{t})$,
  \item $\zeta_{k+1}$  the centripetal transformation determined by $C$ and $\kappa$.
\end{itemize}
For each $t \in \dom(\lambda)$,  $f_{k+1}$ restricts to a bijection between the
$\RV$-pullbacks $\zeta_{k+1}(\gq_{t})$ and
$\eta_{k+1}(\go_{t})$ that is $\rv$-contractible in
both ways and, for any $\gq \in O_{t}$ with $\gq \neq \gq_{t}$, $f_{k+1}(\gq)$ is an open polydisc contained in an $\RV$-polydisc. So $f_{k+1}$ is $\rv$-contractible.

On the other hand, it is clear that, for any $\RV$-polydisc $\gp \sub B^{\flat}$, $T_{A, k} \circ (T_B \circ f)^{-1}(\gp)$ does not contain any $a_{t}$ and hence, by the construction of $T_{A, k}$, $T_{A, k+1} \circ (T_B \circ f)^{-1}$ is $\rv$-contractible.
\end{proof}

\begin{hyp}\label{hyp:point}
The following lemma is used directly only once in Corollary~\ref{RV:lift}. It should have been presented right after Definition~\ref{defn:corr:cont}. We place it here because this is the first place in this paper, in fact, one of the only two places, the other being Lemma~\ref{blowup:same:RV:coa}, where we need to assume that every definable $\RV$-disc contains a definable point. The easiest way to guarantee this is to assume that $\mdl S$ is $\VF$-generated, which, together with Hypothesis~\ref{hyp:gam}, implies that it is a model of $\TCVF$ and is indeed an elementary substructure (so every definable set contains a definable point). This assumption will be in effect throughout the rest of the paper.
\end{hyp}


\begin{lem}\label{RVlift}
Every definable bijection $f : U \fun V$ between two subsets of $\RV^k$ can be lifted, that is, there is a definable bijection $f^{\sharp} : U^\sharp \fun V^\sharp$ that $\rv$-contracts to $f$.
\end{lem}
\begin{proof}
We do induction on $n = \dim_{\RV}(U) = \dim_{\RV}(V)$. If $n=0$ then $U$ is finite and hence, for every $u \in U$, the $\RV$-polydisc $u^\sharp$ contains a definable point, similarly for $V$, in which case how to construct an $f^{\sharp}$ as desired is obvious.

For the inductive step, by weak \omin-minimality in the $\RV$-sort, there are definable finite partitions $U_i$, $V_i$ of $U$, $V$ and injective coordinate projections
\[
\pi_i : U_i \fun \RV^{k_i}, \quad \pi'_i : V_i \fun \RV^{k_i},
\]
where $\dim_{\RV}(U_i) = \dim_{\RV}(V_i) = k_i$; the obvious bijection $\pi_i(U_i) \fun \pi'_i(V_i)$ induced by $f$ is denoted by $f_i$. Observe that if every $f_i$ can be lifted as desired then, by the construction in the base case above, $F$ can be lifted as desired as well. Therefore, without loss of generality, we may assume $k = n$. For $u \in U$ and $a \in u^\sharp$, the $\RV$-polydisc $f(u)^\sharp$ contains an $a$-definable point and hence, by compactness, there is a definable function $f^{\sharp} : U^\sharp \fun V^\sharp$ that $\rv$-contracts to $f$. By Lemma~\ref{RV:bou:dim}, $\dim_{\RV}(\partial_{\RV}f^{\sharp}(U^\sharp)) < n$ and hence, by the inductive hypothesis, we may assume that $f^{\sharp}$ is surjective. Then there is a definable function $g : V^\sharp \fun U^\sharp$ such that $f^{\sharp}(g(b)) = b$ for all $b \in V^\sharp$. By Lemma~\ref{RV:bou:dim} and the inductive hypothesis again, we may further assume that $g$ is also a surjection, which just means that $f^{\sharp}$ is a bijection as desired.
\end{proof}

The following corollary is an analogue of \cite[Proposition~6.1]{hrushovski:kazhdan:integration:vf}.

\begin{cor}\label{RV:lift}
For every $\RV[k]$-morphism $F : (U, f) \fun (V, g)$ there is a $\VF[k]$-morphism $F^\sharp$ that lifts $F$.
\end{cor}
\begin{proof}
As in Definition~\ref{def:lift}, we may assume that the finite-to-finite correspondence $F^\dag$ is actually a bijection.  Then this is immediate by Lemma~\ref{RVlift}.
\end{proof}

\begin{cor}\label{L:sur:c}
The lifting map $\bb L_{\leq k}$ induces a surjective homomorphism, which is sometimes simply denoted by $\bb L$, between the Grothendieck semigroups
\[
\gsk \RV[{\leq} k] \epi \gsk \VF[k].
\]
\end{cor}
\begin{proof}
By Corollary~\ref{RV:lift}, every $\RV[k]$-isomorphism can be lifted. So $\bb L_{\leq k}$ induces a map on the isomorphism classes, which is easily seen to be a semigroup homomorphism. By Lemma~\ref{altVFdim} and Corollary~\ref{all:subsets:rvproduct}, this  homomorphism is surjective.
\end{proof}

\subsection{$2$-cells}

The remaining object of this section is to identify the kernels of the semigroup homomorphisms $\bb L$ in Corollary~\ref{L:sur:c} and thereby complete the construction of the universal additive invariant. We begin with a discussion of the issue of $2$-cells, as in \cite[\S~4]{Yin:int:acvf}.

The notion of a $2$-cell, which corresponds to that of a bicell in \cite{cluckers:loeser:constructible:motivic:functions}, may look strange and is, perhaps, only of technical
interest. It arises when we try to prove some analogue of Fubini's
theorem, such as Lemma~\ref{contraction:perm:pair:isp} below. The
difficulty is that, although the interaction between $\rv$-contractions and special bijections for definable sets of
$\VF$-dimension $1$ is in a sense ``functorial'' (see Lemma~\ref{simul:special:dim:1}), we are unable to extend the construction to higher $\VF$-dimensions. This is the concern of \cite[Question~7.9]{hrushovski:kazhdan:integration:vf}. It has
also occurred in \cite{cluckers:loeser:constructible:motivic:functions} and actually may be
traced back to the construction of the \omin-minimal Euler characteristic in \cite{dries:1998}; see
\cite[Section~1.7]{cluckers:loeser:constructible:motivic:functions}.

Anyway, in this situation, a natural strategy for $\rv$-contracting the isomorphism class of a
definable set of higher $\VF$-dimension is to apply the result for
$\VF$-dimension $1$ parametrically and proceed with one $\VF$-coordinate at a time. As in the classical theory of integration, this strategy requires some form of Fubini's theorem: for a well-behaved integration (or additive invariant in our case), an integral should yield the same value
when it is evaluated along different orders of the variables. By induction, this problem is immediately reduced to the case of two variables. A $2$-cell is a definable
subset of $\VF^2$ with certain symmetrical (or ``linear'' in the sense described in Remark~\ref{2cell:linear} below) internal structure that satisfies this Fubini-type requirement. Now the idea is that, if we can find a definable partition for every definable set such that each piece is a $2$-cell indexed by some $\RV$-sort parameters, then, by compactness, every definable
set satisfies the Fubini-type requirement. This kind of
partition is achieved in Lemma~\ref{decom:into:2:units}.

\begin{lem}\label{bijection:dim:1:decom:RV}
Let $f : A \fun B$ be a definable bijection between two subsets of $\VF$. Then there is a special bijection $T$ on $A$ such that $A^\flat$ is an $\RV$-pullback and, for each $\RV$-polydisc $\gp \sub A^\flat$, $f \rest T^{-1}(\gp)$ is $\rv$-affine.
\end{lem}
\begin{proof}
By Lemma~\ref{rv:lin} and compactness, for all but finitely many $a \in A$ there is an $a$-definable $\delta_a \in \abs{\Gamma}$ such that $f \rest \go(a, \delta_a)$ is $\rv$-affine. Without loss of generality, we may assume that, for all $a \in A$, $\delta_a$ exists and is the least element that satisfies this condition. Let $g : A \fun |\Gamma|$ be the definable function given by $a \efun \delta_a$. By Corollary~\ref{special:bi:term:constant}, there is a special bijection $T$ on $A$ such that $A^\flat$ is an $\RV$-pullback and, for all $\RV$-polydisc $\gp \sub A^\flat$, $(g \circ T^{-1}) \rest \gp$ is constant. By Lemmas~\ref{one:atomic} and \ref{rv:lin},  we must have $(g \circ T^{-1})(\gp) \leq \rad(\gp)$, for otherwise  the choice of $\delta_a$ is violated for some $a \in T^{-1}(\gp)$. So $T$ is as required.
\end{proof}

\begin{lem}\label{bijection:rv:one:one}
Let $A \sub \VF^2$ be a definable set such that $\ga_1 \coloneqq \pr_1(A)$ and $\ga_2 \coloneqq \pr_2(A)$ are open discs. Suppose that there is a definable bijection $f : \ga_1 \fun \ga_2$ that has dtdp and, for each $a \in \ga_1$, there is a $t_a \in \RVV$ with $A_a = t_a^\sharp + f(a)$. Then there is a special bijection $T$ on $\ga_1$ such that $\ga_1^\flat$ is an $\RV$-pullback and, for each $\RV$-polydisc $\gp \sub \ga_1^\flat$, $\rv$ is constant on the set
\[
\{a - f^{-1}(b) : a \in T^{-1}(\gp) \text{ and } b \in A_a \}.
\]
\end{lem}
\begin{proof}
For each $a \in \ga_1$, let $\gb_a$ be the smallest closed disc that contains $A_a$. Since $A_a - f(a) = t_a^\sharp$, we have $f(a) \in \gb_a$ but $f(a) \notin A_a$ if $t_a \neq 0$. Hence $a \notin f^{-1}(A_a)$ if $t_a \neq 0$ and $\{a\} = f^{-1}(A_a)$ if $t_a = 0$. Since $f^{-1}(A_a)$ is a disc or a point, in either case, the function on $f^{-1}(A_a)$ given by $b \efun \rv(a - b)$ is constant. The function $h : \ga_1 \fun \RVV$ given by $a \efun \rv(a - f^{-1}(A_a))$ is definable. Now we apply Corollary~\ref{special:bi:term:constant} as in the proof of Lemma~\ref{bijection:dim:1:decom:RV}. The lemma follows.
\end{proof}

\begin{defn}\label{defn:balance}
Let $A$, $\ga_1$, $\ga_2$, and $f$ be as in Lemma~\ref{bijection:rv:one:one}. We say that $f$ is \emph{balanced in $A$} if $f$ is actually $\rv$-affine and there are $t_1, t_2 \in \RVV$, called the \emph{paradigms} of $f$, such that, for every $a \in \ga_1$,
\[
A_a = t_2^\sharp + f(a) \dand f^{-1}(A_a) = a - t_1^\sharp.
\]
\end{defn}

\begin{rem}\label{2cell:linear}
Suppose that $f$ is balanced in $A$ with paradigms $t_1$, $t_2$.
If one of the paradigms is $0$ then the other one must be $0$. In this case $A$ is just the (graph of the) bijection $f$ itself.

Assume that $t_1$, $t_2$ are nonzero. Let $\gB_1$, $\gB_2$ be, respectively, the sets of closed subdiscs of $\ga_1$, $\ga_2$ of radii $\abs{\vrv(t_1)}$, $\abs{\vrv(t_2)}$. Let $a_1 \in \gb_1 \in \gB_1$ and $\go_1$ be the maximal open subdisc of $\gb_1$ containing $a_1$. Let $\gb_2 \in \gB_2$ be the smallest closed disc containing the open disc $\go_2 \coloneqq A_{a_1}$. Then, for all $a_2 \in \go_2$, we have
\[
\go_2 = t_2^\sharp + f(\go_1) = A_{a_1} \dand A_{a_2} = f^{-1}(\go_2) + t_1^\sharp = \go_1.
\]
This internal symmetry of $A$ is illustrated by the following diagram:
\[
\bfig
  \dtriangle(0,0)|amb|/.``<-/<600,250>[\go_1`f^{-1}(\go_2)`\go_2; \pm t_1^\sharp`\times`f^{-1}]
  \ptriangle(600,0)|amb|/->``./<600,250>[\go_1`f(\go_1)`\go_2; f`` \pm t_2^\sharp]
 \efig
\]
Since $f$ is $\rv$-affine, we see that its slope must be $-t_2/t_1$ (recall Definition~\ref{rvaffine}).

If we think of $\gb_1$, $\gb_2$ as $\tor(\code {\go_1})$, $\tor(\code {\go_2})$ then the set $A \cap (\gb_1 \times \gb_2)$ may be thought of as the ``line'' in $\tor(\code {\go_1}) \times \tor(\code {\go_2})$ given by the equation
\[
x_2 = - \tfrac{t_2}{t_1}(x_1 - \code{\go_1}) + (\code{\go_2} - t_2).
\]
Thus, by Lemma~\ref{simul:special:dim:1}, the obvious bijection between $\pr_1(A) \times t_2^\sharp$ and $t_1^\sharp \times \pr_2(A)$ is the lift of an $\RV[{\leq}2]$-morphism modulo special bijections; see Lemma~\ref{2:unit:contracted} below for details. The slope of $f$ will play a more important role when  volume forms are introduced into the categories (in a sequel).
\end{rem}

\begin{defn}[$2$-cell]\label{def:units}
We say that a set $A$ is a \emph{$1$-cell} if it is either an open disc contained in a single $\RV$-disc or a point in $\VF$. We say
that $A$ is a \emph{$2$-cell} if
\begin{enumerate}
 \item $A$ is a subset of $\VF^2$ contained in a single $\RV$-polydisc and $\pr_1(A)$ is a $1$-cell,
 \item there is a function $\epsilon : \pr_1 (A) \fun \VF$ and a $t \in \RV$ such that, for every $a \in \pr_1(A)$, $A_a = t^\sharp + \epsilon(a)$,
 \item one of the following three possibilities occurs:
  \begin{enumerate}
   \item $\epsilon$ is constant,
   \item $\epsilon$ is injective, has dtdp, and $\rad(\epsilon(\pr_1(A))) \geq \abs{\vrv(t)}$,\label{2cell:3b}
   \item $\epsilon$ is balanced in $A$.
  \end{enumerate}
\end{enumerate}
The function $\epsilon$ is called the \emph{positioning function} of $A$ and the element $t$ the \emph{paradigm} of $A$.

More generally, a set $A$ with exactly one $\VF$-coordinate is a \emph{$1$-cell} if, for each $t \in \pr_{>1}(A)$, $A_t$ is a $1$-cell in the above sense; the parameterized version of the notion of a $2$-cell is formulated in the same way.
\end{defn}

A $2$-cell is definable if all the relevant ingredients are definable. Naturally we will only be concerned with definable $2$-cells. Notice that Corollary~\ref{all:subsets:rvproduct} implies that for every definable set $A$ with exactly one $\VF$-coordinate there is a
definable function $\pi: A \fun \RV^l$ such that every fiber $A_s$ is a $1$-cell. This should be understood as $1$-cell decomposition and the next lemma as $2$-cell decomposition.

\begin{lem}[$2$-cell decomposition]\label{decom:into:2:units}
For every definable set $A \sub \VF^2$ there is a definable function $\pi: A \fun \RV^m$ such that every fiber $A_s$ is an $s$-definable $2$-cell.
\end{lem}
\begin{proof}
By compactness, we may assume that $A$ is contained in a single
$\RV$-polydisc. For each $a \in \pr_1 (A)$, by Corollary~\ref{all:subsets:rvproduct}, there is an $a$-definable special bijection $T_a$ on $A_a$ such that $A_a^\flat$ is an $\RV$-pullback. By Lemma~\ref{inverse:special:dim:1}, there is an $a$-definable function $\epsilon_a : (A_a^\flat)_{\RV} \fun \VF$ such that, for every $(t, s) \in (A_a^\flat)_{\RV}$, we have
\[
T_a^{-1}(t^\sharp \times (t, s)) =
t^\sharp + \epsilon_a(t,  s).
\]
By compactness, we may glue these functions together, that is, there is a definable set $C \sub \pr_1(A) \times \RV^l$ and a definable function $\epsilon : C \fun \VF$ such that, for every $a \in \pr_1(A)$, $C_a = (A_a^\flat)_{\RV}$ and $\epsilon \rest C_a = \epsilon_a$.
For $(t,  s) \in C_{\RV}$, write $\epsilon_{(t, s)} = \epsilon \rest C_{(t, s)}$. By Corollary~\ref{uni:fun:decom} and compactness, we are reduced to the case
that each $\epsilon_{(t, s)}$ is either
constant or injective. If no $\epsilon_{(t, s)}$ is injective then we can finish by applying
Corollary~\ref{all:subsets:rvproduct} to each $C_{(t, s)}$ and then compactness.

Suppose that some $\epsilon_{(t, s)}$ is injective. Then, by Lemmas~\ref{open:pro} and \ref{bijection:dim:1:decom:RV}, we are reduced to the case that $C_{(t, s)}$ is an open disc and
$\epsilon_{(t, s)}$ is $\rv$-affine and has dtdp. Write $\gb_{(t, s)} = \ran (\epsilon_{(t, s)})$. If $\rad(\gb_{(t, s)}) \geq \abvrv(t)$ then $\epsilon_{(t, s)}$ satisfies the condition (\ref{2cell:3b}) in Definition~\ref{def:units}. So let us suppose $\rad(\gb_{(t, s)}) < \abvrv(t)$. Then
\[
\textstyle \gb_{(t, s)} = \bigcup_{a \in C_{(t, s)}} (t^\sharp + \epsilon_{(t, s)}(a)).
\]
By Lemma~\ref{bijection:rv:one:one}, we are further reduced to the case that there is an $r \in \RV$ such that, for every $a \in C_{(t, s)}$,
\[
\rv(a - \epsilon_{(t, s)}^{-1}(t^\sharp + \epsilon_{(t, s)}(a))) = r \quad \text{and hence} \quad \epsilon_{(t, s)}^{-1}(t^\sharp + \epsilon_{(t, s)}(a)) = a - r^\sharp.
\]
So, in this case, $\epsilon_{(t, s)}$ is balanced. Now we are
done by compactness.
\end{proof}

To extend Lemma~\ref{simul:special:dim:1} to all definable bijections, we need not only $2$-cell decomposition but also the following notions.

Let $A \sub \VF^{n} \times \RV^{m}$, $B \sub \VF^{n} \times \RV^{m'}$, and $f : A \fun B$ be a definable bijection.

\begin{defn}\label{rela:unary}
We say that $f$ is \emph{relatively unary} or, more precisely, \emph{relatively unary in the $i$th $\VF$-coordinate}, if $(\pr_{\tilde{i}} \circ f)(x) = \pr_{\tilde{i}}(x)$ for all $x \in A$,  where $i \in [n]$.  If $f \rest A_y$ is also a special bijection for every $y \in \pr_{\tilde{i}} (A)$ then we say that $f$ is \emph{relatively special in the $i$th $\VF$-coordinate}.
\end{defn}

Obviously the inverse of a relatively unary bijection is a relatively unary bijection. Also note that every special bijection on $A$ is a composition of relatively special bijections.

Choose an $i \in [n]$. By Corollary~\ref{all:subsets:rvproduct} and compactness, there is a bijection $T_i$ on $A$, relatively special in the $i$th $\VF$-coordinate, such that $T_i(A_a)$ is an $\RV$-pullback for every $a \in \pr_{\tilde i}(A)$. Note that $T_i$ is not necessarily a special bijection on $A$, since the special bijections in the $i$th $\VF$-coordinate for distinct $a, a' \in \pr_{\tilde i}(A)$ with $\rv(a) = \rv(a')$ may not even be of the same length. Let
\[
\textstyle A_i = \bigcup_{a \in \pr_{\tilde i}(A)} a \times  (T_i(A_a))_{\RV} \sub \VF^{n-1} \times \RV^{m_i}.
\]
Write $\hat T_i : A \fun A_i$ for the function naturally induced by $T_i$. For any $j \in [n{-}1]$, we repeat the above procedure on $A_i$ with respect to the $j$th $\VF$-coordinate and thereby obtain a set $A_{j} \sub \VF^{n-2} \times \RV^{m_j}$ and a function $\hat T_{j} : A_i \fun A_{j}$. The relatively special bijection on $T_i(A)$ induced by $\hat T_{j}$ is denoted by $T_j$. Continuing thus, we obtain a sequence of  bijections $T_{\sigma(1)}, \ldots, T_{\sigma(n)}$ and a corresponding function $\hat T_{\sigma} : A \fun \RV^{l}$, where $\sigma$ is the permutation of $[n]$ in question. The composition $T_{\sigma(n)} \circ \ldots \circ T_{\sigma(1)}$, which is referred to as the \emph{lift} of $\hat T_{\sigma}$, is denoted by $T_{\sigma}$.


\begin{defn}\label{defn:standard:contraction}
Suppose that there is a $k \in 0 \cup [m]$ such that  $(A_a)_{\leq k} \in \RV[k]$  for every $a \in A_{\VF}$. In particular, if $k=0$ then $A \in \VF_*$. By Lemma~\ref{RV:fiber:dim:same}, $\hat T_{\sigma}(A)_{\leq n+k}$ is an object of $\RV[{\leq} l{+}k]$, where $\dim_{\VF}(A) = l$. The function $\hat T_{\sigma}$ --- or the object $\hat T_{\sigma}(A)_{\leq n+k}$ --- is referred to as a \emph{standard contraction} of the set $A$ with the \emph{head start} $k$.
\end{defn}

The head start of a standard contraction is usually implicit. In fact, it is always $0$ except in Lemma~\ref{isp:VF:fiberwise:contract}, and can be circumvented even there. This seemingly needless gadget only serves to make the above definition more streamlined: If $A \in \VF_*$ then the intermediate steps of a standard contraction of $A$ may or may not result in objects of $\VF_*$ and hence the definition cannot be formulated entirely within $\VF_*$.

\begin{rem}\label{special:dim:1:RV:iso}
In Lemma~\ref{simul:special:dim:1}, clearly $\rv(A^{\flat})$, $\rv(B^{\flat})$ are standard contractions of $A$, $B$. Indeed, if $A, B \in \VF_*$ then $[\rv(A^{\flat})]_{\leq 1} = [\rv(B^{\flat})]_{\leq 1}$.
\end{rem}

\begin{lem}\label{bijection:partitioned:unary}
There is a definable finite partition $A_i$ of $A$  such that each $f \rest A_i$ is a composition of relatively unary bijections.
\end{lem}
\begin{proof}
This is an easy consequence of weak \omin-minimality. In more detail, for each $a \in \pr_{< n}(A)$ there are an $a$-definable finite partition $A_{ai}$ of $A_a$ and  injective coordinate projections $\pi_i : f(A_{ai}) \fun \VF \times \RV^{m'}$. Therefore, by compactness, there are a definable finite partition $A_{i}$ of $A$, definable injections $f_i : A_i \fun \VF^{n} \times \RV^{m'}$, and $j_i \in [n]$ such that, for all $x \in A_i$,
\[
\pr_{< n}(x) = \pr_{< n}(f_i(x)) \dand \pr_{n \cup [m']}(f_i(x)) = \pr_{j_i \cup [m']}(f(x)).
\]
The claim now follows from compactness and an obvious induction on $n$.
\end{proof}

For the next two lemmas, let $12$ and $21$ denote the permutations of $[2]$.

\begin{lem}\label{2:unit:contracted}
Let $A \sub \VF^2$ be a definable $2$-cell. Then there are standard contractions $\hat T_{12}$, $\hat R_{21}$ of $A$ such that $[\hat T_{12}(A)]_{\leq 2} = [\hat R_{21}(A)]_{\leq 2}$.
\end{lem}
\begin{proof}
Let $\epsilon$ be the positioning function of $A$ and $t \in \RV_0$ the paradigm of $A$. If $t = 0$ then $A$ is (the graph of) the function $\epsilon : \pr_1(A) \fun \pr_2(A)$, which is either a constant function or a bijection. In the former case, since $A$ is essentially just an open ball, the lemma simply follows from
Corollary~\ref{all:subsets:rvproduct}. In the latter case, there
are relatively special bijections $T_2$, $R_1$ on $A$ in the coordinates $2$, $1$ such that
\[
T_2(A) = \pr_1(A) \times 0 \times 0 \dand
R_1(A) = 0 \times \pr_2(A) \times 0.
\]
So the lemma follows from Remark~\ref{special:dim:1:RV:iso}. For the rest of the proof we assume $t \neq 0$.

If $\epsilon$ is not balanced in $A$ then $A = \pr_1(A) \times \pr_2(A)$ is an open polydisc. By Corollary~\ref{all:subsets:rvproduct}, there are special bijections $T_1$, $T_2$ on $\pr_1(A)$, $\pr_2(A)$ such that $\pr_1(A)^\flat$, $\pr_2(A)^\flat$ are $\RV$-pullbacks. In this case the standard contractions determined by $(T_1, T_2)$ and $(T_2, T_1)$ are essentially the same.

Suppose that $\epsilon$ is balanced in $A$. Let $r$ be the other paradigm of $\epsilon$. Recall that $\epsilon : \pr_1 (A) \fun \pr_2(A)$ is again a bijection. Let $T_2$ be the relatively special bijection on $A$ in the coordinate $2$ given by $(a, b) \efun (a, b - \epsilon(a))$ and $R_1$ the relatively special bijection on $A$ in the coordinate $1$ given by $(a, b) \efun (a - \epsilon^{-1}(b), b)$, where $(a, b) \in A$. Clearly
\[
T_2(A) = \pr_1(A) \times t^\sharp \times t \dand
R_1(A) = r^\sharp \times \pr_2(A) \times r.
\]
So, again, the lemma follows from Remark~\ref{special:dim:1:RV:iso}.
\end{proof}

\begin{lem}\label{subset:partitioned:2:unit:contracted}
Let $A \sub \VF^2 \times \RV^m$ be an object in $\VF_*$. Then there are a definable injection $f : A \fun \VF^2 \times \RV^l$, relatively unary in both coordinates, and standard contractions $\hat T_{12}$, $\hat R_{21}$ of $f(A)$ such that $[\hat T_{12}(f(A))]_{\leq 2} = [\hat R_{21}(f(A))]_{\leq 2}$.
\end{lem}
\begin{proof}
By Lemma~\ref{decom:into:2:units} and compactness, there is a definable function $f: A \fun \VF^2 \times \RV^l$ such that $f(A)$ is a $2$-cell and, for each $(a, t) \in A$, $f(a, t) = (a, t, s)$ for some $s \in \RV^{l-m}$. By Lemma~\ref{2:unit:contracted} and compactness, there are standard contractions $\hat T_{12}$, $\hat R_{21}$ of $f(A)$ into $\RV^{k+l}$ such that the following diagram commutates
\[
\bfig
  \Vtriangle(0,0)/->`->`->/<400,400>[\hat T_{12}(f(A))`\hat R_{21}(f(A))`\RV^l; F`\pr_{> k}`\pr_{> k}]
\efig
\]
and $F$ is an $\RV[{\leq} 2]$-morphism  $\hat T_{12}(f(A))_{\leq 2} \fun \hat R_{21}(f(A))_{\leq 2}$.
\end{proof}

\subsection{Blowups and the main theorems}

The central notion for understanding the kernels of the semigroup homomorphisms $\bb L$ is that of a blowup:

\begin{defn}[Blowups]\label{defn:blowup:coa}
Let $\bm U = (U, f) \in \RV[k]$, where $k > 0$, such that, for some $j \leq k$, the restriction $\pr_{\tilde j} \rest f(U)$ is finite-to-one. Write $f = (f_1, \ldots, f_k)$. The \emph{elementary blowup} of $\bm U$ in the $j$th coordinate is the pair $\bm U^{\flat} = (U^{\flat}, f^{\flat})$, where $U^{\flat} = U \times \RV^{\circ \circ}_0$ and, for every $(t, s) \in U^{\flat}$,
\[
f^{\flat}_{i}(t, s) = f_{i}(t) \text{ for } i \neq j \dand f^{\flat}_{j}(t, s) = s f_{j}(t).
\]
Note that $\bm U^{\flat}$ is an object in $\RV[{\leq} k]$ (actually in $\RV[k{-}1] \oplus \RV[k]$) because $f^{\flat}_{j}(t, 0) = 0$.

Let $\bm V = (V, g) \in \RV[k]$ and $C \sub V$ be a definable set. Suppose that $F : \bm U \fun \bm C$ is an $\RV[k]$-morphism, where $\bm C = (C, g \rest C) \in \RV[k]$. Then
\[
\bm U^{\flat} \uplus (V \mi C, g \rest (V \mi C))
\]
is  a \emph{blowup of $\bm V$ via $F$}, denoted by $\bm V^{\flat}_F$. The subscript $F$ is usually dropped in context if there is no danger of confusion. The object $\bm C$ (or the set $C$) is referred to as the \emph{locus} of $\bm V^{\flat}_F$.

A \emph{blowup of length $n$} is a composition of $n$ blowups.
\end{defn}

\begin{rem}
In an elementary blowup, the condition that the coordinate of interest is definably dependent (the coordinate projection is finite-to-one) on the other ones is needed so that the resulting objects stay in $\RV[{\leq} k]$. In the setting of \cite{hrushovski:kazhdan:integration:vf}, this condition is also needed for matching blowups with special bijections, since, otherwise, we would not be able to use (a generalization of) Hensel's lemma to find enough centers of $\RV$-discs to construct focus maps. In our setting, Lemma~\ref{RVlift} plays the role of  Hensel's lemma, which is more powerful,  and hence ``algebraicity'' is no longer needed for this purpose (see Lemma~\ref{blowup:same:RV:coa}).
\end{rem}

If there is an elementary blowup of $(U, f) \in \RV[k]$ then, \textit{a posteriori}, $\dim_{\RV}(f(U)) < k$. Also, there is at most one elementary blowup of $(U, f)$ with respect to any coordinate of $f(U)$. We should have included the coordinate that is blown up as a part of the data. However, in context, either this is
clear or it does not need to be spelled out, and we shall suppress mentioning it below for notational ease.

\begin{lem}\label{blowup:equi:class:coa}
Let $\bm U, \bm V \in \RV[{\leq} k]$ such that $[\bm U] = [\bm V]$ in $\gsk \RV[{\leq} k]$. Let $\bm U_1$, $\bm V_1$ be blowups of $\bm U$, $\bm V$ of lengths $m$, $n$, respectively. Then there are blowups $\bm U_2$, $\bm V_2$ of $\bm U_1$, $\bm V_1$ of lengths $n$, $m$, respectively, such that $[\bm U_2] = [\bm V_2]$.
\end{lem}
\begin{proof}
Fix an isomorphism $I: \bm U \fun \bm V$. We do induction on the sum $l = m + n$. For the base case $l = 1$, without loss of generality, we may assume $n = 0$. Let $C$ be the blowup locus of $\bm U_1$. Clearly $\bm V$ may be blown up by using the same elementary blowup as $\bm U_1$, where the blowup locus is changed to $I(C)$, and the resulting blowup is as required.

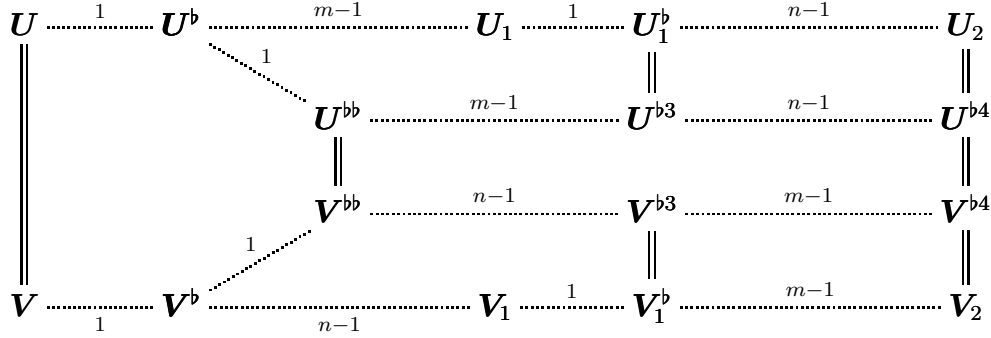
\[
\bfig
  \square(0,0)/.`=``./<500,900>[\bm U`\bm U^{\flat}`
  \bm V`\bm V^{\flat};1```1]
  \square(500,0)/.```./<1000,900>[\bm U^{\flat}`\bm U_1`
  \bm V^{\flat}`\bm V_1; m - 1```n - 1]
  \morphism(500,900)/./<500,-300>[\bm U^{\flat}`\bm U^{\flat\flat};1]
  \morphism(500,0)/./<500,300>[\bm V^{\flat}`
   \bm V^{\flat\flat};1]
  \morphism(1000,300)/=/<0,300>[\bm V^{\flat\flat}`
   \bm U^{\flat\flat};]
  \morphism(1500,900)/./<500,0>[\bm U_1`\bm U_1^{\flat};1]
  \morphism(1500,0)/./<500,0>[\bm V_1`\bm V_1^{\flat};1]
  \morphism(1000,600)/./<1000,0>[\bm U^{\flat\flat}`\bm U^{\flat3};
   m - 1]
  \morphism(1000,300)/./<1000,0>[\bm V^{\flat\flat}`\bm V^{\flat3};
   n - 1]
  \morphism(2000,900)/=/<0,-300>[\bm U_1^{\flat}`\bm U^{\flat3};]
  \morphism(2000,0)/=/<0,300>[\bm V_1^{\flat}`\bm V^{\flat3};]
  \morphism(2000,600)/./<1000,0>[\bm U^{\flat3}`\bm U^{\flat4};n - 1]
  \morphism(2000,300)/./<1000,0>[\bm V^{\flat3}`\bm V^{\flat4};m - 1]
  \morphism(3000,300)/=/<0,300>[\bm V^{\flat4}`\bm U^{\flat4};]
  \morphism(2000,900)/./<1000,0>[\bm U_1^{\flat}`\bm U_2;n - 1]
  \morphism(2000,0)/./<1000,0>[\bm V_1^{\flat}`\bm V_2;m - 1]
  \morphism(3000,0)/=/<0,300>[\bm V_2`\bm V^{\flat4};]
  \morphism(3000,900)/=/<0,-300>[\bm U_2`\bm U^{\flat4};]
\efig
\]

We proceed to the inductive step. How the isomorphic blowups are constructed is illustrated above. Write $\bm U = (U, f)$ and $\bm V = (V, g)$. Let $\bm U^{\flat}$, $\bm V^{\flat}$ be the first blowups in $\bm U_1$, $\bm V_1$ and $C$, $D$ their blowup loci, respectively. Let $\bm U'^{\flat}$, $\bm V'^{\flat}$ be the corresponding elementary blowups contained in $\bm U^{\flat}$, $\bm V^{\flat}$. If, say, $n = 0$, then by the argument in the base case $\bm V$ may be blown up to an object that is isomorphic to $\bm U^{\flat}$ and hence the inductive
hypothesis may be applied. So assume $m,n > 0$. Let $A = C \cap I^{-1}(D)$ and $B = I(C) \cap D$. Since $(A, f \rest A)$ and $(B, g \rest B)$ are isomorphic, the
blowups of $\bm U'$, $\bm V'$ with the loci $(A, f \rest A)$ and $(B, g \rest B)$ are isomorphic. Then, it is not hard to see that the blowup $\bm U^{\flat\flat}$ of $\bm U^{\flat}$ using the locus $I^{-1}(D) \mi C$ and its corresponding blowup of $\bm V'$ and the blowup $\bm V^{\flat\flat}$ of $\bm V^{\flat}$ using the locus $I(C) \mi D$ and its corresponding blowup of $\bm U'$ are isomorphic.

Applying the inductive hypothesis to the blowups $\bm U^{\flat\flat}$, $\bm U_1$ of $\bm U^{\flat}$, we obtain a blowup $\bm U^{\flat3}$ of $\bm U^{\flat\flat}$ of length $m - 1$ and a blowup $\bm U_1^{\flat}$ of $\bm U_1$ of length $1$ such that they are isomorphic. Similarly, we obtain a blowup $\bm V^{\flat3}$ of $\bm V^{\flat\flat}$ of length $n - 1$ and a blowup $\bm V_1^{\flat}$ of $\bm V_1$ of length $1$ such that they are isomorphic. Applying the inductive hypothesis again to the blowups $\bm U^{\flat3}$, $\bm V^{\flat3}$ of $\bm U^{\flat\flat}$, $\bm V^{\flat\flat}$, we obtain a blowup $\bm U^{\flat4}$ of $\bm U^{\flat3}$ of length $n - 1$ and a blowup $\bm V^{\flat4}$ of $\bm V^{\flat3}$ of length $m - 1$ such that they are isomorphic. Finally, applying the inductive hypothesis to the blowups
$\bm U^{\flat4}$, $\bm U_1^{\flat}$ of $\bm U^{\flat3}$, $\bm U_1^{\flat}$ and the blowups $\bm V^{\flat4}$, $\bm V_1^{\flat}$ of $\bm V^{\flat3}$, $\bm V_1^{\flat}$, we obtain a blowup $\bm U_2$ of $\bm U_1^{\flat}$ of length $n - 1$ and a blowup $\bm V_2$ of $\bm V_1^{\flat}$ of length $m - 1$ such that $\bm U^{\flat4}$, $\bm U_2$, $\bm V^{\flat4}$, and $\bm V_2$ are all isomorphic.  So $\bm U_2$, $\bm V_2$ are as desired.
\end{proof}

\begin{cor}\label{blowup:equi:class}
Let $[\bm U] = [\bm U']$ and $[\bm V] = [\bm V']$ in $\gsk \RV[{\leq} k]$. If there are isomorphic blowups of $\bm U$, $\bm V$ then there are isomorphic blowups of $\bm U'$, $\bm V'$.
\end{cor}

\begin{defn}\label{defn:isp}
Let $\isp[k]$ be the set of pairs $(\bm U, \bm V)$ of objects of $\RV[{\leq} k]$ such that there exist isomorphic blowups $\bm U^{\flat}$, $\bm V^{\flat}$. Set $\isp[*] = \bigcup_{k} \isp[k]$.
\end{defn}

We will just write $\isp$ for all these sets when there is no danger of confusion. By Corollary~\ref{blowup:equi:class}, $\isp$ may be regarded as a binary relation on isomorphism classes.

\begin{lem}\label{isp:congruence:vol}
$\isp[k]$ is a semigroup congruence relation and $\isp[*]$ is a semiring congruence relation.
\end{lem}
\begin{proof}
Clearly $\isp[k]$ is reflexive and symmetric. If $([\bm U_1], [\bm U_2])$, $([\bm U_2], [\bm U_3])$ are in
$\isp[k]$ then, by Lemma~\ref{blowup:equi:class:coa}, there are blowups $\bm U_1^{\flat}$ of $\bm U_1$, $\bm U_{2}^{\flat 1}$ and $\bm U_{2}^{\flat 2}$ of $\bm U_2$, and
$\bm U_3^{\flat}$ of $\bm U_3$ such that they are all isomorphic. So $\isp[k]$ is transitive and hence is an
equivalence relation. For any $[\bm W] \in \gsk \RV[l]$, the following are
easily checked:
\[
([\bm U_1 \uplus \bm W], [\bm U_2 \uplus \bm W])\in \isp,\quad
([\bm U_1 \times \bm W], [\bm U_2 \times \bm W])\in \isp.
\]
These yield the desired congruence relations.
\end{proof}

Let $\bm U = (U, f)$ be an object of $\RV[k]$ and $T$ a special bijection on $\bb L \bm U$. The set $(T(\mathbb{L} \bm U))_{\RV}$ is simply denoted by $U_{T}$ and the object $(U_{T})_{\leq k} \in \RV[{\leq} k]$ by $\bm U_{T}$.

\begin{lem}\label{special:to:blowup:coa}
The object $\bm U_T$ is isomorphic to a blowup of $\bm U$ of the same length as $T$.
\end{lem}
\begin{proof}
By induction on the length $\lh (T)$ of $T$ and Lemma~\ref{blowup:equi:class:coa}, this is immediately reduced to the case $\lh (T) = 1$. For that case, let $\lambda$ be the focus map of $T$. Without loss of generality, we may assume that the locus of $\lambda$ is $\mathbb{L} \bm U$. Then it is clear how to construct an (elementary) blowup of $\bm U$ as desired.
\end{proof}

\begin{lem}\label{kernel:dim:1:coa}
Suppose that $[A] = [B]$ in $\gsk \VF[1]$ and $\bm U, \bm V \in \RV[{\leq} 1]$ are two standard contractions of $A$, $B$, respectively. Then $([\bm U], [\bm V]) \in \isp$.
\end{lem}
\begin{proof}
By Lemma~\ref{simul:special:dim:1}, there are special bijections $T$, $R$ on $\bb L \bm U$, $\bb L \bm V$ such that $\bm U_{T}$, $\bm V_{R}$ are isomorphic. So the assertion follows from Lemma~\ref{special:to:blowup:coa}.
\end{proof}

\begin{lem}\label{blowup:same:RV:coa}
Let $\bm U^{\flat}$ be a blowup of $\bm U = (U, f) \in \RV[{\leq} k]$ of length $l$. Then $\bb L \bm U^{\flat}$ is isomorphic to $\bb L \bm U$.
\end{lem}
\begin{proof}
By induction on $l$ this is immediately reduced to the case $l=1$. For that case, without loss of generality, we may assume that $\pr_{\tilde 1} \rest f(U)$ is injective and $\bm U^{\flat}$ is an elementary blowup in the first coordinate. So it is enough to show that there is a focus map into the first coordinate with locus $f(U)^\sharp$.  This is guaranteed by  Hypothesis~\ref{hyp:point}.
\end{proof}

\begin{lem}\label{isp:VF:fiberwise:contract}
Let $A'$, $A''$ be definable sets with $A'_{\VF} = A''_{\VF} \eqqcolon A \sub \VF^n$. Suppose that there is a $k \in \N$ such that, for  every $a \in A$, $([A'_a]_{\leq k}, [A''_a]_{\leq k}) \in \isp$. Let $\hat T_{\sigma}$, $\hat R_{\sigma}$ be respectively standard contractions of $A'$, $A''$. Then
\[
([\hat T_{\sigma}(A')]_{\leq n+k}, [\hat R_{\sigma}(A'')]_{\leq n+k}) \in \isp.
\]
\end{lem}

Note that the condition $([A'_a]_{\leq k}, [A''_a]_{\leq k}) \in \isp$ makes sense only  over the substructure $\mdl S \la a \ra$.

\begin{proof}
By induction on $n$ this is immediately reduced to the case $n=1$. So assume $A \sub \VF$. Let $\phi'$, $\phi''$ be quantifier-free formulas that define $A'$, $A''$, respectively. Let $\theta$ be a quantifier-free formula such that, for  every $a \in A$, $\theta(a)$ defines the necessary data (two blowups and an $\RV[*]$-morphism) that witness the condition $([A'_a]_{\leq k}, [A''_a]_{\leq k}) \in \isp$. Applying Corollary~\ref{special:bi:term:constant} to the top \LT-terms of $\phi'$, $\phi''$, and $\theta$, we obtain a special bijection $F: A \fun A^{\flat}$ such that $A^{\flat}$ is an $\RV$-pullback and, for all $\RV$-polydiscs $\gp \sub A^{\flat}$ and all $a_1, a_2 \in F^{-1}(\gp)$,
\begin{itemize}
 \item $A'_{a_1} = A'_{a_2}$ and $A''_{a_1} = A''_{a_2}$,
 \item $\theta(a_1)$ and $\theta(a_2)$ define the same data.
\end{itemize}
The second item implies that the data defined by $\theta$ over $F^{-1}(\gp)$ is actually $\rv(\gp)$-definable.

Let $B' = \bigcup_{a \in A} F(a) \times A'_a$, similarly for $B''$. Note that $B'$, $B''$ are obtained through special bijections on $A'$, $A''$. For all $t \in A'_{\RV}$, $B'_t$ is an $\RV$-pullback that is $t$-definably bijective to the $\RV$-pullback $T_{\sigma}(A')_t$. By Lemma~\ref{kernel:dim:1:coa}, we have, for all $t \in A'_{\RV}$
\[
([(B'_{\RV})_t]_1, [\hat T_{\sigma}(A')_t]_1) \in \isp
\]
and hence, by compactness,
\[
([B'_{\RV}]_{\leq k+1}, [\hat T_{\sigma}(A')]_{\leq k+1}) \in \isp.
\]
The same holds for $B''$ and $\hat R_{\sigma}(A'')$. On the other hand, by the second item above, for every $\RV$-polydisc $\gp \sub A^{\flat}$, we have $((B'_{\RV})_{\rv(\gp)}, (B''_{\RV})_{\rv(\gp)}) \in \isp$
and hence,  by compactness,
\[
([B'_{\RV}]_{\leq k+1}, [B''_{\RV}]_{\leq k+1}) \in \isp.
\]
Since $\isp$ is a congruence relation, the lemma follows.
\end{proof}

\begin{cor}\label{contraction:same:perm:isp}
Let $A', A'' \in \VF_*$ with exactly $n$ $\VF$-coordinates each and $f : A' \fun A''$ be a relatively unary bijection in the $i$th coordinate. Then for any permutation $\sigma$ of $[n]$ with $\sigma(1) = i$ and any standard contractions $\hat T_{\sigma}$, $\hat R_{\sigma}$ of $A'$, $A''$,
\[
([\hat T_{\sigma}(A')]_{\leq n}, [\hat R_{\sigma}(A'')]_{\leq n}) \in \isp.
\]
\end{cor}
\begin{proof}
This is immediate by Lemmas~\ref{kernel:dim:1:coa} and \ref{isp:VF:fiberwise:contract}.
\end{proof}

The following lemma is essentially a version of Fubini's theorem (also see Theorem~\ref{semi:fubini} below).

\begin{lem}\label{contraction:perm:pair:isp}
Let $A \in \VF_*$ with exactly $n$ $\VF$-coordinates. Suppose that $i, j \in [n]$ are distinct and $\sigma_1$, $\sigma_2$ are permutations of $[n]$ such that
\[
\sigma_1(1) = \sigma_2(2) = i, \quad \sigma_1(2) = \sigma_2(1) = j, \quad \sigma_1
\rest \set{3, \ldots, n} = \sigma_2 \rest \set{3, \ldots, n}.
\]
Then, for any standard contractions $\hat T_{\sigma_1}$, $\hat T_{\sigma_2}$ of $A$,
\[
([\hat T_{\sigma_1}(A)]_{\leq n}, [\hat T_{\sigma_2}(A)]_{\leq n}) \in \isp.
\]
\end{lem}
\begin{proof}
Let $ij$, $ji$ denote the permutations of $E \coloneqq \{i, j\}$. By compactness and
Lemma~\ref{isp:VF:fiberwise:contract}, it is enough to show
that, for any $a \in \pr_{\tilde E}(A)$ and any standard
contractions $\hat T_{ij}$, $\hat T_{ji}$ of $A_a$,
\[
([\hat T_{ij}(A_a)]_{\leq 2}, [\hat T_{ji}(A_a)]_{\leq 2})
\in \isp.
\]
To that end, fix an $a \in \pr_{\tilde E}(A)$. By Lemma~\ref{subset:partitioned:2:unit:contracted}, there are a definable bijection $f$ on $A_a$ that is relatively unary in both $\VF$-coordinates and standard contractions $\hat R_{ij}$, $\hat R_{ji}$ of $f(A_a)$ such that
\[
[\hat R_{ij}(f(A_a))]_{\leq 2} = [\hat R_{ji}(f(A_a))]_{\leq 2}.
\]
So the desired property follows from Corollary~\ref{contraction:same:perm:isp}.
\end{proof}

The following proposition is the culmination of the preceding technicalities; it identifies the
congruence relation $\isp$ with that induced by
$\bb L$.

\begin{prop}\label{kernel:L}
For $\bm U, \bm V \in \RV[{\leq} k]$,
\[
[\bb L \bm U] = [\bb L \bm V] \quad \text{if and only if} \quad ([\bm U], [\bm V]) \in \isp.
\]
\end{prop}
\begin{proof}
The ``if'' direction simply follows from Lemma~\ref{blowup:same:RV:coa} and
Proposition~\ref{L:sur:c}.

For the ``only if'' direction, we show a stronger claim: if $[A] = [B]$ in $\gsk \VF_*$ and $\bm U, \bm V \in \RV[{\leq} k]$ are two standard contractions of $A$, $B$ then $([\bm U], [\bm V]) \in \isp$. We do induction on $k$. The base case $k = 1$ is of course Lemma~\ref{kernel:dim:1:coa}.

For the inductive step, suppose that $F : \bb L \bm U \fun \bb L \bm V$ is a definable bijection. By Lemma~\ref{bijection:partitioned:unary}, there is a partition of $\bb L \bm U$ into definable sets $A_1, \ldots, A_n$  such that each restriction $F_i = F \rest A_i$ is a composition of relatively unary bijections. Applying Corollary~\ref{special:bi:term:constant} as before, we obtain two special bijections
$T$, $R$ on $\bb L \bm U$, $\bb L \bm V$ such that $T(A_i)$, $(R \circ F)(A_i)$ is an $\RV$-pullback for each $i$. By Lemma~\ref{special:to:blowup:coa}, it is enough to show that, for each $i$, there are standard contractions $\hat T_{\sigma}$, $\hat R_{\tau}$ of $T(A_i)$, $(R \circ F)(A_i)$ such that
\[
([(\hat T_{\sigma} \circ T)(A_i)]_{\leq k}, [(\hat R_{\tau} \circ R \circ F)(A_i)]_{\leq k}) \in \isp.
\]
To that end, first note that each $(R \circ F \circ T^{-1}) \rest T(A_i)$ is a composition of relatively unary bijections, say
\[
T(A_i) = B_1 \to^{G_1} B_2 \cdots B_l \to^{G_l} B_{l+1} = (R \circ F)(A_i).
\]
For each $j \leq l - 2$, we can choose five standard contractions
\[
[U_j]_{\leq k}, \quad [U_{j+1}]_{\leq k}, \quad [U'_{j+1}]_{\leq k}, \quad [U''_{j+1}]_{\leq k}, \quad [U_{j+2}]_{\leq k}
\]
of $B_j$, $B_{j+1}$, $B_{j+1}$, $B_{j+1}$, $B_{j+2}$ with the permutations $\sigma_{j}$, $\sigma_{j+1}$, $\sigma'_{j+1}$, $\sigma''_{j+1}$, $\sigma_{j+2}$ of $[k]$, respectively, such that
\begin{itemize}
  \item $\sigma_{j+1}(1)$ and $\sigma_{j+1}(2)$ are the $\VF$-coordinates targeted by $G_{j}$ and $G_{j+1}$, respectively,
  \item $\sigma''_{j+1}(1)$ and $\sigma''_{j+1}(2)$ are the $\VF$-coordinates targeted by $G_{j+1}$ and $G_{j+2}$, respectively,
  \item $\sigma_{j} = \sigma_{j+1}$, $\sigma''_{j+1} =  \sigma_{j+2}$,  and $\sigma'_{j+1}(1) = \sigma''_{j+1}(1)$,
  \item the relation between $\sigma_{j+1}$ and $\sigma'_{j+1}$ is as described in Lemma~\ref{contraction:perm:pair:isp}.
\end{itemize}
By Corollary~\ref{contraction:same:perm:isp} and Lemma~\ref{contraction:perm:pair:isp}, all the adjacent pairs of these standard contractions are $\isp$-congruent, except $([U'_{j+1}]_{\leq k}, [U''_{j+1}]_{\leq k})$.  Since we can choose $[U'_{j+1}]_{\leq k}$, $[U''_{j+1}]_{\leq k}$ so that they start with the same contraction in the first targeted $\VF$-coordinate of $B_{j+1}$, the resulting sets from this step are the same. So, applying  the inductive hypothesis in each fiber over the just contracted coordinate, we see that this last pair is also $\isp$-congruent. This completes the ``only if'' direction.
\end{proof}

This proposition shows that the semiring congruence relation on $\gsk \RV[*]$ induced by $\bb L$ is generated by the pair $([1], \bm 1_{\K} + [(\RV^{\circ \circ}, \id)])$ and hence its corresponding ideal in the graded ring $\ggk \RV[*]$ is generated by the element $\bm 1_{\K} + [\bm P]$ (see Notation~\ref{nota:RV:short} and Remark~\ref{gam:res}).

\begin{thm}\label{main:prop}
For each $k \geq 0$ there is a canonical isomorphism of Grothendieck semigroups
\[
\textstyle \int_{+} : \gsk  \VF[k] \fun \gsk  \RV[{\leq} k] /  \isp
\]
such that
\[
\textstyle \int_{+} [A] = [\bm U]/  \isp \quad \text{if and only if} \quad  [A] = [\bb L\bm U].
\]
Putting these together, we obtain a canonical isomorphism of Grothendieck semirings
\[
\textstyle \int_{+} : \gsk \VF_* \fun \gsk  \RV[*] /  \isp.
\]
\end{thm}
\begin{proof}
This is immediate by Corollary~\ref{L:sur:c} and Proposition~\ref{kernel:L}.
\end{proof}

\begin{thm}\label{thm:ring}
The Grothendieck semiring isomorphism $\int_+$ naturally induces a ring isomorphism:
\[
\textstyle \Xint{\textup{G}}  :  \ggk \VF_* \to \ggk \RV[*] / (\bm 1_{\K} + [\bm P]) \to^{\bb E_{\Gamma}} \Z^{(2)}[X],
\]
and two ring homomorphisms onto $\Z$:
\[
\textstyle \Xint{\textup{R}}^g, \Xint{\textup{R}}^b: \ggk \VF_* \to \ggk \RV[*] / (\bm 1_{\K} + [\bm P]) \two^{\bb E_{\Gamma, g}}_{\bb E_{\Gamma, b}} \Z.
\]
\end{thm}
\begin{proof}
This is just a combination of Theorem~\ref{main:prop} and Remark~\ref{rem:poin} (or Proposition~\ref{prop:eu:retr:k}).
\end{proof}

Let $F$ be a definable set with $A \coloneqq F_{\VF} \sub \VF^n$. Then $F$ may be viewed as a representative of a \emph{definable} function $\bm F : A \fun \gsk \RV[*] / \isp$ given by $a \efun [F_a] / \isp$. Note that the class $[F_a]$ depends on the parameter $a$ and hence can only be guaranteed to lie in the semiring $\gsk \RV[*]$ constructed over $\mdl S \la a \ra$ instead of $\mdl S$, but we abuse the notation. Similarly, for distinct $a, a' \in A$, there is a priori no way to compare $[F_a]$ and $[F_{a'}]$ unless we work over the substructure $\mdl S \la a, a' \ra$; given another definable set $G$ with $A = G_{\VF}$, the corresponding definable function $\bm G$ is the same as $\bm F$ if $\bm G(a) = \bm F(a)$ over $\mdl S \la a \ra$ for all $a \in A$. The set of all such functions is denoted by $\fn_+(A)$, which is a semimodule over $\gsk  \RV[*] / \isp$. Let $E \sub [n]$ be a nonempty set. Then, for each $a \in \pr_{E}(A)$, the definable function in $\fn_+(A_a)$ represented by $F_a$ is denoted by $\bm F_a$.

Let $\bb L F = \bigcup_{a \in A} a \times F_a^\sharp$ and then set $\int_{+A} \bm F = \int_+ [\bb L F]$,
which, by Proposition~\ref{kernel:L} and compactness, does not depend on the representative $F$. Thus there is a canonical homomorphism of semimodules:
\[
\textstyle \int_{+A} : \fn_+(A) \fun \gsk \RV[*] /  \isp.
\]

\begin{thm}\label{semi:fubini}
For all $\bm F \in \fn_+(A)$ and all nonempty sets $E, E' \sub [n]$,
\[
\textstyle \int_{+ a \in \pr_{E}(A)} \int_{+ A_a} \bm F_a = \int_{+ a \in \pr_{E'}(A)} \int_{+ A_a} \bm F_a.
\]
\end{thm}
\begin{proof}
This is clear since both sides equal $\int_{+A} \bm F$.
\end{proof}

\providecommand{\bysame}{\leavevmode\hbox to3em{\hrulefill}\thinspace}
\providecommand{\MR}{\relax\ifhmode\unskip\space\fi MR }
\providecommand{\MRhref}[2]{%
  \href{http://www.ams.org/mathscinet-getitem?mr=#1}{#2}
}
\providecommand{\href}[2]{#2}



\end{document}